\documentclass[11.5pt]{amsart}
\usepackage{amsmath, amsthm, amssymb, url, color, hyperref, graphicx, esint}
\hypersetup{
	colorlinks,
	citecolor=blue,
	filecolor=black,
	linkcolor=black,
	urlcolor=blac
}

\usepackage{todonotes}
\usepackage{comment}
\usepackage[normalem]{ulem}

\usepackage{amsmath, latexsym, amsfonts, amssymb, amsthm}
\usepackage{graphicx, color,hyperref,dsfont,epsfig,caption,wrapfig,subfig}
\usepackage{enumerate}
\usepackage{pstricks}
\usepackage{movie15}
\usepackage{tikz}
\usepackage{bm}
\usetikzlibrary{shapes,decorations.markings}
\hypersetup{
	linktoc=page,
	linkcolor=red,          
	citecolor=blue,        
	filecolor=blue,      
	urlcolor=cyan,
	colorlinks=true           
}

\numberwithin{equation}{section}

\newtheorem{theorem}{Theorem}[section]
\newtheorem{lemma}[theorem]{Lemma}
\newtheorem{proposition}[theorem]{Proposition}
\newtheorem{corollary}[theorem]{Corollary}

\newtheorem{remark}[theorem]{Remark}
\newtheorem{definition}[theorem]{Definition}

\newcounter{conj}
\newtheorem{conjecture}[conj]{Conjecture}

\newcommand{\BS}{\mathrm{BS}}
\newcommand{\cS}{\mathcal{S}}
\newcommand{\cA}{\mathcal{A}}
\newcommand{\cL}{\mathcal{L}}
\newcommand{\rd}{\mathrm{d}}
\newcommand{\cC}{\mathcal{C}}
\newcommand{\cB}{\mathcal{B}}
\newcommand{\cD}{\mathcal{D}}
\newcommand{\cF}{\mathcal{F}}

\newcommand{\SLE}{\mathrm{SLE}}
\newcommand{\CLE}{\mathrm{CLE}}

\newcommand{\QA}{\mathrm{QA}}
\newcommand{\BA}{\mathrm{BA}}
\newcommand{\BD}{\mathrm{BD}}

\newcommand{\LF}{\mathrm{LF}}
\newcommand{\bbA}{\mathbb{A}}
\renewcommand{\P}{\mathbb{P}}
\newcommand{\R}{\mathbb{R}}
\newcommand{\C}{\mathbb{C}}
\newcommand{\E}{\mathbb{E}}
\renewcommand{\H}{\mathbb{H}}
\newcommand{\D}{\mathbb{D}}

\newcommand{\ann}{\mathrm{ann}}
\newcommand{\disk}{\mathrm{disk}}

\newcommand{\cQ}{\mathcal{Q}}

\newcommand{\ball}{\mathrm{ball}}
\newcommand{\eps}{\varepsilon}
\newcommand{\ol}{\overline}
\newcommand{\wt}{\widetilde}

\def\alb#1\ale{\begin{align*}#1\end{align*}}
\newcommand{\eqb}{\begin{equation}}
\newcommand{\eqe}{\end{equation}}

\theoremstyle{definition}

\begin{document}
	
	\title{The moduli of annuli in random conformal geometry}
	\date{\today}

	\author{Morris Ang  \and Guillaume Remy \and Xin Sun}

	\begin{abstract}
		We  obtain   exact  formulae for three basic quantities in random conformal geometry that depend on the modulus of an annulus. The first is for the law of the  modulus of the Brownian annulus describing the scaling limit of uniformly sampled planar maps with annular topology, which is as  predicted from the ghost partition function in bosonic string theory. The second is for the law of the modulus of the annulus bounded by a loop of a simple conformal loop ensemble (CLE) on a disk   and the disk boundary.  The formula is as conjectured from the partition function of the O$(n)$ loop model on the annulus derived by {Saleur-Bauer (1989)} and Cardy (2006).
		The third is for the annulus partition function of the SLE$_{8/3}$ loop introduced by Werner (2008), {confirming another} prediction of Cardy (2006). The physics  principle underlying  our proofs is that  2D quantum gravity coupled with conformal matters can be decomposed into three conformal field theories (CFT): the matter CFT, the Liouville CFT, and the ghost CFT. At the technical level, we rely on  two types of integrability in Liouville quantum gravity, one from the scaling limit of random planar maps, the other from the  Liouville CFT.\\
		
		\noindent R\'ESUM\'E. Dans cet article
		nous obtenons des formules exactes pour trois quantités de base en géométrie conforme aléatoire qui dépendent du module d'un anneau. La première concerne la loi du module de l’anneau brownien décrivant la limite d'échelle des cartes planaires aléatoires uniformes avec la topologie de l'anneau, comme prédit par la fonction de partition de la théorie des cordes bosoniques. La seconde concerne la loi du module de l'anneau délimité par une boucle du ``conformal loop ensemble" (CLE) dans le disque et par le bord du disque. La formule {a \'et\'e conjectur\'ee}  à partir de la fonction de partition du modèle de boucle O$(n)$ sur l'anneau obtenue par Saleur-Bauer (1989) et Cardy (2006). Le troisième formule concerne la fonction de partition sur l'anneau de l'ensemble de boucle SLE$_{8/3}$ introduite par Werner (2008), et confirmant une autre prédiction de Cardy (2006). Le principe physique qui sous-tend nos preuves est que la gravité quantique 2D couplée à un champ de matière conforme peut être décomposée en trois théories conformes des champs (CFT) : la CFT des champs de matière, la CFT de Liouville et la   {CFT fant\^ome}. Les techniques de preuve utilisent deux types d'intégrabilité dans la gravité quantique de Liouville, l'une à partir de la limite d'échelle des cartes planaires aléatoires, et l'autre venant de la CFT de Liouville.
	\end{abstract}
	
	\maketitle

	\section{Introduction} 
	Over the last two decades, tremendous progress has been made in understanding random surfaces of simply connected topology. By~\cite{legall-uniqueness,miermont-brownian-map,bet-mier-disk} and their extensions, 
	the uniformly sampled random planar maps on the sphere and the disk converge as metric-measure spaces in the scaling limit. The limiting surfaces are known as the Brownian sphere and disk, respectively. Moreover, as proved in various senses~\cite{lqg-tbm1,lqg-tbm2,lqg-tbm3,gms-poisson-voronoi,hs-cardy-embedding}, once these surfaces are conformally embedded in the complex plane, the embedded random geometry is given by Liouville quantum gravity (LQG)~\cite{shef-kpz,DDDF-tightness,gm-uniqueness} with $\gamma=\sqrt{8/3}$.  Finally, as shown in~\cite{ahs-sphere,cercle-quantum-disk,AHS-SLE-integrability}, the law of the random field inducing the random  geometry is given by the Liouville conformal field theory on the sphere~\cite{dkrv-lqg-sphere} and the disk~\cite{hrv-disk}. For random planar maps decorated with statistical physics models,  the scaling limit is described by LQG surfaces  decorated  with Schramm Loewner evolution (SLE) and conformal loop ensemble (CLE).
	Although there are still major open questions concerning the scaling limit, the picture in the continuum is well understood, thanks to the theory of the quantum zipper~\cite{shef-zipper} and  the mating-of-trees~\cite{wedges}.
	For more background on this subject, see the surveys~\cite{ghs-mating-survey,Sheffield-ICM22}.
	
	\begin{figure}[ht!]
		\begin{center}

			\includegraphics[scale=0.65]{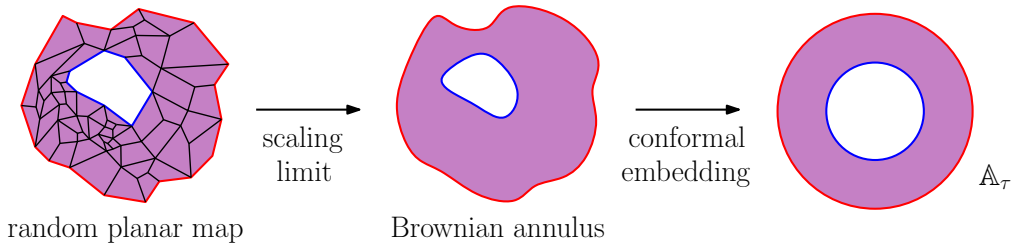}
			\caption{{Conformal embedding of the Brownian annulus to the annulus $\mathbb A_\tau = \{ z \: : \: |z| \in (e^{-2\pi \tau}, 1)\}$ with (random) modulus $\tau$.} 
			} \label{fig-mod-ann}
		\end{center}
	\end{figure}

	For non-simply-connected surfaces, the key new feature is the non-uniqueness of the conformal structure. For example, if we conformally embed the Brownian annulus to the complex plane (Figure~\ref{fig-mod-ann}), the modulus of the planar annulus is random.
	Precise  descriptions of these surfaces, including the laws of the random moduli,  were conjectured in ~\cite{drv-torus} for the torus, in~\cite{grv-higher-genus} for higher genus surfaces, and in~\cite{Remy_annulus} by the second named author for the annulus. The conjecture is based on a fundamental principle that dates back to Polyakov's seminal work~\cite{Polyakov1981} on bosonic string theory. Namely, 2D quantum gravity coupled with conformal matter can be decomposed into three conformal field theories (CFT): the matter CFT, the Liouville CFT, and the ghost CFT. In this paper, we prove this conjecture for annulus in the case of  pure gravity   which corresponds to the Brownian annulus, and the cases where the matter corresponds to the simple CLE~\cite{shef-cle,shef-werner-cle} (i.e. $\CLE_\kappa$ with $\kappa\in (8/3,4]$) or Werner's $\SLE_{8/3}$ loop~\cite{werner-loops}.
	
	We also derive some exact formulae for SLE/CLE conjectured from physics.
	Since the discovery of SLE,
	it has been proved or widely conjectured that SLE/CLE describe the scaling limits of interfaces  in 2D lattice models such as percolation
	and the Ising model.
	Many scaling exponents/dimensions of such models predicted by CFT~\cite{BPZ1984} were derived rigorously from SLE/CLE (e.g.~\cite{LSW-one-arm}). Moreover, some exact formulae predicted via boundary CFT, such as Cardy's formula for percolation crossing, were rigorously derived from $\SLE$ (see e.g.~\cite{werner-notes}). 
	{Saleur-Bauer~\cite{SB89} and Cardy~\cite{cardy06} derived a formula for the partition function of the $O(n)$ loop model on the annulus, which provides crucial information on the CFT description of its scaling limit. Based on the conjectural relation between CLE and the $O(n)$ loop model (see e.g.~\cite{shef-cle}), this {partition function} corresponds to an exact formula for the loop statistics of CLE. Viewing self avoiding loop as the $O(0)$ model, Cardy~\cite{cardy06} further conjectured a formula for the annulus partition function of the $\SLE_{8/3}$ loop. Deriving these formulae for SLE/CLE has been an open question.  In this paper we prove them for simple CLE and the SLE$_{8/3}$ loop.}
	
	The  key to our proofs is an explicit relation  between the partition function of LQG surfaces and the law of their moduli. This is obtained from the conformal bootstrap of Liouville CFT on the annulus due to Wu~\cite{Wu22}. It can be viewed as a Knizhnik-Polyakov-Zamolodchikov (KPZ) relation~\cite{kpz-scaling} between the quantum and the Euclidean geometry  at the partition function level, whereas the traditional KPZ relation established in probability~\cite{shef-kpz}  
	is at the level of scaling exponents/dimensions. Our KPZ relation reduces the computation of an annular quantity depending on the   modulus to the same quantity for LQG surfaces, which is easier to obtain from the  random planar map perspective.
	
	Our work represents the first step towards understanding the relation between LQG in the random geometry framework and  measures on the moduli space of Riemann surfaces, although quantum gravity intuition has inspired important developments of the latter subject since 1980s. 
	Our work is also a further step towards uncovering the CFT content of CLE, after the proof of the imaginary DOZZ formula by the first and third named authors together with Gefei Cai and Baojun Wu~\cite{acsw-cle}.
	It continues to demonstrate the rich interplay between various kinds of integrability in conformal probability, a theme recently explored 
	in~\cite{AHS-SLE-integrability,ARS-FZZ,acsw-cle,acsw-loop}.
	We expect that methods in our paper together with the remarkable developments in the integrability of Liouville CFT~\cite{DOZZ_proof,GKRV_bootstrap} will result in further progress in this direction.
	
	In the rest of the introduction, we first state our result on the modulus of the Brownian annulus in Section~\ref{subsec:intro-BA}. Then we present the matter-Liouville-ghost decomposition for the Brownian annulus and the CLE decorated quantum annulus in Section~\ref{subsec:LQG}. We state our KPZ relation for annulus partition functions in Section~\ref{subsec:intro-KPZ} and
	results for simple CLE and the $\SLE_{8/3}$ loop in Sections~\ref{subsec:intro-CLE}
	and~\ref{subsec:intro-SAP}, respectively. In Section~\ref{subsec:intro-outlook} we discuss some subsequent projects and open questions concerning other models or topologies.

	\medskip
	\noindent {\bf Notation for the modulus.} Let $\mathbb C$ be the complex plane.
	For $\tau\in (0,\infty)$,  let $\bbA_\tau= \{z\in \mathbb C:  |z| \in (e^{-2\pi \tau},1) \}$. For a Riemann surface $A$ of the annular topology, there exists a unique $\tau$ such that $A$ and $\bbA_\tau$ are conformally equivalent. We call $\tau$ the \emph{modulus} of $A$.

	\subsection{Random moduli for the Brownian annulus}\label{subsec:intro-BA}

	There are various equivalent ways of introducing the Brownian annulus. For the convenience of describing its conformal structure, we will introduce it via the Brownian disk. For $a>0$, the \emph{Brownian disk with boundary length $a$} is a random metric-measure space with the disk topology, which is the Gromov-Hausdorff-Prokhorov (GHP) scaling limit of  random quadrangulations or triangulations of a large  polygon  under the critical Boltzmann weight;
	see Proposition~\ref{prop-brownian-disk} for a  precise scaling limit result. 
	If we reweight the law of a Brownian disk by its total area  and add an interior point according to the area measure, then we get a Brownian disk with one interior marked point which we denote by $(\cD,p, d,\mu)$. Here $p\in \cD$ is the marked point, $d$ is the metric and $\mu$ is the area measure. For $0<r<d(p,\partial \cD)$, let ${\cB}(p,r)=\{x\in \cD: d(p, x)\le r \}$ be the metric ball of radius $r$ centered at $p$. Then  each connected component of $\cD\setminus B_r(p,r)$ is locally absolutely continuous to a Brownian disk near the boundary, hence  has a notion of boundary length; see~\cite[Theorem 3]{legall-disk-snake}.
	\begin{definition}\label{def:BA}
		Fix $a>0$ and $b>0$. Let $(\cD,p, d,\mu)$ be a Brownian disk with boundary length $a$ and an interior marked point. On the event  $d(p,\partial \cD)>1$,  
		let $\cA$ be the (annular) connected component of $\cD \backslash {\cB}(p,1)$ whose boundary contains $\partial \cD$. Let  {$\cL$} be the length of the  boundary component of $\partial \cA$ other than $\partial \cD$. Let 
		{the probability measure} $\BA(a,b)^{\#}$ be the conditional law of the metric-measure space $(\cA,d,\mu)$ conditioning on ${\cL}=b$. We call a sample from $\BA(a,b)^{\#}$ a \emph{Brownian annulus of boundary lengths $a$ and $b$}.
	\end{definition}
	We give the full detail for Definition~\ref{def:BA} in Section~\ref{subsec:BA-def}. Although  we define $\BA(a,b)^{\#}$ through regular conditional probability, which is only specified for almost every $b$, it is possible to show that  $\BA(a,b)^{\#}$ is equivalent to various other definitions that make sense for all $b$. {For example, 
		Le Gall and Metz–Donnadieu~\cite{LeGall-Annulus} defines $\BA(a,b)^{\#}$ as by growing the metric ball centered at the origin until the inner boundary length of the complementary annulus reaches $b$; see~\cite[Theorem 3.1]{LeGall-Annulus}. This definition agrees with our Definition~\ref{def:BA} as they showed in~\cite[Proposition 8.1]{LeGall-Annulus}. Furthermore, they proved in~\cite[Theorem 7.1]{LeGall-Annulus}  that $\BA(a,b)^{\#}$ is the $n\to \infty$ GHP scaling limit of the critical Boltzmann weighted triangulation of an annulus with $\lfloor an \rfloor$ and $\lfloor bn \rfloor$ edges on the two boundaries, respectively.} Alternatively, $\BA(a,b)^{\#}$ can be defined directly in the continuum by some conditioned variants of the Brownian snake as in the sphere~\cite{legall-uniqueness,miermont-brownian-map} and disk~\cite{bet-mier-disk} cases; see \cite[Section 7]{Bet-Mier2022}. One can also give a mating-of-trees~\cite{wedges} description of $\BA(a,b)^{\#}$ based on the bijection in~\cite{bhs-site-perc}. Since our main curiosity lies in understanding the law of the modulus, we  will not elaborate on these aspects of the Brownian annulus in this paper but address them elsewhere.

	According to~\cite{lqg-tbm1,lqg-tbm2} which we recall in Section~\ref{subsec:BA-def}, given $(\cD,p,d,\mu)$ in Definition~\ref{def:BA},  we can construct a random metric $d_\H$ and a random measure $\mu_\H$ on the upper half plane $\H$ using $\sqrt{8/3}$-LQG  such that viewed as a (marked) metric-measure space, $(\H, i, d_\H,\mu_\H)$ equals $(\cD,p, d,\mu)$. We call  $(\H, i, d_\H,\mu_\H)$ a conformally embedded Brownian disk with boundary length $a$ and an interior marked point at $i$. On the event $d(p,\partial D)>1$, let $B_{d_\H}(i,1)=\{x: d_\H(i,x)\le 1  \}$ and let $A$ be the connected component of $\H\setminus B_{d_\H}(i,1)$ with $\partial \H$ on its boundary. Then $(A,d_\H,\mu_\H)$ is a conformal embedding of  $(\cA,d,\mu)$. Let $\mathrm{Mod}(\cA)$ be the modulus of the annulus $A$. We call $\mathrm{Mod}(\cA)$ the \emph{modulus} of $(\cA,d,\mu)$. By~\cite[{Theorem 1.4}]{lqg-tbm3}, $(d_\H,\mu_\H)$ is determined by $(\cD,p, d,\mu)$ modulo conformal symmetries. From this we can see that $\mathrm{Mod}(\cA)$ is determined by the metric-measure space $(\cA,d,\mu)$; see Lemma~\ref{lem:Mod-det}. Our theorem below describes the law of $\mathrm{Mod}(\cA)$ under $\BA(a,b)^{\#}$. 
	\begin{theorem}\label{thm:BA-mod}
		For $\tau>0$, let $X_\tau$ be a random variable such that \(\E[X^{i t}_\tau]= \frac{2\pi t e^{ -\frac{2\pi  \tau t^2}{3}}}{   3\sinh(\frac{2}{3} \pi t )} \) for $t\in \R$. Let   $\rho_\tau(\cdot)$ be the probability density function of $X_\tau$, which is supported on $(0,\infty)$. Then the law of  $\mathrm{Mod}(\cA)$ under $\BA(a,b)^{\#}$ is the probability measure on $(0,\infty)$ proportional to $1_{\tau>0}\eta(i2\tau)\rho_\tau(\frac{b}{a})\rd \tau$, where   $\eta(i\tau)=e^{-\frac{\pi\tau}{12}} \prod_{k=1}^\infty (1-e^{-2\pi k\tau})$ is the Dedekind eta function. 
	\end{theorem} 
	Theorem~\ref{thm:BA-mod} yields statements  on the scaling limit of the conformal structure of uniformly random planar maps on the annulus. For example,
	the law of $\mathrm{Mod}(\cA)$ should describe the limiting distribution of the effective resistance of the critical Boltzmann triangulation with the given boundary lengths, viewed as an electric network. Proving this rigorously amounts to establishing  the quenched scaling limit of random walks  on such maps. The analogous statements should hold for other quantities that  encode the modulus of the annulus, such as the annulus crossing probability for the critical percolation or Ising model. The quenched scaling limit  needed to prove such statements is only known for  site percolation on triangulations~\cite{hs-cardy-embedding}.

	Theorem~\ref{thm:BA-mod} is a consequence of the matter-Liouville-ghost decomposition for the Brownian annulus (Theorem~\ref{thm:BAconj}), as we will explain in the next subsection.
	
	\subsection{Matter-Liouville-ghost decomposition for 2D quantum gravity}\label{subsec:LQG}
	The decomposition of 2D quantum gravity coupled with conformal matter was originally due to Polyakov~\cite{Polyakov1981} where the matter is the free boson.
	The ghost part of this decomposition only manifests itself for non-simply connected surfaces. 
	In the probabilistic framework of LQG, this was first explained in~\cite{drv-torus} for the torus case. The annulus case was treated by the second named author~\cite{Remy_annulus} in parallel to the torus case. For higher genus closed surfaces, it was shown in~\cite{grv-higher-genus} that for free boson matter this decomposition indeed produces a convergent theory envisioned in~\cite{Polyakov1981}. Based on the belief that random planar maps decorated with statistical physics models are discretizations  of matter-coupled 2D quantum gravity, precise scaling limit conjectures were formulated in each of the three papers~\cite{drv-torus,grv-higher-genus,Remy_annulus}.  
	In particular, applying the decomposition to the pure gravity on the annulus, a precise LQG description of the Brownian annulus was provided in~\cite[Section 4.3]{Remy_annulus}.
	We now recall this conjecture, which we prove as Theorem~\ref{thm:BAconj} below.
	
	We first introduce the free Brownian annulus with no constraint on the boundary lengths. Let 
	\begin{equation}\label{eq:free-BA}
	\BA=\iint_0^\infty \frac{1}{2\sqrt{ab}(a+b)} \BA(a,b)^{\#} \rd a \, \rd b.
	\end{equation}
	We call a sample from the infinite measure $\BA$ a \emph{free Brownian annulus}. Here the factor $\frac{1}{\sqrt{ab}(a+b)}$ is according to enumeration results for annular planar maps~\cite{bernardi-fusy}; {see Lemma~\ref{lem-boltz-len}}.  From the relation between $\sqrt{8/3}$-LQG and Brownian surfaces established in~\cite{lqg-tbm1,lqg-tbm2,lqg-tbm3}, a free Brownian annulus can be represented as $(\phi,\tau)$ where $\tau\in (0,\infty)$ gives the modulus {of the annulus where the Brownian annulus is conformally embedded}, and $\phi$ describes the variant of Gaussian free field on the annulus that determines the $\sqrt{8/3}$-LQG geometry.

	The first part of the conjecture in~\cite{Remy_annulus} is that the law of $\phi$ given $\tau$ can be described by the Liouville conformal field theory. To make the statement as clean as possible, we work on  the horizontal cylinder $\cC_\tau:= [0,\tau]\times[0,1]/{\sim}$ of unit
	circumference and length $\tau$, where $\sim$ means identifying $[0,\tau]\times \{0\}$ with $[0,\tau]\times \{1\}$; see Figure~\ref{fig-many} (left).  This way, the modulus of $\cC_\tau$ is $\tau$.\footnote{Our modular parameter $\tau$ is denoted by $\ell$ in~\cite{Remy_annulus}.} Let $\P_\tau$ be the law of free boundary Gaussian free field on $\cC_\tau$ defined as in~\cite{shef-gff}; {see Section~\ref{subsec:GFF}}. The \emph{Liouville field measure} $\LF_\tau$ on $\cC_\tau$ is the pushforward of $\P_\tau\times \rd c$  under $(h,c)\mapsto h+c$, where $\rd c$ is the Lebesgue measure on $\R$. The first part of the conjecture can be stated as follows. There exists a measure $m(\rd \tau)$ on $(0,\infty)$ such that $(\phi,\tau)$ sampled from  $\LF_\tau (d\phi)m(\rd\tau)$ gives  a free Brownian annulus conformally embedded onto  a horizontal cylinder. 
	
	The second and more mysterious part of the conjecture is the exact formula for $m(\rd \tau)$. According to matter-Liouville-ghost decomposition of pure gravity, the free Brownian annulus can be written as  
	\begin{equation}\label{eq:mlg-BA}
	\mathcal Z_{\textrm{GFF}}(\tau)\LF_\tau (\rd\phi)\times \mathcal Z_{\textrm{ghost}}(\tau) \rd \tau, \quad \textrm{with }\mathcal Z_{\textrm{GFF}}(\tau)=\frac{1}{\sqrt{2}\eta(2i\tau)}\textrm{ and }Z_{\textrm{ghost}}(\tau)=\eta(2i\tau)^2.
	\end{equation} 
	The measure $\mathcal Z_{\textrm{GFF}}(\tau)\LF_\tau (d\phi)$ is the base measure for the path integral of the Liouville action on $\cC_\tau$, where $\mathcal Z_{\textrm{GFF}}(\tau)$ is the free field partition function coming from this path integral.  The function $\mathcal Z_{\textrm{ghost}}(\tau)$ is the partition function on $\cC_\tau$ of the ghost field for bosonic string theory, which is the non-physical conformal field theory with central charge $c_{\mathrm{ghost}}=-26$   coming from conformal gauge fixing. See \cite[Section 5.2]{Remy_annulus} for more background on $\mathcal Z_{\textrm{GFF}}(\tau)$  and $\mathcal Z_{\textrm{ghost}}(\tau)$.
	Putting the two parts of the conjecture together we get:
	\begin{conjecture}[\cite{Remy_annulus}]\label{conj:BA}
		We have $\BA= 1_{\tau>0}2^{-1/2}\eta(2i\tau) \cdot  \LF_\tau (\rd\phi)\rd\tau$ in the sense that a pair $(\phi,\tau)$ sampled from~\eqref{eq:mlg-BA} gives a free Brownian annulus conformally embedded on $\cC_\tau$. 
	\end{conjecture} 
	Our formulation of Conjecture~\ref{conj:BA} is an identity on two infinite measures while the formulation in~\cite{Remy_annulus} 
	is in terms of convergent functional integrals over these   measures, where the convergence is achieved by introducing cosmological constants. {In this paper we prove the following theorem.}
	\begin{theorem}\label{thm:BAconj}
		Conjecture~\ref{conj:BA} holds.
	\end{theorem}
	
	Once we couple 2D quantum gravity with a conformal matter, the influence on the random geometry can be intuitively understood from the scaling limit of random planar maps. In this context, a conformal matter corresponds to a statistical physics model $M$ with a conformal invariant scaling limit. The 2D quantum gravity coupled with $M$ can be approximated by random planar maps decorated by model $M$, where the law of the planar map is the uniform measure weighted by the partition function of $M$. Now the conjectural description of the conformally embedded geometry for the scaling limit   with annular topology becomes 
	\begin{equation}\label{eq:mlg-QA}
	\mathcal Z_M(\tau)\times \mathcal Z_{\textrm{GFF}}(\tau)\LF_\tau (\rd\phi)\times \mathcal Z_{\textrm{ghost}}(\tau) \rd \tau,
	\end{equation}
	where $\mathcal Z_M(\tau)$ is the partition function of the matter $M$ on $\cC_\tau$. 
	{As explained in \cite{drv-torus,grv-higher-genus,Remy_annulus}, suppose the CFT describing the scaling limit of $M$ has central charge  $c_{\mathrm M}\le 1$.}
	Then the random geometry is induced by the Liouville CFT with central charge $c_{\mathrm L}=-c_{\mathrm{ghost}}-c_{\mathrm M}=26-c_{\mathrm M}$. Here $c_{\mathrm L}$ is related to the background charge $Q$ and the coupling parameter $\gamma$ via $c_{\mathrm L}=1+6Q^2$ and $Q=\frac{\gamma}{2}+\frac{2}{\gamma}$.
	
	For each $\gamma\in (\sqrt{8/3},2)$ (i.e.\ $c_{\mathrm M}\in (0,1)$), 
	the \emph{quantum annulus} is a quantum surface of annular topology which arises naturally in the context of $\gamma$-LQG coupled with $\CLE_{\gamma^2}$, as we recall in Section~\ref{sec:QA}.
	Intuitively, the quantum annulus describes the scaling limit of the random planar maps of annular topology decorated with a particular matter $M$. For example, for $\gamma=\sqrt{3}$,  the matter $M$ is the Ising model on the annulus with $+$ boundary condition on both sides and the extra condition that there is a $+$ arm crossing the annulus. Using the relation between the Ising model and $\CLE_3$~\cite{bh-Ising}, the scaling limit of $M$ is the $\CLE_3$ on the annulus without non-contractible loops. Similarly, for  $\gamma\in (\sqrt{8/3},2)$, the \emph{quantum annulus} $\QA^\gamma$ describes the natural LQG surface of annular  topology decorated by $\CLE_{\gamma^2}$ without non-contractible loops. {(CLE on an annulus without non-contractible loops was studied in~\cite{shef-watson-wu-simple}.)}
	Assuming the relation between $\CLE$ and the scaling limit of O$(n)$-loop model, {the conjectural formula from~\cite{SB89,cardy06}} for the O$(n)$-loop model (see~\eqref{eq:On} and Corollary~\ref{cor:nesting})  yields a formula for the annulus partition function $\mathcal Z_{\mathrm{M}}(\tau)$. Combined with~\eqref{eq:mlg-QA} we arrive at a conjectural description of $\QA^\gamma$, which we prove as the next theorem.
	\begin{theorem}\label{thm:QAconj}
		In the same sense as in Conjecture~\ref{conj:BA}, for $\gamma \in (\sqrt{8/3},2)$, $\QA^\gamma$ is given by
		\begin{equation}\label{eq:QAconj}
		\QA^\gamma=1_{\tau>0}\cdot {\cos(\pi (\frac4{\gamma^2}-1))} \cdot \frac{\gamma}{2\pi } \theta_1(\frac{\gamma^2}{8}, \frac{\gamma^2}{4}i\tau)\LF_\tau (\rd\phi) \rd\tau
		\end{equation}
		where \(\theta_{1}\!\left(z , i\tau \right) := -i {e}^{-\pi \tau / 4} \sum_{n=-\infty}^{\infty} {\left(-1\right)}^{n} {e}^{-n \left(n + 1\right)\pi  \tau} {e}^{(2 n + 1)\pi i z}\) is the Jacobi theta function.
	\end{theorem} 
	Theorems~\ref{thm:BAconj} and~\ref{thm:QAconj} are proved using the same and widely applicable method as we outline below.

	\subsection{A KPZ relation for annulus partition functions} \label{subsec:intro-KPZ}

	Fix $\tau>0$. Let $h$ be the free boundary Gaussian free field (see e.g.~\cite{shef-gff}) on the  cylinder $\cC_\tau$. As an example of Gaussian multiplicative chaos, 
	for  $\gamma\in (0,2)$  the $\gamma$-LQG length measure  on $\partial\cC_\tau$  is defined by $\mathcal L^\gamma_h  = \lim_{\epsilon \to 0} \epsilon^{\frac{\gamma^2}{4}}   e^{\frac{\gamma}{2} h_{\epsilon}(z)} \rd z$,  where $\rd z$ is the Lebesgue measure on $\partial\cC_\tau$, and $h_{\epsilon}(z)$ is the  average of $h$ over $\{ w\in \cC_\tau: |w-z|=\epsilon \}$. Similarly, the $\gamma$-LQG area measure  on $ \cC_\tau$ is $\mathcal A^\gamma_\phi  = \lim_{\epsilon \to 0} \epsilon^{\frac{\gamma^2}{2}}   e^{\gamma  h_{\epsilon}(z)} \rd z$.
	Let $\partial_0 \cC_\tau=\{0\}\times [0,1]/{\sim}$ and $\partial_1 \cC_\tau=\{\tau\}\times [0,1]/{\sim}$. 
	In our previous work~\cite{ARS-FZZ}, based on the conformal bootstrap of boundary Liouville CFT, we gave a conjecture (\cite[Conjecture 1.4]{ARS-FZZ}) on the joint distribution of $\mathcal L^\gamma_\phi(\partial_0 \cC_\tau)$, $\mathcal L^\gamma_\phi(\partial_1 \cC_\tau)$ and $\mathcal A^\gamma_\phi(\cC_\tau)$  under the Liouville field measure $\LF_\tau$. This conjecture was recently proved by Wu~\cite{Wu22} based on  the rigorous bootstrap framework for Liouville CFT developed in~\cite{GKRV_bootstrap}.
	The starting  point of our proof of Theorems~\ref{thm:BAconj} and~\ref{thm:QAconj} is a corollary of Wu's theorem, which is interesting in its own right. It gives a simple description of the law of  $\frac{\mathcal L^\gamma_h(\partial_1\cC_\tau)}{\mathcal L^\gamma_h(\partial_0\cC_\tau)}$ where $h$ is a free boundary GFF. (Although the definition of free boundary GFF involves an additive constant as we recall in Section~\ref{sec:LCFT},  the ratio $\frac{\mathcal L^\gamma_h(\partial_1\cC_\tau)}{\mathcal L^\gamma_h(\partial_0\cC_\tau)}$ is canonically defined.)
	\begin{theorem}\label{thm:GMC-ratio}
		Let $\gamma\in (0,2)$ and $\tau>0$, the characteristic function of $\log\mathcal L^\gamma_h(\partial_1\cC_\tau)- \log\mathcal L^\gamma_h(\partial_0\cC_\tau)$ is 
		\begin{equation}\label{eq:GMC-ratio}
		\E[e^{i x(\log\mathcal L^\gamma_h(\partial_1\cC_\tau)- \log\mathcal L^\gamma_h(\partial_0\cC_\tau)) }]= \frac{\pi\gamma^2 x e^{ -\frac{\pi\gamma^2 \tau x^2}{4}}}{   4\sinh(\frac{\gamma^2}{4} \pi x )} \qquad \textrm{for }x\in \R.
		\end{equation}
	\end{theorem}
	
	Theorem~\ref{thm:GMC-ratio} further demonstrates the elegant integrability of 1D Gaussian multiplicative chaos, which was first explored in~\cite{remy-fb-formula}.
	We recall Wu's theorem and prove Theorem~\ref{thm:GMC-ratio} in Section~\ref{sec:LF-PFN}. From this we derive what we call the \emph{KPZ relation 
		at the annulus partition function level}.
	
	\begin{theorem}\label{thm:KPZ}
		Fix $\gamma\in(0,2)$. Let $m(\rd \tau)$ be a measure on $(0,\infty)$. For a measurable function $f$, write $\langle  f(L_0,L_1)\rangle_{\gamma}:=\int  f(\mathcal L^\gamma_\phi(\partial_0\cC_\tau),\mathcal L^\gamma_\phi(\partial_1\cC_\tau)) \LF_\tau(\rd \phi) m(\rd \tau )$.  Then  $m(\rd \tau)$ satisfies   
		\begin{align}\label{eq:KPZ}
		\int_0^\infty e^{ -\frac{\pi\gamma^2 x^2\tau}{4}}m(\rd \tau )=\frac{2\sinh(\frac{\gamma^2}{4} \pi x )}{\pi\gamma x\Gamma(1+ix)} \langle  L_1 e^{-L_1} L_0^{ix} \rangle_\gamma \quad \textrm{for } x\in \R.
		\end{align} 
	\end{theorem}
	
	Given Theorem~\ref{thm:KPZ}, our proofs of Theorems~\ref{thm:BAconj} and~\ref{thm:QAconj} follow the same two-step strategy:
	\begin{enumerate}
		\item Show that $\BA $ and $\QA^\gamma$ can be written as $\LF_\tau(\rd \phi ) m(\rd \tau)$ for some measure $m(\rd \tau)$. We achieve this by proving a characterization of $\LF_\tau$ in~Section~\ref{sec:LCFT} and verifying this property for  $\BA $ and $\QA^\gamma$.
		\item The integral $\langle  L_1 e^{-L_1} L_0^{ix} \rangle_\gamma$ for $\BA$ and $\QA^\gamma$ can be easily extracted from the existing literature. Then computing $m(\rd\tau)$ from~\eqref{eq:KPZ} is accomplished via an inverse Laplace transform.
	\end{enumerate}
	
	We view Theorem~\ref{thm:KPZ} as a KPZ relation in the following sense.
	For many natural random surfaces such as $\BA$ and $\QA^\gamma$,   $\langle  L_1 e^{-L_1} L_0^{ix} \rangle_\gamma$ can be computed
	from the perspective that LQG describes the scaling limit of random planar maps. From this we can compute $m(\rd \tau)$ by~\eqref{eq:KPZ}. This reduces the computation of a difficult quantity depending on the conformal structure to an easier quantity in  LQG. This is the essence of the derivation of the scaling exponents/dimensions of random fractals using the traditional KPZ relation~\cite{shef-kpz}. 
	For example, for $\QA^\gamma$, the $\tau\to \infty$ asymptotic of $m(\rd \tau)$ given in~\eqref{eq:QAconj} encodes the dimension of the $\CLE$ carpet~\cite{ssw-radii}. Therefore, using~\eqref{eq:KPZ} to determine the  asymptotic of $m(\rd \tau)$ can be viewed as applying the traditional KPZ relation to determine the $\CLE$ carpet dimension. In this sense Theorem~\ref{thm:KPZ} upgrades the KPZ relation from the  exponents/dimensions level to the annular partition function level. 
	
	In principle our KPZ relation~\eqref{eq:KPZ} allows one to obtain the annulus partition $Z_{\mathrm{M}}(\tau)$ of any conformal matter $M$ from its quantum counterpart. In physics $Z_{\mathrm{M}}(\tau)$ is usually obtained as   $\lim_{\delta\to 0}Z^\delta_{\mathrm{M}}(\tau)$ where $Z^\delta_{\mathrm{M}}(\tau)$ is the partition function of $M$ on a lattice of mesh size $\delta$ that approximates $\cC_\tau$. When the scaling limit of $M$ is described by SLE/CLE, the function $Z_{\mathrm{M}}(\tau)$ is supposed to describe a natural SLE/CLE quantity that depends on the modular parameter.  For example, for $\QA^\gamma$, the corresponding $Z_{\mathrm{M}}(\tau)$ describes the law of the modulus of the annulus bounded by an outermost loop  of a $\CLE_{\gamma^2}$ on a disk and the disk boundary; see Theorem~\ref{thm:CLE-mod}. However, in many cases the convergence of $M$ to SLE/CLE is an outstanding problem. Theorem~\ref{thm:KPZ} allows us to bypass this difficulty and obtain $Z_{\mathrm M} (\tau)$ from SLE/CLE using quantum gravity .

	In Section~\ref{subsec:intro-CLE} we summarize our results  for simple CLE  based on the KPZ approach. In Section~\ref{subsec:intro-SAP} we apply the same approach to the case when $M$ is  the self avoiding {loop}.
	\subsection{Random moduli for simple CLE loops}\label{subsec:intro-CLE}
	$\CLE_\kappa$ with $\kappa\in (8/3,8)$ was first defined in~\cite{shef-cle} as a canonical conformally invariant probability measure on infinite collections of non-crossing loops, where each loop looks like an $\SLE_\kappa$ curve.  For $\kappa \in (8/3, 4]$, the loops are simple and $\CLE_\kappa$ can be characterized by a domain Markov property and conformal invariance~\cite{shef-werner-cle}. CLE has been proved or conjectured to be the scaling limit of many important 2D  loop models such as percolation, the Ising model, the $O(n)$-loop model, and the random cluster model; see e.g.~\cite{camia-newman-sle6,bh-Ising}. LQG coupled with $\CLE$
	are proved or conjectured to describe the scaling limit of random planar maps decorated with these models. 
	{Recently, this coupling  was used to derive results for CLE itself; see e.g.~\cite{MSW1,MSW2,acsw-cle, acsw-loop}.}
	As we recall in Section~\ref{sec:QA}, the quantum annulus $\QA^\gamma$ is defined by a symmetric modification of the law of the annulus on a quantum disk bounded by an outermost loop and the disk boundary. In fact choosing any loop in a CLE gives a natural LQG surface of annular topology. Using the KPZ relation~\eqref{eq:KPZ} we  get an analog of Theorem~\ref{thm:QAconj} for these surfaces. This yields exact formulae for the moduli   of the annuli bounded by CLE loops, as  we summarize below.
	
	Let $\D$ be the unit disk. For  a simple loop $\eta\subset \D$, let $A_\eta$ be the annulus bounded by $\eta$ and $\partial\D$ and let $\mathrm{Mod}(\eta)$  be the modulus of $A_\eta$. 
	For $\kappa\in (\frac{8}{3}, 4]$, let $\Gamma$ be a $\CLE_{\kappa}$ on $\D$. Let $\{\eta_j \}_{j\ge 1}$ be the sequence of loops in $\Gamma$ that surrounds the origin, ordered such that $\eta_{j+1}$ is surrounded by $\eta_j$; see Figure~\ref{fig-many} (middle). 
	Our first result is for the moment generating function of $\mathrm{Mod}(\eta_j)$.
	\begin{theorem}\label{thm:CLE-mod}
		For $j\ge 1$ and $\lambda> \frac{ 3\kappa}{32} + \frac{2}{\kappa} -1$, we have 
		\begin{equation}\label{eq:modn}
		\E[e^{-2\pi \lambda \mathrm{Mod}(\eta_j)}]= \frac{ (\frac{4}{\kappa}-1)\cos^j(\pi (\frac{4}{\kappa}-1))}{  \sin( \pi (1-\frac{\kappa}{4}))}\times    \frac{ \sin(\frac{\kappa}{4} \pi \sqrt{   (\frac{4}{\kappa}-1)^2 -\frac{8\lambda}{\kappa}  } )}{ \sqrt{   (\frac{4}{\kappa}-1)^2 -\frac{8\lambda}{\kappa}  }\cos^j(\pi \sqrt{   (\frac{4}{\kappa}-1)^2 -\frac{8\lambda}{\kappa}  })}.
		\end{equation} 
	\end{theorem}
	
	\begin{figure}[ht!]
		\begin{center}				\includegraphics[scale=0.4]{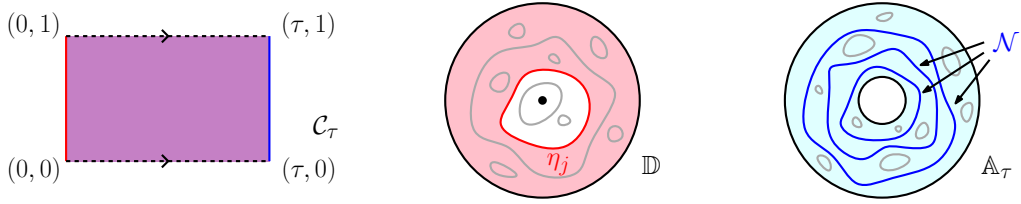}
			\caption{ 
				\textbf{Left:} The cylinder $\cC_\tau$ is obtained by identifying the top and bottom edges of the rectangle $[0,\tau] \times [0,1]$.   \textbf{Middle:}
				With $\eta_j$  the $j^{\text{th}}$ nested loop around the origin, {Theorem~\ref{thm:CLE-mod}} identifies the law of the modulus of the annulus bounded by $\partial \D$ and $\eta_j$. 
				\textbf{Right:} For $\tau>0$,  Corollary~\ref{cor:nesting} gives the law of the number $\mathcal N$  of non-contractible loops for $\CLE$ in $\mathbb A_\tau$. 
			} \label{fig-many}
		\end{center}
	\end{figure}
	
	Let $\mathrm{CR}(\eta_j)$ be the conformal radius of $\eta_j$ viewed from 0. Namely, $\mathrm{CR}(\eta_j)=|\psi'(0)|$ where $\psi$ is a conformal map  from the region surrounded by $\eta_j$ to $\D$ that fixes the origin.  The law of $\mathrm{CR}(\eta_j)$ is obtained  in ~\cite{ssw-radii}:
	\begin{equation}\label{eq:SSW}
	\E[\mathrm{CR}(\eta_j)^\lambda]=\left( \frac{-\cos(4\pi\kappa)}{\cos(\pi \sqrt{(\frac4{\kappa}-1)^2-\frac{8\lambda}{\kappa}} )} \right)^j\quad \textrm{ for }\lambda > \frac{3\kappa}{32} + \frac2\kappa-1.
	\end{equation} Our next theorem gives the joint law of $\mathrm{Mod}(\eta_j)$ and $\mathrm{CR}(\eta_j)$ for $\kappa\in (\frac{8}{3},4]$ and $j\ge 1$.
	\begin{theorem}\label{thm:mod-CR}
		For $j\ge 1$, $\mathrm{Mod}(\eta_j) $  and $e^{2\pi\mathrm{Mod}(\eta_j) }\mathrm{CR}(\eta_j)$ are independent. 
	\end{theorem}
	The case $\kappa=4,j=1$ for Theorems~\ref{thm:CLE-mod} and~\ref{thm:mod-CR}  was obtained in  \cite{ALS20} via the level set coupling of $\CLE_4$ and the Gaussian free field. As shown in \cite{ALS20, ssw-radii}, the law of $\mathrm{Mod}(\eta_1)$ for $\kappa=4$ and $\mathrm{CR}(\eta_1)$ for general $\kappa$ can both be  described in terms of hitting times of Brownian motion. This is also possible for $\mathrm{Mod}(\eta_j)$ for general $\kappa$ and $m$; see Remark~\ref{rmk:BM}. The independence between  $\mathrm{Mod}(\eta_j) $  and $e^{2\pi\mathrm{Mod}(\eta_j) }\mathrm{CR}(\eta_j)$ does not appear to be a simple consequence of the domain Markov property of $\CLE$. Our proof relies on the method 
	in~\cite{AHS-SLE-integrability,ARS-FZZ,acsw-cle,acsw-loop}
	which solves the law of conformal radii of SLE/CLE loops using  LCFT.  
	
	{For a fixed $n>0$, the O($n$) loop model  on a graph is a statistical physics model where a configuration $\mathcal L$ is a collection of disjoint loops (i.e.\ simple cycles). Given the  temperature  $T>0$,  the weight of each configuration is $T^{-\mathcal V(\mathcal L)} n^{\mathcal N(\mathcal L)}$, where $\mathcal V(\mathcal L)$ is the total number of vertices occupied by all loops in $\mathcal L$ and $\mathcal N(\mathcal L)$ is the total number of loops in $\mathcal L$.  When $n\in (0,2]$, for the O($n$) loop model on a planar  lattice, there exists a critical temperature $T_c>0$ such that at $T_c$ the scaling limit is conjectured to be $\CLE_\kappa$ for $\kappa\in (\frac{8}{3},4]$; see~\cite[{Section 2.3}]{shef-cle} for more background.}
	The parameters $\kappa$ and $n$ related as follows:
	\begin{equation}\label{eq:On}
	g=\frac{4}{\kappa}\in [1,\frac32), \quad \chi= (1-g)\pi\in (-\frac\pi2,0], \quad\textrm{and}\quad n=2\cos \chi\in (0,2].
	\end{equation}
	The central charge of {the CFT describing the scaling limit}  is given by: 
	\begin{equation}\label{eq:central}
	c=1-\frac{6(1-g)^2}{g}=1-6(\frac{2}{\sqrt{\kappa}} -\frac{\sqrt{\kappa}}{2})^2 \in (0,1].
	\end{equation}
	
	Consider the $O(n)$ loop model on a lattice approximation of the cylinder $\cC_\tau$. Let  
	\begin{equation}
	q=e^{-\pi /\tau}, \quad  \tilde q=e^{-2\pi \tau}, \quad \textrm{and}\quad  n'=2\cos\chi'\textrm{ for }\chi'\in \R.
	\end{equation}
	Now suppose each contractible loop is assigned weight $n$ while each non-contractible loop is assigned weight $n'$ instead. 
	{Then as predicted in~\cite{SB89,cardy06}, in the continuum limit at $T_c$ the partition function is}
	\begin{align}\label{eq:On}
	Z(\tau, \kappa, \chi')  & =q^{\frac{-c}{24}}\prod_{r=1}^\infty(1- q^r)^{-1}\sum_{p\in \mathbb Z} \frac{\sin(p+1)\chi'}{\sin \chi'}  q^{\frac{gp^2}{4}-\frac{(1-g)p}{2}}\\
	\label{eq:ZOn-tilde}
	&= (\frac{2}{g})^{1/2} \tilde q^{-\frac{c}{12}} \prod_{r=1}^\infty(1-\tilde q^{2r})^{-1}\sum_{m\in \mathbb Z} \frac{\sin((\chi'+2m\pi)/g ) }{\sin \chi'} \tilde q^{\frac{(\chi'+2\pi m)^2}{2\pi^2 g}-\frac{(1-g)^2}{2g}}.  
	\end{align}
	
	Equations~\eqref{eq:On} and~\eqref{eq:ZOn-tilde} were derived in~\cite{SB89} via the quantum group method, and derived  in~\cite{cardy06} via the Coulomb gas method. 
	As pointed out in~\cite{cardy06}, it would be interesting to derive $Z(\tau, \kappa, \chi')$ rigorously in the SLE framework. 
	As a consequence of Theorem~\ref{thm:mod-CR}, we solve this problem. 
	\begin{theorem}\label{thm:CLE-moduli}
		For any non-negative measurable function $f$ on $(0,\infty)$, we have 
		\begin{equation}\label{eq:CLE-main}
		\E\left[ \sum_{j=1}^{\infty} \left(\frac{n'}{n}\right)^{j-1} f(\mathrm{Mod}(\eta_j))\right] =\frac{\sqrt{\kappa }(\frac{4}{\kappa}-1)\cos(\pi (\frac{4}{\kappa}-1))}
		{ \sqrt 2\sin( \pi (1-\frac{\kappa}{4})) }\int_0^\infty f(\tau) e^{\frac{\kappa\pi}{4}
			(\frac{4}{\kappa}-1)^2 \tau}Z(\tau,\kappa,\chi' )\eta(2i\tau)\rd \tau,
		\end{equation}
		where   $\eta(i\tau)=e^{-\frac{\pi\tau}{12}} \prod_{k=1}^\infty (1-e^{-2\pi k\tau})$ is the Dedekind eta function.
	\end{theorem}  
	A more transparent way to see the geometric meaning of $Z(\tau, \kappa, \chi')$ is to consider CLE on the annulus.
	Given $\Gamma$, we choose $\eta$ from the counting measure of $\{\eta_j\}_{j\ge 1}$ and let $A_\eta$ be the annulus bounded by $\eta$ and $\partial \D$. Let $\CLE^\tau_\kappa$ be the conditional law of the loops of $\Gamma$ inside $A_\eta$ conditioning on $\mathrm{Mod(\eta)=\tau}$. Then we use conformal invariance to extend the definition of $\CLE^\tau_\kappa$ to any annulus of modulus $\tau$. This gives one natural definition of $\CLE_\kappa$ on an annulus that allows non-contractible loops; see Figure~\ref{fig-many} (right). In fact, if $\eta=\eta_j$, then the number of non-contractible loops is $j-1$. Therefore Theorem~\ref{thm:CLE-moduli} has  the following corollary. 
	\begin{corollary}\label{cor:nesting}
		For $\tau>0$, let $\mathcal N$ be the number of non-contractible loops of a $\CLE^\tau_\kappa$. Then 
		\begin{equation}\label{eq:nesting}
		\E\left[  \left(\frac{n'}{n}\right)^{\mathcal N} \right] = \frac{Z(\tau,\kappa,\chi' )}{Z(\tau,\kappa,\chi)}
		\end{equation}
		where $\chi=(1-\frac{\kappa}4)\pi $, $n=2\cos \chi$, and  $n'=2\cos \chi'$ for $\chi'\in \R $.
	\end{corollary}

	Annulus partition functions are fundamental in  the CFT framework of 2D statistical physics. The annulus  partition function of a CFT 
	in term of the $q$ variable is of the form $q^{-c/24} \sum_{s \in \cS_{\mathrm {bdy}}} c(s) K_s (q)$ as in~\eqref{eq:On}, where {$c$ is the central charge and} $\cS_{\mathrm {bdy}}$ is the so called boundary spectrum, and $K_s (q)$ is a universal $s$-dependent function called the Virasoro character. This summation is called the conformal bootstrap in the open string channel.  When $\cS_{\mathrm {bdy}}$ is finite, the  CFT is called rational. For example, the Ising model partition function equals $Z(\tau,\kappa,\chi )$ with $n=1$. It can be  written as $Z_{++}+Z_{+-}$, where $Z_{++}$ (resp.,  $Z_{+-}$) is the annulus partition function with $++$ (resp., $+-$) boundary condition. Both $Z_{++}$ and $Z_{+-}$ are examples of Virasoro characters.
	In terms of the $\tilde q$ variable, annulus partition functions are of the form 
	$\tilde q^{-c/12} \sum_{s \in \cS_{\mathrm {bul}}} |c(s)|^2 K_s (\tilde q^2)$. The set $\cS_{\mathrm {bul}}$ is called the bulk spectrum and the summation is called the conformal bootstrap  in the closed string channel.  The equality between~\eqref{eq:On} and~\eqref{eq:ZOn-tilde} is called modular invariance or  open-closed  duality. Conformal bootstrap formulas were also derived for torus partition functions of various models; see e.g.~\cite{dFSZ87}. The spectrum is given by $\cS_{\mathrm{bul} }$ and a similar modular invariance holds. As a cornerstone for 2D CFT,  modular invariance imposes strong constraints on these theories, which is crucial to the classification of {an important class of rational CFTs called} the minimal models. 
	The annulus and torus partition functions can {often} be derived for lattice models by methods such as algebraic CFT principle, Coulomb gas techniques, and transfer matrix, and along the way one obtains {$c$}, $\cS_{\mathrm{bdy}}$ and $\cS_{\mathrm{bul}}$.  See~\cite{Yellow}  for more background on CFT and 2D statistical physics.
	
	The spectra $\cS_{\mathrm{bdy}}$ and $\cS_{\mathrm{bul}}$  determine the allowed boundary and bulk scaling dimensions for the corresponding model. Many scaling dimensions  were later rigorously derived in the SLE framework as the exponents of certain arm events, or as the dimensions of certain fractals. Our paper provides a method to derive the full annulus partition function, strengthening the connection between {the SLE and the CFT}  approaches  to 2D statistical physics. Although the connection between CFT and SLE has been studied intensively from the martingale observable approach (e.g.~\cite{BB-CFT-SLE}), it appears challenging to derive results for CLE using this approach. 
	In the next subsection we demonstrate how our method applies to another model: the self avoiding loop.

	\subsection{Annulus partition function of the SLE$_{8/3}$ loop}\label{subsec:intro-SAP}  
	Werner~\cite{werner-loops} constructed an infinite conformally invariant measure on simple loops on each Riemann surface which we call the \emph{$\SLE_{8/3}$ loop measure}. Modulo a multiplicative constant, the $\SLE_{8/3}$ loop measure is  uniquely characterized by  the conformal restriction property. It describes the conjectural scaling limit of self avoiding polygons~\cite[Conjecture 1]{werner-loops}. A planar map version of this conjecture was 
	proved in~\cite{AHS-loop} based on~\cite{gwynne-miller-saw}.
	The construction of the  $\SLE_{8/3}$ loop measure in~\cite{werner-loops} is based on the Brownian loop measure.
	Let $\mu_\C$ be an $\SLE_{8/3}$ loop measure  on the complex plane $\C$. 
	In Section~\ref{sec:SAP} we will recall a construction of $\mu_\C$  due to Zhan~\cite{zhan-loop-measures}. 
	
	Given a domain $D\subset \C$, let $\mu_D$ be the restriction of $\mu_\C$ to the set of loops that are contained in $D$. Let $\psi$ be a conformal map between two domains $D$ and $D'$. Then  the  conformal restriction property of the $\SLE_{8/3}$ loop measure asserts that $\mu_D$ is the pushforward of $\mu_{D'}$ under $\psi$; see Theorem~\ref{thm:conf-res}.  
	Given an annulus $A\subset \C$ of modulus $\tau$, let  $Z_{8/3}(\tau)$ be the $\mu_\C$-mass of loops in $A$ that are {not contractible}. Then the conformal restriction property implies that  $Z_{8/3}(\tau)$ depends on $A $ through its modulus $\tau$. 
	We call $Z_{8/3}(\tau)$ the annulus partition function of the  SLE$_{8/3}$ loop. It is supposed to be the scaling limit of $\sum  \nu_c^{-\textrm{loop length}}$ where the summation runs over {non-contractible} loops of a lattice approximation of the cylinder $\cC_\tau$, and $\nu_c$ is the so-called  connectivity constant of the lattice. 
	
	The sharp asymptotic of  $Z_{8/3}(\tau)$  was obtained in \cite[Section 7]{werner-loops}, where {deriving} an exact formula for $Z_{8/3}(\tau)$ was left as an open question. A formula  was conjectured by Cardy~\cite[Eq (5)]{cardy06} based on the perspective  that self-avoiding loop can be viewed as $O(n)$ loop model with $n=0$. We rigorously prove Cardy's formula for $Z_{8/3}(\tau)$ as the next theorem.

	\begin{theorem}\label{thm-werner-partition}
		There exists a constant $C>0$ such that $Z_{8/3}(\tau)=C Z_{\mathrm{Cardy}}(\tau) $ where 
		\begin{equation}
		Z_{\mathrm{Cardy}} (\tau)= \prod_{r=1}^\infty(1-  q^r)^{-1} \sum_{k\in \mathbb Z} k (-1)^{k-1}   q^{\frac{3k^2}2-k+\frac18}\quad \textrm{for } \tau>0\textrm{ and }q = e^{-\pi/\tau}.
		\end{equation}
	\end{theorem}
	Since $\mu_\C$ is  characterized up to a multiplicative constant, the constant $C$ in Theorem~\ref{thm-werner-partition} is not canonical.
	We prove Theorem~\ref{thm-werner-partition} in Section~\ref{sec:SAP}. The starting point of our proof is the following fact {proved in~\cite{AHS-loop}}: gluing together two Brownian disks with equal boundary lengths gives a Brownian sphere decorated with an $\SLE_{8/3}$ loop. From this we can construct the $\SLE_{8/3}$-loop-decorated quantum annulus by gluing together two Brownian annuli. Then the two-step strategy outlined below~\eqref{eq:KPZ} gives the law of its modulus. Finally, the conformal restriction property of the $\SLE_{8/3}$ loop  allows us to derive $Z_{8/3}(\tau)$ from the law of the modulus.

	\subsection{Perspectives and outlook}\label{subsec:intro-outlook}We discuss a few directions that a subset of the authors and their collaborators are planning to investigate in subsequent works, and some open questions. 
	\begin{enumerate}
		\item Our KPZ method for proving Theorems~\ref{thm:QAconj} and~\ref{thm-werner-partition} can extend to other models, such as percolation. In a subsequent {work~\cite{Annulus-Per} with Xu and Zhuang, the third named author applied this method to prove  Cardy's crossing formula for percolation on the annulus from~\cite{Cardy2002crossing,cardy06}.} 
		
		\item We can consider variants of~\eqref{eq:KPZ} where we have one boundary marked point on the annulus. Combining with the conformal bootstrap results for Liouville CFT in~\cite{Wu22}, this will give interesting relations between the 1-point torus conformal block~\cite{GRSS-block} and SLE/CLE observables. 
		
		\item We can also consider variants of~\eqref{eq:KPZ} where the underlying domain is a disk with four boundary marked points or with  one bulk and two boundary points. This will give access to SLE observables such as crossing probabilities and the one-point Green's function. 
		
		\item We will consider the variant of~\eqref{eq:KPZ} in the pair of pants case, where we expect to obtain the analogous results  for the Brownian pair of pants and for CLE on a disk with two holes. As a limiting case, we expect to obtain an exact formula for certain two-point correlation functions for CLE on the disk.
		
		\item At the high level, our two-step strategy below~\eqref{eq:KPZ}  can extend to closed surfaces such as the torus. But in the second step,  we would need to find other observables than the boundary length to   obtain the law of the modulus. This could still be tractable for the torus, but we consider it as a major open question to extend our results to arbitrary Riemann surfaces. Essential new ideas are required to understand the much more complicated moduli space and the integrability of Liouville CFT on these surfaces. See~\cite[Section 5]{grv-higher-genus} for precise conjectures in the closed surface case. 
		
		\item Determining the full CFT content of CLE and related discrete models  is an  active topic in both mathematics and physics. See~\cite{Rilbault2024} for a summary of the state of the art  in physics for the $Q$-Potts model. In mathematics, the so-called imaginary DOZZ formula was proved
		in~\cite{acsw-cle}, and the annulus partition function is proved in this paper.  It would be a breakthrough to carry out a functional analytic conformal bootstrap program for CLE similar to the Liouville case~\cite{GKRV_bootstrap}. It would already be very interesting to give a purely CLE-based proof of  results in our Section~\ref{subsec:intro-CLE}, as for the $\kappa=4$ case in~\cite{ALS20}. Such a proof might be extended to prove the torus counterpart  of Corollary~\ref{cor:nesting} as predicted in physics~\cite{dFSZ87}. If knowledge on the conformal matter can be obtained without LQG, then it can be used as a substitute for the boundary length observable in~\eqref{eq:KPZ} to determine the law of the modulus.
		
		\item 
		It would be interesting to understand the coupling of the ghost field with the matter CFT and Liouville CFT at the probabilistic level, {which goes beyond the factorization of partition functions}. 
	\end{enumerate}

	\subsection*{Organization of the paper.} 
	In Section~\ref{sec:LCFT} we give a characterization of $\LF_\tau$. In Section~\ref{sec:LF-PFN} we prove Theorems~\ref{thm:GMC-ratio} and~\ref{thm:KPZ}. In Section~\ref{sec:QA} we carry out the two-step strategy outlined below~\eqref{eq:KPZ} to prove Theorem~\ref{thm:QAconj}. In Section~\ref{sec:CLE} we prove results  for CLE from Section~\ref{subsec:intro-CLE}. In Section~\ref{sec:BA} we prove Theorem~\ref{thm:BAconj} using the same two-step strategy, which implies Theorem~\ref{thm:BA-mod}. In Section~\ref{sec:SAP} we prove Theorem~\ref{thm-werner-partition}.

	\medskip
	\noindent\textbf{Acknowledgements.} 
	We thank Baojun Wu for communicating with us his work~\cite{Wu22}. 
	{We thank Jesper Jacobsen and Hubert Saleur for useful comments on an earlier version of this paper.}
	M.A. was partially supported by NSF grant DMS-1712862 and by the Simons Foundation as a Junior Fellow at the Simons Society of Fellows. 
	G.R.\ was supported by an NSF mathematical sciences postdoctoral research fellowship, NSF Grant DMS-1902804, and a fellowship from the Institute for Advanced Study (IAS) at Princeton. 
	X.S.\ was partially supported by the NSF Career award 2046514, a start-up grant from the University of Pennsylvania, and a fellowship from IAS.

	\section{Liouville fields and their domain Markov property}\label{sec:LCFT}
	In this section we introduce Liouville fields  on the annulus and the disk, and prove their domain Markov property  (Propositions~\ref{prop:Markov-LF} and~\ref{prop:Markov-LFH}) that we will use throughout the paper. We then give a characterization of the Liouville field on the annulus (Proposition~\ref{prop:inv-unique}). 
	In the rest of the paper, we let $\H=\{z\in \C: \Im z>0 \}$ be the upper half plane.  
	For $x<y$, let $\cC(x,y)= [x,y]\times[0,1]/{\sim}$ be a finite horizontal cylinder so that $\cC_\tau=\cC(0,\tau)$. Let $\partial_0 \cC(x,y)=\{x\}\times [0,1]/{\sim}$ and $\partial_1 \cC(x,y)=\{y\}\times [0,1]/{\sim}$. For infinite measures, we will still adopt the probabilistic terminologies such as random variable and  law.

	\subsection{Gaussian free field}\label{subsec:GFF}
	We first gather some facts about two dimensional Gaussian free field (GFF).  See e.g.~\cite{shef-gff} for more background. For concreteness, let $D$ be a planar domain that is conformally equivalent of a disk or an annulus. {Throughout this section we assume that}
	\begin{equation}\label{eq:rho}
	\textrm{$\rho(\rd x)$ is a compactly supported probability measure such that  }\iint -\log |x-y| \, \rho(\rd x) \rho (\rd y)<\infty.
	\end{equation}
	{Our discussion could extend to a broader class of measures but~\eqref{eq:rho} is sufficient for us. For a measure $\rho$ as in~\eqref{eq:rho} supported on $D$,} let $H(D;\rho)$ be the Hilbert closure of $\{f\in C^\infty(D): \int_D f(x) \rho(\rd x)=0 \}$ under the inner product $(f,g)_\nabla= (2\pi)^{-1} \int_D \left(\nabla f \cdot \nabla g \right) \rd x$. Let $(f_n)_{n\ge 1}$ be an orthonormal basis for  $H(D;\rho)$ and $(\alpha_n)_{n\ge 1}$ be a sequence of independent standard normal random variables. Then the random series $\sum_{n=1}^\infty \alpha_n f_n$ converges almost surely as a random generalized function. We  call  a random generalized function on $D$ with the law of $\sum_{n=1}^\infty \alpha_n f_n$ a \emph{free boundary Gaussian free field} on $D$ normalized to have average zero over $\rho(\rd x)$.
	
	Now suppose $A$ is conformally equivalent to an annulus, with the two boundary components denoted by $\partial_1 A$ and $\partial_2 A$.  Let $H(A)$ be the Hilbert closure of $\{f\in C^\infty(A): \mathrm{supp}(f)\cap \partial_1 A=\emptyset\}$ under the inner product $(f,g)_\nabla= (2\pi)^{-1} \int_A \left(\nabla f \cdot \nabla g \right) \rd x$. Here $\mathrm{supp}(f)$ means the support of $f$. Let $(g_n)_{n\ge 1}$ be an orthonormal basis for $H(A)$.
	Let  $(\alpha_n)_{n\ge 1}$ be a sequence of independent standard normal random variables. Then $\sum_{n=1}^\infty \alpha_n g_n$ converges almost surely as a random generalized function. We call a random generalized function on $A$ with the law of $\sum_{n=1}^\infty \alpha_n g_n$ a  Gaussian free field on $A$ with zero-boundary condition on $\partial_1 A$ and free boundary condition on $\partial_2 A$. 
	
	We  will need an instance of the domain  Markov property for the GFF. Let $D$ be a domain conformally equivalent to an annulus or disk and such that $\partial D$ is smooth. Let $\eta$ be a Jordan curve (namely simple closed curve) inside $D$ such that {we can find a connected component $A$ of  $D\setminus \eta$ that is}  conformally equivalent to an annulus. We let $\partial_1 A$ be the boundary component $\eta$ and $\partial_2 A$ be the other boundary component of $A$. Let $\rho(\rd x)$ be a {probability}  measure {on $D\setminus A$ satisfying~\eqref{eq:rho}}. Let  $ \mathrm{Harm}(D;\rho; A)$ be the space of functions in $H(D;\rho)$ that are harmonic  on $A$ and have normal derivative zero on $\partial_2 A$.  As explained in \cite[Theorem 2.17]{shef-gff}, we have the  orthogonal decomposition
	\begin{equation}\label{eq:orthogonal}
	H(D,\rho)=H(A) \oplus \mathrm{Harm}(D;\rho; A),
	\end{equation}
	which yields the following domain Markov property  of the free boundary GFF.
	\begin{lemma}\label{lem:Markov-GFF}
		With $D,A,\rho$ defined {in the paragraph right} above,  let $h$ be a free boundary GFF on $D$ normalized by {$\int h(x)\rho (\rd x)=0$}. Let $h^{\mathrm{har}}$ be the harmonic extension of $h|_{D\setminus A}$ onto $A$ with zero normal derivative on $\partial_2 A$.  Let $h_{A} = h   - h^{\mathrm{har}}$. Then
		$h_{A}$ is a GFF on $A$ with zero boundary condition on $\partial_1 A$ {(namely $\eta$)}, and free boundary condition on $\partial_2 A$. Moreover,  $h_{A}$ is independent of  $h^{\mathrm{har}}$. 
	\end{lemma}
	\begin{proof}
		Since   $h^{\mathrm{har}}$ is the projection of $h$ onto $\mathrm{Harm}(D;\rho; A) $, we obtain Lemma~\ref{lem:Markov-GFF} from~\eqref{eq:orthogonal} via the standard argument from~\cite[{Section 2.6}]{shef-gff}.
	\end{proof}
	\subsection{Liouville fields on the annulus and the disk}
	We first define Liouville fields on annuli. For convenience we focus on the horizontal cylinder $\cC_\tau=\cC(0,\tau)=[0,\tau]\times [0,1]/{\sim}$. 
	\begin{definition}\label{def:LF-tau}
		For $\tau>0$, let $\rho$ be a {probability} measure on $\cC_\tau$ {satisfying~\eqref{eq:rho}} and $\P_{\tau,\rho}$ be the law of the free boundary Gaussian free field on $\cC_\tau$ with $\int_{\cC_\tau} h(x)\rho(\rd x) =0$.  	Sample $(h, \mathbf c)$ from $\P_{\tau,\rho}\times dc$ and set $\phi =  h +\mathbf c$. 
		We write  $\LF_{\tau}$ as the law of $\phi$, and call $\phi$  a  \emph{Liouville field on $\cC_\tau$}.
	\end{definition}
	\begin{lemma}\label{lem:ind-rho}
		The measure $\LF_\tau$ does not depend on the choice of the measure $\rho$. 
	\end{lemma}
	\begin{proof}
		This  follows from the   translation invariance of the Lebesgue measure $dc$. 
	\end{proof}
	Lemma~\ref{lem:Markov-GFF} and~\ref{lem:ind-rho} yield a domain Markov property for $\LF_\tau$. To state it, we recall that $\partial_0 \cC_\tau=\{0\}\times [0,1]/{\sim}$ and $\partial_1 \cC_\tau=\{\tau\}\times [0,1]/{\sim}$ are the two boundary components of $\cC_\tau$.
	\begin{proposition}\label{prop:Markov-LF}
		Let $\eta\subset \cC_\tau$ be a non-contractible Jordan curve so that  both components of $\cC_\tau\setminus \eta$ are of annular topology. Let $A_\eta$  be the component bounded by $\eta$ and $\partial_0 \cC_\tau$.  
		Let $\phi$ be a Liouville field on  $\cC_\tau$. Let $\phi^{\mathrm{har}}$ be the harmonic extension of $\phi|_{\cC_\tau\setminus A_\eta}$ onto $A_\eta$ with zero normal derivative on $\partial_0 \cC_\tau$.  Let $\phi_{A_\eta} = \phi   - \phi^{\mathrm{har}}$. Then conditioning on $\phi|_{\cC_\tau\setminus A_\eta}$, the conditional law of 
		$\phi_{A_\eta}$ is a GFF on $A_\eta$ with zero boundary condition on $\eta$ and free boundary condition on $\partial_0 \cC_\tau$. If  $A_\eta$ is  bounded by $\eta$ and $\partial_1 \cC_\tau$ instead, the same  holds with $\partial_1 \cC_\tau$ in place of $\partial_0 \cC_\tau$.
	\end{proposition}	
	Since $\LF_\tau$ is an infinite measure, the word ``conditioning'' above requires a proper interpretation. 
	\begin{definition}\label{def:Kernel}
		Suppose $(\Omega,\cF)$ and $(\Omega',\cF')$ are two measurable spaces. 
		We say $\Lambda:\Omega \times \cF'    \to [0,1]$ is a Markov kernel if $\Lambda(\omega, \cdot)$ is a probability measure on $(\Omega',\cF')$ for each $\omega\in \Omega$ and  $\Lambda(\cdot,A)$ is $\cF$-measurable for each $A\in \cF'$. If $(X,Y)\in \Omega\times \Omega'$ is a sample from $\Lambda (\omega, \rd \omega') \mu (\rd \omega)$ for some measure $\mu$ on $(\Omega,\cF)$, we say that the conditional law of $Y$ given $X$ is $\Lambda(X, \cdot)$.
	\end{definition}
	\begin{proof}[Proof of Proposition~\ref{prop:Markov-LF}]
		We focus on the case where $A_\eta$ is bounded by $\eta$ and $\partial_0\cC_\tau$ as the other case follows by symmetry. Let $\rho$ be a {probability} measure on the connected component of $\cC_\tau\setminus \eta$ other than $A_\eta$ {satisfying~\eqref{eq:rho}}. Let $\P_{\tau,\rho}$ be the law of the free boundary GFF on $\cC_\tau$ with $\int_{\cC_\tau} h(x)\rho(\rd x) =0$.  Sample $(h, \mathbf c)$ from $\P_{\tau,\rho}\times dc$ and set $\phi =  h +\mathbf c$. Then   $\phi$ is a Liouville field on $\cC_\tau$ by definition. 
		
		Let   $h^{\mathrm{har}}$ be the harmonic extension of $h|_{\cC_\tau\setminus A_\eta}$ onto $A_\eta$ with zero normal derivative on $\partial_0 \cC_\tau$.  Let $h_{A_\eta} =h   - h^{\mathrm{har}}$.  Then  $\phi^{\mathrm{har}}= h^{\mathrm{har}}+\mathbf c$ and $\phi_{A_\eta}=h_{A_\eta}$. By Lemma~\ref{lem:Markov-GFF}, conditioning on $(h|_{\cC_\tau\setminus A_\eta},\mathbf c)$, the conditional law of 
		$h_{A_\eta}$ is a GFF on $A_\eta$ with zero boundary condition on $\eta$ and free boundary condition on $\partial_0 \cC_\tau$. This gives the desired conditional law of  
		$\phi_{A_\eta}$ given  $\phi|_{\cC_\tau\setminus A_\eta}$.
	\end{proof}
	
	We now recall the Liouville field on the disk. We focus on the upper half plane $\H$. 
	\begin{definition}\label{def:LF-H} Fix $\gamma\in (0,2)$ and $Q=\frac{\gamma}{2}+\frac{2}{\gamma}$. 
		Let $\P_{\H}$ be the law of the free boundary Gaussian free field on $\H$ with average zero over the uniform measure on the semi-circle $\{z\in \H: |z|=1\}$.  	Sample $(h, \mathbf c)$ from $\P_{\H}\times e^{-Qc} \rd c$ and set $\phi (z)=  h(z) +\mathbf c- 2Q\log |z|_+$, where $|z|_+ = \max\{|z| , 1\}$.  
		We write  $\LF_{\H}$ as the law of $\phi$, and call $\phi$  a  \emph{Liouville field on $\H$}.
	\end{definition}
	Here we omit the dependence of $\LF_{\H}$ in $\gamma$ because it will be a globally fixed parameter whenever we consider $\LF_{\H}$. 
	We recall Liouville fields on $\H$ with one bulk insertion, following~\cite[{Section 2.2}]{ARS-FZZ}.
	\begin{definition}\label{def:LF-alpha} 
		Fix $\alpha\in \R$ and $z_0\in \H$. 
		Sample $(h, \mathbf c)$ from $(2\mathrm{Im} z_0)^{-\alpha^2/2} |z_0|_+^{2\alpha(Q-\alpha)}\P_{\H}\times e^{(\alpha-Q)c} \rd c$. Set $\phi (z)=  h(z) +\mathbf c-2Q\log |z|_+ + \alpha G_\H (z,z_0)$ where $G_\H(z,w)=-\log (|z-w||z-\bar w|)+2\log|z|_+ + 2\log |w|_+$. We write  $\LF^{(\alpha,z_0)}_{\H}$ as the law of $\phi$, and call $\phi$  a  \emph{Liouville field on $\H$ with an $\alpha$-insertion at $z_0$}.
	\end{definition}
	
	Unlike the case of the annulus, the choice of normalization of the GFF $h$ affects the multiplicative constant of the Liouville field {on the disk}, a phenomenon called the \emph{Weyl anomaly} 
	\cite{dkrv-lqg-sphere, hrv-disk}. The following computation was carried out in \cite[Section 3.3]{hrv-disk}; we rephrase it here and give a quick proof. 
	\begin{proposition}\label{prop-weyl}
		Let $\rho$ be a  {probability} measure on $\H$ {satisfying~\eqref{eq:rho}.} Set
		\[Z_\rho :=  \exp(\frac{Q^2}2 \iint (-\log|z-w|-\log|z-\ol w|)\rho(\rd z)\rho(\rd w)). \]
		Let $P_\rho$ denote the law of the GFF $h$ on $\H$ normalized so $(h, \rho) = 0$.  Sample $(h, \mathbf c)$ from $P_\rho \times [Z_\rho e^{-Qc}\, dc]$. Then the law of $h + \mathbf c - Q G_\rho(\cdot, \infty)$ is $\LF_\H$, where $G_\rho{(\cdot, \cdot)}$ is the {covariance kernel} for $P_\rho$. 
	\end{proposition}
	\begin{proof}
		{In Definition~\ref{def:LF-H}, the covariance kernel for $P_\H$ is $G_\H$  and $G_\H(\cdot, \infty)=2\log|z|_+$. Therefore 
			\begin{align}
			G_\rho(z,w)&=G_\H(z,w) -\int G_\H(z,x) \rho(d x)    -\int G_\H(y,w) \rho(\rd y)+ \iint G_\H(x,y) \rho(\rd x)\rho(\rd y);\nonumber\\
			G_\rho(z,\infty)&=G_\H(z,\infty) -\int G_\H(z,x) \rho(d x)    -\int G_\H(y,\infty) \rho(\rd y)+ \iint G_\H(x,y) \rho(\rd x)\rho(\rd y). \label{eq:Green}
			\end{align}
			Sample $(h, \mathbf c)$ from $P_\H (dh) e^{-Qc}\,dc$. 
			Let $\phi = h + \mathbf c - Q G_\H(\cdot, \infty)$ so that the law of $\phi$ is $\LF_\H$.  Let $h_0 =h-Q \int G_\H(\cdot, x)\rho(\rd x)$ and $\mathbf c' = (h, \rho) - {Q \int G_\H(y,\infty) \rho(\rd y)}  + \mathbf c$. Then 
			$\phi = h - (h, \rho) + \mathbf c' - Q G_\H(\cdot, \infty) + {Q \int G_\H(y,\infty) \rho(\rd y)}$
			hence by~\eqref{eq:Green} we have	$\phi = h_0 - (h_0, \rho) + \mathbf c' - Q G_\rho(\cdot, \infty)$. By Girsanov's theorem
			the law of $(h_0, \mathbf c')$ is  $ Z_\rho P_\H\times e^{-Qc'}\rd c'$, where we used that $Z_\rho$ defined above equals $\E_{P_\H}[e^{Q(h,\rho)}] e^{-Q^2\int G_\H(y,\infty) \rho(\rd y)}$.
			Since the law of $(h_0 - (h_0, \rho),\mathbf c')$ is $P_\rho \times [Z_\rho e^{-Qc'}\, \rd c']$ , we are done.}
	\end{proof}
	We have the following analog of Proposition~\ref{prop:Markov-LF}. 
	\begin{proposition}\label{prop:Markov-LFH}
		Let $\eta\subset \H$ be a Jordan curve. Let $A_\eta$  be the component of $\H\setminus \eta$ bounded by $\eta$ and $\partial \H=\R$.   Let $\phi$ be a Liouville field on  $\H$. Let $\phi^{\mathrm{har}}$ be the harmonic extension of $\phi|_{\H\setminus A_\eta}$ onto $A_\eta$ with zero normal derivative on $\partial \H$ and with $\lim_{|z| \to \infty} \phi^\mathrm{har}(z)/\log|z| = -2Q$.  
		Let $\phi_{A_\eta} = \phi   - \phi^{\mathrm{har}}$. Then conditioning on $\phi|_{\H\setminus A_\eta}$, the conditional law of  $\phi_{A_\eta}$ is a GFF on $A_\eta$ with zero boundary condition on $\eta$ and free boundary condition on $\partial\H$.
	\end{proposition}
	\begin{proof}
		{The same argument for the proof of Proposition~\ref{prop:Markov-LF} works here.}	
		Let $\rho$ be a {probability} measure supported in $\H \backslash A_\eta$ {satisfying~\eqref{eq:rho}}.   Write $\phi=h+\mathbf c -Q G_\rho(\cdot, \infty)$ as in Proposition~\ref{prop-weyl}. 
		Let $h^{\mathrm{har}}$ be the harmonic extension of $h|_{\H\setminus A_\eta}$ onto $A_\eta$ with zero normal derivative on $\partial_0 \H$.  Let $h_{A_\eta} =h   - h^{\mathrm{har}}$. {By~\eqref{eq:Green}, $G_\H (z,\infty)-G_\rho(z,\infty)=o(\log|z|)$ as $z\to \infty$. Since $G_\H (z,\infty)=2\log |z|_+$, we have $\lim_{z\to\infty}\frac{G_\rho(z,\infty)}{\log |z|}=2$. Therefore} $\phi^{\mathrm{har}}= h^{\mathrm{har}}+\mathbf c-Q G_\rho(\cdot, \infty)$ and $\phi_{A_\eta}=h_{A_\eta}$. By Lemma~\ref{lem:Markov-GFF}, conditioning on $(h|_{\H\setminus A_\eta},\mathbf c)$, the conditional law of 
		$h_{A_\eta}$ is a GFF on $A_\eta$ with zero boundary condition on $\eta$ and free boundary condition on $\partial \H$. This gives the desired conditional law of  $\phi_{A_\eta}$ given $\phi|_{\H\setminus A_\eta}$.
	\end{proof}
	
	\begin{lemma}\label{lem:Markov-LFH}
		Fix $\alpha\in \R$ and $z_0\in \H$. 
		Let $\eta\subset \H$ be a Jordan curve surrounding $z_0$. Then the statement of Proposition~\ref{prop:Markov-LFH} holds with  $\LF^{(\alpha,z_0)}_{\H}$ in place of $\LF_{\H}$. 
	\end{lemma}
	\begin{proof}
		By~\cite[Lemma 2.2]{ARS-FZZ}, we have $\LF^{(\alpha,z_0)}_{\H}(\rd \phi )=\lim_{\epsilon\to 0 }
		\epsilon^{\alpha^2/2} e^{\alpha \phi_\epsilon(z)} \LF_\H(\rd \phi)$ where $\phi_\epsilon$ is the average of $\phi$ over $\{ z\in \H: |z-z_0|=\epsilon \}$. Since $\eta$ surrounds  $\{ z\in \H: |z-z_0|=\epsilon \} $ for small enough $\epsilon$, Lemma~\ref{lem:Markov-LFH} follows from Proposition~\ref{prop:Markov-LFH}. 
	\end{proof} 
	
	We need the domain Markov property for $\LF^{(\alpha,z_0)}_{\H}$ to some local sets {in the sense of~\cite{ss-contour}.We restrict ourselves to  two concrete cases needed later; see~\cite[Section 3]{ss-contour} for the general background.}
	\begin{lemma}\label{lem:local}
		Fix $\alpha\in \R$ and $z_0\in \H$.  
		Suppose $(\phi, \eta)$ satisfies one of the following conditions:
		\begin{itemize}
			\item The law of $(\phi,\eta)$  is a product measure $\LF^{(\alpha,z_0)}_{\H} \times \mathbb P$ where $\mathbb P$ is a probability measure on {Jordan curves} in $\H$ {surrounding $z_0$}.
			\item The law of $\phi$ is $\LF^{(\alpha,z_0)}_{\H}$, and $\eta\subset \H$ is a Jordan curve {surrounding $z_0$} such that for each deterministic open set $U\subset \H$, the event $\{D \subset U\}$ is measurable w.r.t. $\phi|_U$, where $D$ is the bounded {simply-}connected component of $\H\backslash \eta$. 
		\end{itemize}
		Then {with $\phi_{A_\eta}$ defined as in Proposition~\ref{prop:Markov-LFH},} conditioning on $(\phi|_{\H\setminus A_\eta},\eta)$, the conditional law of  $\phi_{A_\eta}$ is the GFF  on $A_\eta$ {with zero boundary condition on $\eta$ and free boundary condition on $\partial\H$.}
	\end{lemma}
	\begin{proof}
		Both of the two cases are examples of local sets as defined in~\cite[Section 3]{ss-contour}. Using the standard argument from~\cite[Lemma 3.9]{ss-contour} that extends the Markov property to the strong Markov property, we obtain Lemma~\ref{lem:local} from Lemma~\ref{lem:Markov-LFH}.
	\end{proof}

	\subsection{A resampling characterization of $\LF_\tau$}\label{subsec:resampling}
	Let $H^1(\cC_\tau)$ be the Sobolev space on $\cC_\tau$ with inner product $\int_{\cC_\tau} (fg + \nabla f\cdot \nabla g) \rd x$. Let $H^{-1}(\cC_\tau)$ be the continuous dual of $H^1(\cC_\tau)$ and let $\cF$ be its Borel $\sigma$-algebra. Then both the GFF on $\cC_\tau$  and $\LF_\tau$ can be viewed as measures on $(H^{-1}(\cC_\tau), \cF)$.
	Let $\phi^{\mathrm{har}}$ be the harmonic extension of $\phi|_{\cC(\frac{2\tau}{3},\tau)}$ onto $\cC(0,\frac{2\tau}{3})$ with zero normal derivative on $\partial_0 \cC_\tau$. We define the Markov kernel $\Lambda: H^{-1}(\cC_\tau) \times \cF \to [0,1]$ by letting  $\Lambda(\phi,\cdot)$ be the law of  $h+\phi^{\mathrm{har}}$ where $h$ is a GFF on $\cC(0,\frac{2\tau}{3})$ with zero boundary condition on $\{\frac{2\tau}{3}\} \times [0,1]/{\sim_\gamma}$ and free boundary condition on $\partial_0 \cC_\tau$. By Proposition~\ref{prop:Markov-LF} $\LF_\tau$ is an invariant measure of the Markov kernel $\Lambda$, namely 
	$\LF_\tau = \int \Lambda(\phi,\cdot)\LF_\tau(\rd \phi)$.
	
	Let $\bar \Lambda$ be the Markov kernel defined in the same way as $\Lambda$ in the flipped direction. More precisely, 
	let $\phi^{\mathrm{har}}$ be the harmonic extension of $\phi|_{\cC(0,\frac{\tau}{3})}$ onto $\cC(\frac{\tau}{3},\tau)$ with zero normal derivative on $\partial_1 \cC_\tau$.  Let $\bar \Lambda( \phi,\cdot)$ be the law of  $h+\phi^{\mathrm{har}}$ where $h$ is a GFF on $\cC(\frac{\tau}{3},\tau)$ with zero boundary condition on $\{\frac{\tau}{3}\} \times [0,1]/{\sim_\gamma}$ and free boundary condition on $\partial_1 \cC_\tau$. Then by symmetry  $\LF_\tau$ is an invariant measure of $\bar \Lambda$ as well. We now explain that modulo a multiplicative constant $\LF_\tau$ is the unique measure that is invariant under both $\Lambda$ and $\bar \Lambda$. For both its proof and applications, it will be convenient for us to truncate  $\LF_\tau$  by the quantum length.
	
	We first recall the notion of \emph{quantum length}.   Let $h$ be a free boundary GFF on a domain $D$ which contains a straight line segment $L$. Fix $\gamma\in (0,2)$. Then the $\gamma$-LQG length measure  on $L$  is defined by $\mathcal L^\gamma_h  = \lim_{\epsilon \to 0} \epsilon^{\frac{\gamma^2}{4}}   e^{\frac{\gamma}{2} h_{\epsilon}(z)} dz$,  where $dz$ is the Lebesgue measure on $L$, and $h_{\epsilon}(z)$ is the  average of $h$ over $\{ w\in D: |w-z|=\epsilon \}$. {This is} an example of Gaussian multiplicative chaos (GMC); see e.g.~\cite[Section 3]{berestycki-lqg-notes}, the limit exists in probability in the weak topology of measures on $L$.  Suppose $\phi$ is a Liouville field on $\cC_\tau$, we can define $\mathcal L^\gamma_\phi= \lim_{\epsilon \to 0} \epsilon^{\frac{\gamma^2}{4}}   e^{\frac{\gamma}{2} \phi_{\epsilon}(z)} dz$ as well. Let $I$ be a finite interval in $(0,\infty)$. Let $\Omega_I\subset H^{-1}(\cC_\tau)$ be the event that $\mathcal L^\gamma_\phi (\partial_0\cC_\tau)\subset I$ and let $\ol \Omega_I \subset H^{-1}$ be the event that  $  \mathcal L^\gamma_\phi (\partial_1\cC_\tau) \subset I$. 
	By the following lemma, $\LF_{\tau}(\Omega_I \cap \ol \Omega_I)<\infty$.
	\begin{lemma}\label{lem:GMC}
		Fix $\gamma\in (0,2)$ and $\tau>0$. Suppose we are in the setting of Definition~\ref{def:LF-tau} where $h$ is a free boundary GFF on the horizontal cylinder $\cC_\tau$ with a certain normalization. Then for non-negative measurable functions $f$ and $g$ on $(0,\infty)$,  we have 
		\begin{equation}\label{eq:GMC}
		\LF_\tau[f(\mathcal L^\gamma_\phi (\partial_0\cC_\tau)) g(\mathcal L^\gamma_\phi (\partial_1\cC_\tau))]=  \frac{2}{\gamma}\int_0^\infty \ell^{-1}f(\ell) \E\left[g\left(\ell \frac{\mathcal L^\gamma_h(\partial_1\cC_\tau)}{\mathcal L^\gamma_h(\partial_0\cC_\tau)}\right)\right]\rd\ell.
		\end{equation}
		Note that the ratio $\frac{\mathcal L^\gamma_h(\partial_1\cC_\tau)}{\mathcal L^\gamma_h(\partial_0\cC_\tau)}$ does not depend on normalization of $h$.
	\end{lemma}
	\begin{proof}By definition 
		$\LF_\tau[f(\mathcal L^\gamma_\phi (\partial_0\cC_\tau)) g(\mathcal L^\gamma_\phi (\partial_1\cC_\tau))]=\int_\R \E[f(e^{\frac{\gamma c}{2}}\mathcal L^\gamma_h(\partial_0\cC_\tau))g(e^{\frac{\gamma c}{2}}\mathcal L^\gamma_h(\partial_1\cC_\tau))]dc.$
		Setting $\ell=e^{\frac{\gamma c}{2}}\mathcal L^\gamma_h(\partial_0\cC_\tau)$ and use $dc=\frac{2}{\gamma} \ell^{-1} \rd\ell$ we conclude the proof. 
	\end{proof}
	
	\begin{proposition}\label{prop:inv-unique}
		{Let $I\subset (0,\infty)$ be a compact interval. Let $\LF_{\tau,I}$  be  the probability measure proportional to the restriction of $\LF_{\tau}$ to $\Omega_I \cap \ol \Omega_I$.}
		For $\phi\in H^{-1}(\cC_\tau)$ and $A\in \cF$, let $\Lambda_I(\phi,A )=\frac{\Lambda(\phi,A)}{\Lambda( \phi,  \Omega_I)}$ and  $\bar \Lambda_I(\phi,A)=\frac{\bar \Lambda(\phi,A)}{\bar \Lambda( \phi, \ol \Omega_I )}$. 
		Then $\LF_{\tau,I}$ is the unique  probability measure on $\Omega_I$ that is invariant under both $\Lambda_I $ and $\ol \Lambda_I$, i.e.\ the only probability measure $M$ with $M = \int \Lambda(\phi, \cdot) M(\rd \phi) = \int \ol \Lambda(\phi, \cdot) M(\rd \phi)$.
	\end{proposition}
	\begin{proof}
		We say that a Markov kernel $K: H^{-1}(\cC_\tau) \times \cF \to [0,1]$ is irreducible if there exists a positive measure $\rho$ on $H^{-1}(\cC_\tau)$  such that the following holds. For every $x \in H^{-1}(\cC_\tau)$ and $A \in \cF$ with $\rho(A)>0$ there exists a positive integer $n$ (which may depend on $x$ and $A$) such that  $K^n(x,A) > 0$, {where $K^n$ denotes the Markov kernel corresponding to $n$ steps of the Markov chain}. It is known {(from e.g.~\cite[Propositions 4.2.1 and 10.1.1, and Theorem 10.0.1]{meyn-tweedie})} that if an irreducible  Markov kernel has an invariant probability measure, then it has only one  invariant probability measure. 
		
		By Proposition~\ref{prop:Markov-LF}, we see that  $\LF_{I,\tau}$ is an invariant measure of both $\Lambda_I$ and $\ol \Lambda_I$.  Now we claim that the Markov kernel $K$ corresponding to the composition of $\ol \Lambda_I$ then $\Lambda_I$ 
		is irreducible, which will imply the uniqueness of the invariant measure. It is well known that if $h$ is a GFF in $\cC(\frac\tau3, \tau)$ with zero boundary conditions on $\{\frac\tau3\}\times[0,1]/{\sim}$ and free boundary conditions on $\{\tau\} \times [0,1]/{\sim_\gamma}$ and $g$ is a smooth function in $\cC(\frac\tau3, \tau)$ {with support bounded away from $\{\frac\tau3\} \times [0,1]/{\sim}$ and} with normal derivative zero on $\{\tau\}\times [0,1] /{\sim}$, then the laws of $h|_{\cC(\frac{2\tau}3, \tau)}$ and $(h+g)|_{\cC(\frac{2\tau}3, \tau)}$ are mutually absolutely continuous, see e.g.\ the {proof} of \cite[Proposition 2.9]{ig4}. 
		Therefore, for any $x \in H^{-1}(\cC_\tau)$, the 
		probability  measures $\ol \Lambda(x, \cdot)|_{\cC(\frac{2\tau}3, \tau)}$ and $\LF_{\tau, I}|_{\cC(\frac{2\tau}3, \tau)}$ are mutually absolutely continuous, where we write $M|_{\cC(\frac{2\tau}3, \tau)}$ to denote the law of $\phi|_{\cC(\frac{2\tau}3, \tau)}$ where $\phi \sim M$. Since $\Lambda(\phi, \cdot)$ depends only on $\phi|_{\cC(\frac{2\tau}3, \tau)}$, we conclude that for any $x \in H^{-1}(\cC_\tau)$ the measures $K(x, \cdot)$ and $\LF_{\tau, I}$ are mutually absolutely continuous. We can take $\rho = \LF_{\tau, I}$ to conclude that $K$ is irreducible as desired. 
	\end{proof}

	\section{The annulus boundary partition function for Liouville CFT}\label{sec:LF-PFN}

	In this section we prove Theorem~\ref{thm:GMC-ratio} for the free boundary GFF on $\cC_\tau$, and the KPZ relation in Theorem~\ref{thm:KPZ}. 
	We first recall {the result we need from~\cite{Wu22}} for annulus partition function of Liouville CFT.
	
	{
		\subsection{Input from~\cite{Wu22}  on Liouville CFT on the annulus}\label{subsec:Wu}
		The result we need for the proofs of  Theorems~\ref{thm:GMC-ratio} and~\ref{thm:KPZ} is the following.
		\begin{theorem}\label{thm:bdy-bootstrap}
			Fix $\tau>0$. For $\gamma \in (0,2)$ and $\mu_0, \mu_1>0$, let $L_0=\mathcal L^\gamma_\phi (\partial_0\cC_\tau)$ and $L_1=\mathcal L^\gamma_\phi (\partial_1\cC_\tau)$. Then  there is a positive constant $C=C(\tau)$ such that
			\begin{align}\label{eq:bdy-bootstrap}
			\LF_\tau[L_0L_1 e^{- \mu_0 L_0 - \mu_1 L_1}] = \frac{{C(\tau)}}{2 \pi} \int_{\mathbb{R}}   \partial_{\mu_0}U(Q + i P, \mu_0)  \partial_{\mu_1}U(Q - i P, \mu_1) e^{- \pi \tau P^2}\, \rd P,
			\end{align}
			where $U(\alpha, \mu_i) = \frac{2}{\gamma} \Gamma(\frac{2(\alpha - Q)}{\gamma}) \mu_i^{\frac{2(\alpha - Q)}{\gamma}} \Gamma(\frac{\gamma \alpha}{2} -\frac{\gamma^2}{4}) \Gamma(1 - \frac{\gamma^2}{4})^{\frac{2(Q - \alpha)}{\gamma}}$.
		\end{theorem}
		Theorem~\ref{thm:bdy-bootstrap} is one instance of the various conformal bootstrap formulae for annulus partition function of Liouville CFT obtained in~\cite{Wu22}. To extract~\eqref{eq:bdy-bootstrap} from~\cite{Wu22},  we first note that  \cite[Theorem 1.3]{Wu22} can be formulated in our notation as follows:
		\begin{align}\label{eq:bdy-bootstrap2}
		\LF_\tau[L_0L_1 e^{- \mu_0 L_0 - \mu_1 L_1-\mu A}] = \frac{{C(\tau)}}{2 \pi} \int_{\mathbb{R}}   \partial_{\mu_0}U(Q + i P, \mu_0,\mu)  \partial_{\mu_1}U(Q - i P, \mu_1,\mu) e^{- \pi \tau P^2}\, \rd P.
		\end{align}
		Here $\mu>0$ is the bulk cosmological constant and $A$ is the total quantum area for the Liouville field, and 
		$U(\alpha, \mu_i,\mu)$ is for the one-point bulk structure constant for Liouville CFT on the disk with bulk cosmological constant $\mu$ and boundary cosmological constant $\mu_i$, which was computed in~\cite{ARS-FZZ}.
		We note that~\cite{Wu22} expresses $C(\tau)$ in terms of a determinant of Laplacian operator and a free field partition function; our proof of Theorem~\ref{thm:GMC-ratio} below will in fact yield that $C(\tau)=1$ for all $\tau$.

		As explained in~\cite[Remark 1.4]{Wu22}, when $\mu=0$, \cite[Theorem 1.3]{Wu22} still holds with the same proof. Indeed, the starting point of the proof of the bootstrap formula~\eqref{eq:bdy-bootstrap2} is the spectrum resolution of Liouville Hamiltonian in~\cite{GKRV_bootstrap}. For $\mu>0$, the precise input needed is~\cite[Proposition 6.9]{GKRV_bootstrap}. Based on~\cite[Proposition 6.9]{GKRV_bootstrap}, the conformal ward identity yields that $\LF_\tau[L_0L_1 e^{- \mu_0 L_0 - \mu_1 L_1-\mu A}]$ can be written as an integral involving structure constants (in this case $U(\alpha, \mu_i,\mu)$) and a term only depending on the Virasoro symmetry called the annulus conformal block. The main contribution of~\cite{Wu22} is a fine analysis of the annulus conformal block based on a rigorous realization of Cardy's method on boundary CFT. This analysis gives that although the initial output of the spectrum resolution and conformal Wald identity is a complicated bootstrap expression for  $\LF_\tau[L_0L_1 e^{- \mu_0 L_0 - \mu_1 L_1-\mu A}]$, the   final output is simply~\eqref{eq:bdy-bootstrap2}.

		For the case of Theorem~\ref{thm:bdy-bootstrap} where $\mu=0$, the starting point is still the spectrum resolution. The input~\cite[Proposition 6.9]{GKRV_bootstrap} is replaced by~\cite[Proposition 4.9]{GKRV_bootstrap} (also see \cite[(4.14)]{Wu22}), which concerns the free field Hamiltonian. Now the structure constant is $U(\alpha, \mu_i,0) = U(\alpha, \mu_i)$. The expression of $U(\alpha, \mu_i)$ in Theorem~\ref{thm:bdy-bootstrap}  was derived in~\cite{remy-fb-formula}; the result is stated there as a GMC moment computation, but see \cite[Theorem 1.6]{RZ_boundary} for the formulation in terms of $U(\alpha, \mu_i)$. The analysis in~\cite{Wu22} for the annulus conformal block based on the conformal Wald identity and Cardy's method still applies here. As a consequence, 
		\eqref{eq:bdy-bootstrap2} still holds with $\mu=0$ hence our Theorem~\ref{thm:bdy-bootstrap} holds. 
	}

	\subsection{Proof of Theorems \ref{thm:GMC-ratio} and~\ref{thm:KPZ}}\label{subsec:2proof}
	
	\begin{proof}[Proof of Theorem~\ref{thm:GMC-ratio}]
		Starting from \eqref{thm:bdy-bootstrap}, by using the above formula for $U(\alpha, \mu_i)$, the identity $\Gamma(x) \Gamma(1-x) = \frac{\pi}{\sin(\pi x)}$, and setting $\mu_1=1$  we have:
		\begin{align*}
		\LF_\tau[L_1 L_0 e^{- L_1 - \mu_0 L_0}] &=  -\frac{2 \pi {C(\tau)}}{\mu_0 \gamma^2 }  \int_\R   \frac{\mu_0^{- \frac{2 i P}{\gamma}}}{\sin(\frac{2\pi i P}{\gamma})  \sin(\frac{\gamma\pi i P}{2})} e^{- \pi \tau P^2} P^2 \, \rd P \\
		&=  -\frac{\gamma \pi {C(\tau)}}{4\mu_0 }  \int_\R   \frac{\mu_0^{iu}}{\sin(u \pi i)  \sin(\frac{\gamma^2}4 u \pi i)} e^{- \frac{\pi\gamma^2}4 \tau u^2} u^2 \, \rd u.
		\end{align*}
		{with $C(\tau)$ as in Theorem~\ref{thm:bdy-bootstrap}.} By Lemma \ref{lem:GMC}, for all real $x$:
		\begin{align}\label{eq:fraction1}
		\LF_\tau[L_1e^{-L_1} L_0^{ix}]=  \frac{2}{\gamma}\int_0^\infty \ell^{ix}e^{-\ell}  \rd\ell
		\E\left[\left( \frac{\mathcal L^\gamma_h(\partial_1\cC_\tau)}{\mathcal L^\gamma_h(\partial_0\cC_\tau)}\right)^{ix}\right]= \frac{2}{\gamma} \Gamma(1+ix) \E\left[\left( \frac{\mathcal L^\gamma_h(\partial_1\cC_\tau)}{\mathcal L^\gamma_h(\partial_0\cC_\tau)}\right)^{ix}\right].
		\end{align}
		Using furthermore $\Gamma(1 - ix) L_0^{ix -1} = \int_{0}^{\infty}  \lambda^{-ix } e^{- \lambda L_0} \, \rd\lambda = \int_{\mathbb{R}}  e^{(1 -ix) \omega} e^{- e^{\omega} L_0}\, \rd \omega$, we obtain:
		\begin{align*}
		\E\left[\left( \frac{\mathcal L^\gamma_h(\partial_1\cC_\tau)}{\mathcal L^\gamma_h(\partial_0\cC_\tau)}\right)^{ix}\right] &= \frac{\gamma}{2 \Gamma(1 + ix) \Gamma( 1 -ix) } \int_{- \infty}^{\infty}  e^{(1-ix) \omega} \LF_\tau[L_0 L_1e^{-L_1} e^{- e^{\omega} L_0}] \, \rd \omega \\
		&= -\frac{\gamma^2  \sin (\pi ix) {C(\tau)}}{8 ix} \iint_{-\infty}^\infty \frac{e^{i  \omega u}}{\sin(u \pi i) \sin(\frac{\gamma^2}4 u \pi i)} e^{- \frac{\pi\gamma^2}4 \tau u^2} u^2 \, \rd u \, e^{-ix\omega}\, \rd \omega.
		\end{align*}
		Using the Fourier inversion formula $f(x) = \frac{1}{2\pi} \int_{\mathbb{R}}  e^{- i x \omega} \int_{\mathbb{R}} e^{i \omega u} f(u) \,\rd u \,\rd \omega$, we obtain:
		\begin{align}\label{eq:sec3.1}
		\E\left[\left( \frac{\mathcal L^\gamma_h(\partial_1\cC_\tau)}{\mathcal L^\gamma_h(\partial_0\cC_\tau)}\right)^{ix}\right] 
		={-\frac{\gamma^2  \sin (\pi ix) C(\tau)}{8 ix}\times \frac{2\pi e^{- \frac{\pi\gamma^2}4 \tau x^2} x^2}{\sin(x \pi i) \sin(\frac{\gamma^2}4 x \pi i)}
			=} {C(\tau)}\frac{\pi \gamma^2 x e^{-  \frac{\pi \tau \gamma^2 x^2}{4} }}{4 \sinh(\frac{\gamma^2 \pi x}{4})}.
		\end{align}
		Sending $x\to0$ we get   {$C(\tau)=1$ for all $\tau>0$ as desired.}
	\end{proof}

	\begin{remark}\label{rmk:circle}We can decompose $h$ into $h = R + Y$ where $R$ is the radial part of the field given by $R(x) = \int_0^1 h(x + i y) dy$, and $Y = h - R$ has zero average along the $y$-axis for every $x$. It is well known that $Y$ and $R$ are independent. From here it is straightforward to decompose our functional of interest as
		\begin{align*}
		\E[ \left(\mathcal L^\gamma_h(\partial_1\cC_\tau) \right)^{ix} (\mathcal L^\gamma_h(\partial_0\cC_\tau))^{-ix} ] &= \E[ e^{\frac{i \gamma t}{2} (R(\tau) -R(0))} ]  \E[ \left(\mathcal L^\gamma_Y(\partial_1\cC_\tau) \right)^{ix} (\mathcal L^\gamma_Y(\partial_0\cC_\tau))^{-ix} ],
		\end{align*}
		where here $\mathcal L^\gamma_Y(\partial_0\cC_\tau)$ has the same definition as $\mathcal L^\gamma_h(\partial_0\cC_\tau)$ with $h$ replaced by $Y$, and similarly for $\mathcal L^\gamma_Y(\partial_1\cC_\tau)$. By direct computation on the covariance of $R$, one can obtain that $\E[ e^{\frac{i \gamma x}{2} (R(\tau) -R(0))} ] = e^{- \frac{\gamma^2 x^2 \pi \tau}{4}} $. By Theorem~\ref{thm:GMC-ratio}, we have $\E[ \left(\mathcal L^\gamma_Y(\partial_1\cC_\tau) \right)^{ix} (\mathcal L^\gamma_Y(\partial_0\cC_\tau))^{-ix} ]=\frac{\pi\gamma^2 x }{   4\sinh(\frac{\gamma^2}{4} \pi x )}$. Therefore, the law of $\mathcal L^\gamma_Y(\partial_1\cC_\tau)/\mathcal L^\gamma_Y(\partial_0\cC_\tau)$ is a log-logistic distribution.
		We observe that interestingly the law of this ratio does not depend on  $\tau$. Moreover, as $\tau\to\infty$,  $\mathcal L^\gamma_Y(\partial_1\cC_\tau)$ and $\mathcal L^\gamma_Y(\partial_0\cC_\tau)$ become independent and both of them converge in law to the total mass of the Gaussian multiplicative chaos on  the circle explicitly solved in~\cite{remy-fb-formula}.  {Justifying the independence in $\tau$ directly would provide an alternative proof of} Theorem~\ref{thm:GMC-ratio} without relying on Theorem~\ref{thm:bdy-bootstrap}, but we do not pursue it in this paper.
	\end{remark}

	\begin{proof}[Proof of Theorem~\ref{thm:KPZ}]
		By~\eqref{eq:fraction1} and Theorem~\ref{thm:GMC-ratio}, we have 
		\begin{equation} \label{eq:fraction}
		\LF_\tau[L_1e^{-L_1} L_0^{ix}]=  \frac{2}{\gamma}\int_0^\infty \ell^{ix}e^{-\ell}  \rd\ell
		\E\left[\left( \frac{\mathcal L^\gamma_h(\partial_1\cC_\tau)}{\mathcal L^\gamma_h(\partial_0\cC_\tau)}\right)^{ix}\right]=  \frac{\pi\gamma x e^{ - \frac{\pi\gamma^2 \tau x^2}{4}}}{   2\sinh(\frac{\gamma^2}{4} \pi x  )}\Gamma(1+ix).
		\end{equation}
		Since $\langle L_1e^{-L_1} L_0^{ix}\rangle_\gamma=\int_0^\infty\LF_\tau[L_1e^{-L_1} L_0^{ix}] m(\rd \tau)$ by definition, we get~\eqref{eq:KPZ}.
	\end{proof}

	We note for later use that by following the computation~\eqref{eq:fraction} with the change of variables $y = ix$ gives
	\begin{equation} \label{eq:fraction-real}
	\LF_\tau[L_1e^{-L_1} L_0^y]=   \frac{\pi\gamma y e^{  \frac{\pi\gamma^2 \tau y^2}{4}}}{   2\sin(\frac{\gamma^2}{4} \pi y  )}\Gamma(1+y) \quad \text{ for }y \in (-1,\frac4{\gamma^2}).
	\end{equation}
	The constraints on $y$ arise as follows: we need $y \geq -1$ for the gamma function integral in~\eqref{eq:fraction1}  to converge, and we need $|y| < \frac4{\gamma^2}$ for $\E\left[\left( \frac{\mathcal L^\gamma_h(\partial_1\cC_\tau)}{\mathcal L^\gamma_h(\partial_0\cC_\tau)}\right)^{y}\right]$ to be finite; indeed  Theorem~\ref{thm:GMC-ratio} implies $-\log \cL_h^\gamma(\partial_1 \cC_\tau) + \log \cL_h^\gamma(\partial_0 \cC_\tau)$ has the law of a logistic random variable with scale $s = \frac{\gamma^2}4$ plus an independent Gaussian, and for a logistic random variable $X$ with scale $s$ we have $\E[e^{tX}] < \infty$ when $|t|< \frac1s$.

	\newcommand{\Md}{\mathcal{M}^{\mathrm{disk}}_1}
	\section{Quantum annulus: Liouville field description and the random modulus}\label{sec:QA}
	
	The quantum annulus is a quantum surface  introduced in \cite{acsw-cle,acsw-loop} in the context of the quantum disk coupled with CLE. {We denote its law by $\QA^\gamma$. In Section~\ref{sub:QAbackground} we recall the definition and some fundamental properties of $\QA^\gamma$.}

	The main goal of this section is to prove Theorem~\ref{thm:QAconj} which 
	describes the matter-Liouville-ghost decomposition of $\QA^\gamma$. Namely,  when $(\tau, \phi)$ is sampled from 
	\begin{equation}
	1_{\tau>0}\cdot {\cos(\pi (\frac4{\gamma^2}-1))} \cdot \frac{\gamma}{2 \pi } \theta_1(\frac{\gamma^2}{8}, \frac{\gamma^2}{4}i\tau)\LF_\tau (\rd\phi) \rd\tau,    
	\end{equation}
	the law of the quantum surface $(\cC_\tau,\phi)/{\sim_\gamma}$ is $\QA^\gamma$. 
	In Section~\ref{subsec:QA-LF} we prove that  $\QA^\gamma$ can be written in the form of $ \LF_\tau (d\phi) m(\rd\tau)$ for some $m(\rd \tau)$. This is based on the domain Markov property of Liouville fields from Section~\ref{sec:LCFT}. Then in Section~\ref{subsec:QA-Z} we identify $m(\rd \tau)$ as given in Theorem~\ref{thm:QAconj} {using the KPZ relation Theorem~\ref{thm:KPZ}. In Section~\ref{sec:CLE}, we will prove the CLE results stated in Section~\ref{subsec:intro-CLE} using Theorems~\ref{thm:QAconj} and~\ref{thm:QA-welding}.}
	
	\subsection{{Background on quantum annulus}}\label{sub:QAbackground}
	
	We first recall the definition of quantum surfaces.  
	Fix $\gamma\in (0,2)$ and $Q=\frac{\gamma}{2}+\frac{2}{\gamma}$. 
	Consider pairs of the form  $(D,h)$   where $D$ is a planar domain and $h$ is a
	generalized function on $D$. We say that  $(D,h) \sim_\gamma(\tilde D,\tilde h)$  if there exists a conformal map $\psi: \tilde D\to D$ such that  $\tilde h=h \circ \psi +Q\log |\psi'|$.  A quantum surface in $\gamma$-LQG  is an equivalence class under $\sim_\gamma$.  We write $(D, h)/{\sim_\gamma}$ as the quantum surface corresponding to $(D, h)$. An embedding of a quantum
	surface is a choice of its representative. 
	We extend $\sim_\gamma$ to quantum surfaces with decorations/markings. For example, for triples $(D,h,z)$ where $z\in D$, we  further require  $z=\psi (\tilde  z)$ when defining $(D,h,z) \sim_\gamma(\tilde D,\tilde h,\tilde z)$ and call $(D,h,z)/{\sim_\gamma}$ a quantum surface with an interior marked point. We can similarly define quantum surfaces decorated by curves. 
	
	The notion of quantum length is intrinsic to quantum surfaces.
	Suppose $(D,h) \sim_\gamma(\tilde D,\tilde h)$ through the conformal map $\psi: \tilde D\to D$. {Recall the GMC definition of quantum length above Lemma~\ref{lem:GMC}.}
	If $D$ and $\tilde D$ each contain line segments which are related by $\psi$, then $\mathcal L^\gamma_{\tilde h}$ is  the pushforward of $\mathcal L^\gamma_h$ under $\psi$.  Therefore, we can use this coordinate change to consistently define $\mathcal L^\gamma_h$ on $\partial D$ even if it does not contain a line segment. We now recall the Liouville field description of the quantum disk with an $\alpha$ bulk insertion {defined in~\cite{ARS-FZZ}.} 
	\begin{definition}\label{def:QD}
		For $\gamma\in (0,2)$ and $\alpha>\frac{\gamma}{2}$,  let $\Md(\alpha)$ be the law of the   quantum surface $(\H,\phi, i)/{\sim_\gamma}$ where $\phi$ is sampled from $\LF^{(\alpha,i)}_\H$. Let $\LF^{(\alpha,i)}_\H(\ell)$ be the disintegration  $\LF^{(\alpha,i)}_\H=\int_0^\infty \LF^{(\alpha,i)}_\H (\ell) \rd\ell$ such that samples from $\LF^{(\alpha,i)}_\H(\ell)$ have boundary length $\ell$. Let $\Md(\alpha;\ell)$ be   the law of $(\H,\phi, i)/{\sim_\gamma}$ where $\phi$ is sampled from $\LF^{(\alpha,i)}_\H(\ell)$.
	\end{definition} 
	In this section we only use $\Md(\gamma;\ell)$ but we will use $\Md(\alpha;\ell)$ in the next one. 
	By \cite [Theorem 3.4]{ARS-FZZ}, $\frac{\gamma}{2\pi(Q-\gamma)^2}\Md(\gamma;\ell)$ is the law of a quantum disk with boundary length $\ell$ and one interior marked point as defined in \cite[{Section 4.5}]{wedges}. 
	The reason we require $\alpha>\frac{\gamma}{2}$ is that the total mass 
	$|\LF^{(\alpha,i)}_\H (\ell)|$ of $\LF^{(\alpha,i)}_\H (\ell)$ (hence also  $\Md(\alpha;\ell)$) is finite in this case. The precise value of $|\LF^{(\alpha,i)}_\H (\ell)|$ was computed in~\cite{remy-fb-formula} {as explained in~\cite[Lemma 2.7 and Proposition 2.8]{ARS-FZZ}:}
	\begin{equation}
	\label{eq:LFl-explicit}
	|\LF^{(\alpha,i)}_\H (\ell)| = \frac{2}{\gamma} 2^{\frac{\alpha^2}{2}-\alpha Q} \cdot \left( \frac{2\pi}{\Gamma(1-\frac{\gamma^2}4)} \right)^{\frac2\gamma(Q-\alpha)} 
	\Gamma( \frac{\gamma\alpha}2-\frac{\gamma^2}4)\times \ell^{\frac{2}{\gamma}{(\alpha-Q)}-1} \quad \textrm{for } \alpha>\frac{\gamma}{2} .
	\end{equation}
	
	{Next, we give the definition of the quantum annulus. This is a quantum surface first defined in the unpublished article ``Integrability of the conformal loop ensemble'' of the first and third authors, the main accomplishment of which  was a proof that for $8/3 < \kappa \leq 4$ the nesting loop statistics of CLE$_\kappa$ agree with the imaginary DOZZ formula.
		This unpublished article has been superseded by the two papers \cite{acsw-cle, acsw-loop} written jointly with Gefei Cai and Baojun Wu, which treat the full parameter range $\kappa \in (8/3,8)$, establish the Delfino-Viti conjecture on the three-point function of critical 2D percolation, and derive a number of other CLE observables. Our subsequent exposition will reference \cite{acsw-cle, acsw-loop}.
	}
	
	Let $\kappa\in (\frac83,4)$ and $\gamma=\sqrt{\kappa}$.  Let $\Gamma$ be a  $\CLE_\kappa$ on $\H$ and $h$ be a free boundary GFF on $\H$ independent of $\Gamma$. Let $\eta$ be a loop in $\Gamma$. Let $A_\eta$ and $D_\eta$ be the two connected components of $\H\setminus \eta$ where $D_\eta$ is simply connected and $A_\eta$ is annular. By the theory of quantum zipper~\cite{shef-zipper},  the quantum length measures of the quantum surfaces $(A_\eta,h)/{\sim_\gamma}$ and $(D_\eta, h)/{\sim_\gamma}$ agree on $\eta$. This defines a quantum length measure on $\eta$. If $h$ is replaced by a Liouville field, the quantum length measure on $\eta$ can also be defined. We recall the definition of the quantum annulus from \cite[Definition 5.2]{acsw-loop}.
	{
		\begin{definition}\label{def:QA} 
			For $a>0$, let $(\phi,\Gamma)$  be a sample from $\LF_\H^{\gamma,i}(a)\times \CLE_\kappa^\H$ where $\CLE^\H_\kappa$ is the law of a  $\CLE_\kappa$ on $\H$. Let $\eta$ be the outermost loop of $\Gamma$ surrounding $i$. Let $\wt \QA(a)$ denote the law of the quantum surface $(A_\eta, \phi)/{\sim_\gamma}$, and let $\wt \QA(a,b)$ be the disintegration of $\wt \QA(a)$ with respect to the quantum length of the inner boundary component, so $\wt \QA(a) = \int_0^\infty \wt \QA(a,b)\, \rd b$, and each measure $\wt \QA(a,b)$ is supported on the space of annular quantum surfaces with quantum boundary lengths $a$ and $b$. Let 
			\[
			\QA^\gamma(a,b)=\frac{1}{b |\mathcal M_1^\mathrm{disk}(\gamma;b)|}\wt\QA^\gamma(a,b).
			\]
			Let $\QA^\gamma=\int_0^\infty\int_0^\infty\QA^\gamma(a,b) \, \rd a \rd b$. 
			We call a sample of $\QA^\gamma$ a  \emph{quantum annulus}.
		\end{definition}
	}
	The measure $\QA^\gamma(a,b)$ is finite for all $a,b>0$ but $\QA^\gamma$ is infinite; see Proposition~\ref{prop-qa-mass} below.
	As explained in \cite[Remark 5.3]{acsw-loop}, the above definition in terms of disintegration is only well-specified for almost every 
	$b\in (0,\infty)$. This ambiguity does not affect the definition of $\QA^\gamma$ and the proof of  Theorem~\ref{thm:QAconj} below. After we prove Theorem~\ref{thm:QAconj}, it can be used to canonically define $\QA^\gamma(a,b)$ for every $(a,b)$. Therefore we omit this temporary ambiguity. 
	
	{
		The reason we work with $\QA^\gamma(a,b)$ rather than $\wt \QA^\gamma(a,b)$ is because $\QA^\gamma(a,b)$ is natural in the context of conformal welding, as demonstrated in Theorem~\ref{thm:QA-welding} below. Theorem~\ref{thm:QA-welding} was implicitly proved in \cite{MSW1}, and a detailed explanation is given in \cite[Section 4.4]{acsw-loop}.}
	\begin{theorem}[{\cite[Theorem 4.2]{acsw-loop}}]\label{thm:QA-welding}
		In the setting of Definition~\ref{def:QA}, the joint law of $(A_\eta, \phi)/{\sim_\gamma}$ and $(D_\eta, \phi)/{\sim_\gamma}$ is
		\begin{equation}\label{eq:QA-weld}
		\int_0^\infty b \QA^\gamma(a,b)\times \Md(\gamma;b)\rd b.
		\end{equation}
	\end{theorem}
	The intuitive meaning of the factor $b$ in~\eqref{eq:QA-weld} is that the number of ways of gluing together $\QA^\gamma(a,b)$ and $\Md(\gamma;b)$ is proportional to the boundary length $b$ of their interface.

	{
		Thus far, we have introduced $\QA^\gamma$ and explained the motivation for its definition. We now discuss some fundamental properties, starting with the formula for the partition function. 
		\begin{proposition}[{\cite[Proposition 5.5]{acsw-loop}}] \label{prop-qa-mass}
			The total mass of $\QA^\gamma(a,b)$ is 
			\begin{equation}\label{eq-qa-mass}
			|\QA^\gamma(a,b)|=\frac{\cos(\pi (\frac4{\gamma^2}-1))}{\pi \sqrt{ab} (a+b)}.
			\end{equation}
		\end{proposition}
		Here are some high level ideas for the proof. 
		By definition, proving Proposition~\ref{prop-qa-mass} is equivalent to solving for the probability density function of the quantum length of the outermost loop around a marked point $\eta$. This is carried out in \cite{acsw-loop} using several inputs. First, \cite{MSW1} describes the joint law of the quantum lengths of all outermost CLE loops via the jumps of a growth-fragmentation process. Second, \cite{bbck-growth-frag} shows that this growth-fragmentation process describes the scaling limit of outermost loop lengths in $O(n)$-decorated quadrangulations. Third, \cite{ccm-perimeter-cascade} gives a L\'evy process description of the scaling limit of outermost loop lengths in $O(n)$-decorated quadrangulations. Together, these inputs yield a L\'evy process description of the quantum lengths of all outermost CLE loops \cite[Proposition 4.1]{acsw-loop}. A careful analysis of this L\'evy process gives Proposition~\ref{prop-qa-mass}, see \cite[Section 5]{acsw-loop} for details. 
	}
	
	Finally, Proposition~\ref{prop-qa-mass} implies that $|\QA^\gamma(a,b)| = |\QA^\gamma (b,a)|$, but we have the following stronger symmetry:
	\begin{proposition}[{\cite[Proposition 7.6]{acsw-loop}}]\label{prop:QA-symmetry}
		For each $a,b>0$, we have $\QA^\gamma(a,b)=\QA^\gamma(b,a)$. {In other words, the quantum annulus is invariant in law under  reordering of boundary components.}
	\end{proposition}
	{For the reader's convenience, we sketch here the proof of Proposition~\ref{prop:QA-symmetry}.
		\cite{wedges} introduced a canonical quantum surface called the \emph{two-pointed quantum sphere}; let $(\C, \phi, 0, \infty)$ be an embedding of this quantum surface in $\C$ where the first (resp.\ second) marked point is sent to $0$ (resp.\ $\infty$). 
		Let $\Gamma$ be an independent whole-plane $\CLE_\kappa$, let $\Gamma^{\mathrm{sep}} \subset \Gamma$ be the set of loops separating $0$ and $\infty$, and let $\eta \in \Gamma^\mathrm{sep}$ be sampled according to counting measure on $\Gamma^\mathrm{sep}$. Let $\eta' \in \Gamma^\mathrm{sep}$ be the outermost loop surrounded by $\eta_0$. Let $D_\eta$ be the domain surrounded by $\eta$, and $A_\eta$ the annular domain with boundary $\eta \cup \eta'$.
		By \cite[Theorem 1.1]{acsw-loop} $\eta_0$ is an SLE loop, so one can apply \cite[Theorem 1.1]{AHS-loop} to show that $(D_\eta, \phi, 0)/{\sim_\gamma}$ is a quantum disk. Consequently, Definition~\ref{def:QA} implies $(A_\eta, \phi)/{\sim_\gamma}$ is a quantum annulus. 
		Since the two-pointed quantum sphere is invariant in law under reordering of its marked points \cite[Proposition A.13]{wedges}, and CLE is invariant in law under inversion \cite[Theorem 1]{kw-cle}, we conclude that the quantum annulus is invariant in law under reordering of its boundary components. 
	}

	\subsection{Identification of the Liouville field via the domain Markov property}\label{subsec:QA-LF}
	
	\begin{proposition}  \label{prop:LF}
		There exists a measure $m(\rd \tau)$ on $(0,\infty)$ such that if
		we sample $(\tau, \phi)$ from $ \LF_\tau (d\phi) m(\rd \tau)$, then the law of the quantum surface $(\cC_\tau,\phi)/{\sim_\gamma}$ is $\QA^\gamma$. 
	\end{proposition}
	Our proof of  Proposition~\ref{prop:LF} based on the domain Markov property will be used multiple times to prove similar statements in later sections. 
	We recall the notions in Section~\ref{sec:LCFT}.
	Let $\cC(x,y)$ be the cylinder $(x,y) \times [0,1]/{\sim}$ where $\sim$ is the identification of $[x,y]\times \{0\}$ and $[x,y]\times\{1\}$. We will write the circle $\{z\} \times[0,1]/{\sim}$ as  $\{z\} \times[0,1]$ for simplicity. 
	Let $\cC_\tau=\cC(0,\tau)$ for each $\tau>0$. 
	{We prove the needed domain Markov property through the following two lemmas.} 
	\begin{lemma}\label{lem:DMP1}
		Let $(\phi,\Gamma)$ be a sample from $\LF^{\gamma,i}_\H\times \CLE_\kappa^\H$. Let $\eta$ be the outermost loop of $\Gamma$ surrounding $i$ and $A_\eta$ be the component of $\H \setminus \eta$ bounded by $\eta$ and $\partial \H$. Let $\hat\tau$ be such that $A_\eta$ is conformally equivalent to $\cC_{\hat \tau}$. {Fix $x\in  \{0\}\times [0,1]$} and let $\psi: \cC_{\hat\tau} \to A_\eta$  be the conformal map such that $\psi(x)=0$. Let $\hat \phi(z)=\phi \circ \psi(z) +Q\log |\psi(z)'|$ for $z\in \cC_{\hat\tau}$. Let $\hat\phi^{\mathrm{har}}$ be the harmonic extension of $\hat \phi|_{\cC(\frac{2\hat \tau}{3}, \hat \tau)}$ onto $\cC(0, \frac{2\hat \tau}{3})$ with zero normal derivative on $\{0\}\times  [0,1]$. Then conditioning on $\hat \tau$ and $\hat \phi|_{\cC(\frac{2\hat \tau}{3}, \hat \tau)}$, the conditional law of 
		$\hat \phi   - \hat \phi^{\mathrm{har}}$ is a GFF on $\cC(0, \frac{2\hat \tau}{3})$  with zero boundary condition on $\{ \frac{2\hat \tau}{3} \}\times [0,1]$ and free boundary condition on $\{0\}\times  [0,1]$.
	\end{lemma}
	\begin{proof}
		Let $\hat \eta\subset {\H}$ be the image of  $\{ \frac{2\hat \tau}{3} \}\times [0,1]$ under $\psi$. Then the law of $(\phi,{\hat\eta})$ is a product measure hence satisfies the first condition in Lemma~\ref{lem:local}. Let $D_{\hat\eta}$ and $A_{\hat\eta}$ be the disk and annulus component of $\H\setminus \hat\eta$, respectively. Let $\phi^{\mathrm{har}}$ be the harmonic extension of $\phi|_{D_{\hat\eta}}$ onto $A_{\hat\eta}$ with zero normal derivative on $\partial \H$ and satisfying $\lim_{|z| \to \infty} \phi^\mathrm{har}(z) / \log |z| = -2Q$.  
		Let $\phi_{A_{\hat \eta}} = \phi   - \phi^{\mathrm{har}}$. Then by Lemma~\ref{lem:local}, conditioning on $\eta$ and $\phi|_{D_{\hat\eta}}$, the conditional law of  $\phi_{A_{\hat\eta}}$ is a GFF on $A_{\hat\eta}$ with zero boundary condition on ${\hat\eta}$ and free boundary condition on $\partial\H$. {Since $(\hat\tau,  \hat\phi|_{\cC(\frac{2\hat \tau}{3}, \hat \tau)})$ is determined by $(\eta,\phi|_{D_{\hat\eta}})$, by the conformal invariance of GFF, conditioning on $(\hat\tau, \phi|_{\cC(\frac{2\hat \tau}{3}, \hat \tau)})$ the conditional law of 
			$\phi_{A_{\hat\eta}} \circ \psi$ is a GFF on $\cC(0, \frac{2\hat \tau}{3})$  with zero boundary condition on $\{ \frac{2\hat \tau}{3} \}\times [0,1]$ and free boundary condition on $\{0\}\times  [0,1]$. 
			
			Note that $\hat \phi- (\phi^{\mathrm{har}} \circ \psi +Q \log|\psi'|)= \phi \circ \psi - \phi^{\mathrm{har}}\circ \psi=\phi_{A_{\hat\eta}} \circ \psi$, which is zero on $\cC(\frac{2\hat \tau}{3},\hat\tau)$.
			{We just need} to show that $\hat\phi^{\mathrm{har}}=\phi^{\mathrm{har}} \circ \psi +Q \log|\psi'|$. To see this,  note that  $\phi^{\mathrm{har}} \circ \psi +Q \log|\psi'|$ is harmonic inside  $\cC(0,\frac{2\hat \tau}{3})$ since $\psi$ is conformal. It remains to show that $\phi^{\mathrm{har}} \circ \psi +Q \log|\psi'|$ has the desired boundary condition on $\{0\}\times [0,1]$. Let $\psi_1(z)=\frac{z-1}{i(z+1)}$ which maps the unit disk $\D$ to $\H$. Then the normal derivative of $\phi^{\mathrm{har}}_1:=\phi^{\mathrm{har}} \circ \psi_1+Q\log |\psi'_1|$ on $\partial \D$ is zero everywhere. For $w\in \cC_{\hat\tau}$, let $\psi_2(w)=\psi_1^{-1}(\psi(w))$ so that $\psi=\psi_1\circ \psi_2$. Then $\phi^{\mathrm{har}} \circ \psi +Q \log|\psi'|=\phi_1^{\mathrm{har}} \circ \psi_2+Q\log |\psi'_2|$. Since $\psi_2$ maps $\partial \D$ to  $\{0\}\times [0,1]$,  the normal derivative of $\phi^{\mathrm{har}} \circ \psi +Q \log|\psi'|$ is zero everywhere on  $\{0\}\times [0,1]$ as desired.}
	\end{proof}

	\begin{figure}[ht!]
		\begin{center}				
			
			\includegraphics[scale=0.7]{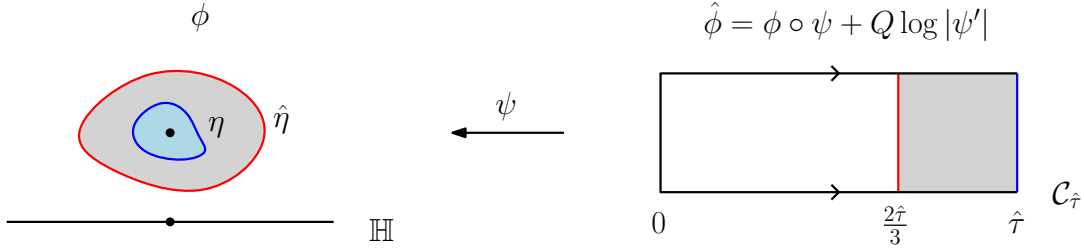}
			\caption{ 
				Illustration of objects in the proof of Lemma~\ref{lem:DMP1}.
			} \label{fig-QA}
		\end{center}
	\end{figure}

	\begin{lemma}\label{lem:DMP2}
		For a sample $(\phi,\Gamma)$ from $\LF^{\gamma,i}_\H\times \CLE_\kappa^\H$, let $\mathcal L^\gamma_\phi(\eta)$ be the quantum length of the outermost loop $\eta$ surrounding $i$. Let $f(b)=b|\LF^{\gamma, i}_\H(b)|$ and consider the reweighted measure  $M= \frac{1}{f(\mathcal L^\gamma_\phi(\eta) )}\cdot\LF^{\gamma,i}_\H\times \CLE_\kappa^\H$. {Let $A_\eta$ be the component of $\H \setminus \eta$ bounded by $\eta$ and $\partial \H$.} Then the $M$-law of $(A_\eta, \phi)/{\sim_\gamma}$ is $\QA^\gamma$. Moreover, the statement {in Lemma~\ref{lem:DMP1}} on  the conditional law of 
		$\hat \phi   - \hat \phi^{\mathrm{har}}$ given 
		$(\hat \tau,\hat \phi|_{\cC(\frac{2\hat \tau}{3}, \hat \tau)})$ still holds under $M$.
	\end{lemma}
	\begin{proof}
		By definition $\LF_\H^{\gamma, i} = \int_0^\infty \LF_\H^{\gamma, i}(a)\,\rd a$, so by~\eqref{eq:QA-weld} and our choice of reweighting, the  $M$-law of $(A_\eta, \phi)/{\sim_\gamma}$ is $\QA^\gamma$. For the assertion on the conditional law of 
		$\hat \phi   - \hat \phi^{\mathrm{har}}$, we note that the most crucial {point} in the proof of Lemma~\ref{lem:DMP1} is   the following. Conditioning on $\eta$ and $\phi|_{D_{\hat\eta}}$, the conditional law of  $\phi_{A_{\hat\eta}}$ is a mixed boundary GFF on $A_{\hat\eta}$. Since the quantum length $\mathcal L^\gamma_\phi(\eta)$ of $\eta$ is determined by $(\eta,\phi|_{D_{\hat\eta}})$,  under the reweighted measure $M$,  the conditional law of  $\phi_{A_{\hat\eta}}$  given $( \eta,\phi|_{D_{\hat\eta}})$ does not change. The rest of the proof of Lemma~\ref{lem:DMP1} now gives the {last} assertion in Lemma~\ref{lem:DMP2}.
	\end{proof}

	\begin{proof}[Proof of Proposition~\ref{prop:LF}]
		{Throughout this proof, we let $(\mathring \phi, \mathring \tau)$ be an  embedding of $\QA^\gamma$ such that  the law of $\mathring\phi$ is invariant under rotations against the axis of the cylinder $\cC(0, \mathring \tau)$. Since rotations are the only conformal automorphisms of a finite cylinder, the law of  $(\mathring \phi, \mathring \tau)$ is unique. We call $(\mathring \phi, \mathring \tau)$ the rotational invariant embedding of $\QA^\gamma$. We will show that the law of $(\mathring \phi, \mathring \tau)$ is $\LF(\rd\phi) m(\rd \tau)$ for some $m(\rd\tau)$, which will prove Proposition~\ref{prop:LF}.}
		
		Let $I$ be a finite interval on $(0,\infty)$. Given a quantum surface of annular topology, let $E_I$ be the event that both boundary lengths are in $I$. Since  $|\QA^\gamma(a,b)|$ is proportional to $\frac{1}{\sqrt{ab}(a+b)}$, we have  $\QA^\gamma(E_I)=\int_{I\times I} |\QA^\gamma (a,b)|\rd a\rd b<\infty$. 
		Let $\P_I$ be the probability measure describing the law of $(\mathring \phi,\mathring \tau)$ 
		conditioning on $E_I$. For $\tau>0$, let $P_{I,\tau}$ be the regular conditional probability of $\mathring\phi$ under $\P_I$ given $\{\mathring\tau=\tau\}$. 
		
		{Note that the domain Markov property in Lemmas~\ref{lem:DMP1} and~\ref{lem:DMP2} holds for each fixed $x\in \{0\}\times [0,1]$ in the definition of the conformal map $\psi$ from $\cC_{\hat \tau}$ to $A_\eta$. Therefore Lemmas~\ref{lem:DMP1} and~\ref{lem:DMP2} still hold if we choose $x$ uniformly at random on $\{0\}\times [0,1]$. This way, under the measure $M$ in Lemma~\ref{lem:DMP2} $(\hat \phi,\hat \tau)$ defined there is a rotational invariant embedding of $\QA^\gamma$, hence equal in law to $(\mathring \phi, \mathring \tau)$. Therefore by the domain Markov property in} Lemma~\ref{lem:DMP2}, $P_{I,\tau}$ is an invariant probability measure of the kernel $\Lambda_{I}$ in Proposition~\ref{prop:inv-unique}.

		By the symmetry in Proposition~\ref{prop:QA-symmetry}, 
		$P_{I,\tau}$ is also an invariant probability measure of  the kernel $\bar \Lambda_{I}$. Hence the uniqueness of invariant measure in Proposition~\ref{prop:inv-unique} gives   $P_{I,\tau}=\LF_{\tau,I}$. More precisely, for each $I$,  $\LF_{\tau,I}$ is a version of the regular conditional probability of $\mathring \phi$ given $\{\mathring\tau=\tau\}$. We now write the assertion of Proposition~\ref{prop:LF} as 
		$\QA^\gamma= \int \LF_\tau m (d\tau)$ for simplicity. Then $P_{I,\tau}=\LF_{\tau,I}$ implies that {restricted to $E_I$}, we have  $\QA^\gamma= \int \LF_\tau m_I (\rd \tau)$  for some measure $m_I(d\tau)$ on $(0,\infty)$. {Setting $m (\rd \tau)=m_I(\rd \tau)$ on each interval $I$ we get a well-defined measure $m (\rd \tau)$ on $(0,\infty)$. Then    $\QA^\gamma=\int \LF_\tau m (\rd \tau)$.}
	\end{proof}

	\subsection{Identification of the law of the modulus by comparing the boundary lengths}\label{subsec:QA-Z}
	By Theorem~\ref{thm:KPZ} and Proposition~\ref{prop:LF}, the measure $m(\rd \tau) $ satisfies the following: for $x\in \R$  
	\begin{align}\label{eq:cm-simple}
	\int_0^\infty e^{ -\frac{\pi\gamma^2 x^2\tau}{4}}m(\rd \tau )=
	\frac{2\sinh(\frac{\gamma^2}{4} \pi x  )}{\pi\gamma x\Gamma(1+ix)}\QA^\gamma [L_1 e^{-L_1} L_0^{ix} ].
	\end{align} 
	Here we abuse notion and use $L_0$ and $L_1$ to represent the quantum lengths of the two boundaries of a sample from $\QA^\gamma$.
	We now identify the measure $m$ hence proving Theorem~\ref{thm:QAconj}. 
	\begin{lemma}\label{lem:QA-Mellin}
		For $\gamma\in (\sqrt{\frac{8}{3}},2)$, we have
		\begin{equation}\label{eq:QA-Mellin} 
		\QA^\gamma [L_1 e^{-L_1} L_0^{ix} ]={\cos(\pi (\frac4{\gamma^2}-1))} \cdot \frac{\Gamma(1+ix)}{\cosh(\pi x)}\quad \textrm{for }x\in \R.
		\end{equation}
	\end{lemma}
	\begin{proof} 
		By Proposition~\ref{prop-qa-mass},
		the total mass $|\QA^\gamma(a, b)|$ of $\QA^\gamma(a, b)$ is given by
		\begin{equation}
		|\QA^\gamma(a, b)| = \frac{\cos(\pi (\frac4{\gamma^2}-1))}\pi  \cdot \frac {1 }{\sqrt{ab} (a+b)} \quad \textrm{for }a>0\textrm{ and }b>0.
		\end{equation}
		Therefore
		\begin{align*}
		\QA^\gamma [L_1 e^{-L_1} L_0^{ix} ]=\iint_0^\infty ae^{-a} b^{ix} |\QA^\gamma(a,b)|\,\rd a\,\rd b=\frac{\cos(\pi (\frac4{\gamma^2}-1))}\pi \iint_0^\infty \frac {ae^{-a} b^{ix}}{\sqrt{ab} (a+b)}\,\rd a\,\rd b.
		\end{align*} 
		Setting $b=at$ and using $\int_0^\infty a^{ix} e^{-a}\,\rd a=\Gamma(1+ix)$ and $\int_0^\infty \frac { t^{ix-\frac12} }{1+t}\,\rd t=\frac{\pi}{\cosh(\pi x)}$,  we get
		\begin{align}\label{eq:integral}
		\iint_0^\infty   \frac {ae^{-a} b^{ix} }{\sqrt{ab} (a+b)}\,\rd a\,\rd b=\iint_0^\infty  \frac {  ae^{-a}(at)^{ix} }{a\sqrt t(1+t)}\,\rd a\,\rd t=\int_0^\infty e^{-a} a^{ix} \,\rd a \int_0^\infty \frac { t^{ix-\frac12} }{1+t}\, \rd t=\frac{\pi\Gamma(1+ix)}{\cosh(\pi x)}.
		\end{align}
		This gives \eqref{eq:QA-Mellin}.
	\end{proof}
	
	\begin{proof}[Proof of Theorem~\ref{thm:QAconj}]
		By~\eqref{eq:cm-simple} and Lemma~\ref{lem:QA-Mellin} we have 
		\begin{equation}
		\label{eq:cm-Lap}
		\int_0^\infty e^{ -\frac{\pi\gamma^2 x^2\tau}{4}} m(\rd\tau)= 
		{\cos(\pi (\frac4{\gamma^2}-1))} \cdot \frac{2\sinh(\frac{\gamma^2}{4} \pi x  )}{\pi \gamma x\cosh(\pi x)} \quad \textrm{for } x\in \R.
		\end{equation}
		By \eqref{eq:theta-L} in  Appendix~\ref{app:theta}, we have
		\begin{equation}\label{eq:theta1-L}
		\int_0^{\infty}   \theta_1(\frac{\gamma^2}{8}, \frac{\gamma^2}{4}i\tau)  e^{ -\frac{\pi\gamma^2 x^2\tau}{4}} d\tau= 
		\frac{4\sinh(\frac{\gamma^2}{4} \pi x )}{\gamma^2 x\cosh(\pi x)} \quad \textrm{for }x\in \R.
		\end{equation}
		Taking the inverse Laplace transform,  we get \( m(\rd \tau)={\cos(\pi (\frac4{\gamma^2}-1))} \cdot \frac{\gamma}{2\pi } \theta_1(\frac{\gamma^2}{8}, \frac{\gamma^2}{4}i\tau)\,  \rd \tau\) on $(0,\infty)$.
	\end{proof}

	\section{The random moduli for  CLE loops}\label{sec:CLE}
	
	In this section we prove the three results on $\CLE_\kappa$ presented  in Section~\ref{subsec:intro-CLE}.  We will work with $(\H,i)$ instead of $(\D,0)$. Namely, let  $\kappa \in (8/3,4]$.
	We assume that $\Gamma$ is a $\CLE_\kappa$ on the upper half plane $\H$. Given $\Gamma$,  we let $\{\eta_j\}_{j\ge1}$ be the sequence of loops surrounding the point $i$ ordered such that $\eta_{j+1}$ is surrounded by $\eta_j$.  Moreover, let $D_j$ be the Jordan domain bounded by $\eta_j$, and $A_j$ be the annulus bounded by $\eta_j$ and $\partial\H$.
	Finally, let $\mathrm{Mod}(\eta_j)$ be the modulus of $A_j$.
	\subsection{Random modulus of a single loop: proof of Theorem~\ref{thm:CLE-mod}}\label{subsection:single}
	We first prove the $j=1$ case of Theorem~\ref{thm:CLE-mod} 
	{in the case $\kappa\in (\frac83,4)$.}
	\begin{proposition}\label{prop:outermost}
		For $\kappa\in (\frac83,4)$ and $\gamma=\sqrt{\kappa}$, let $m(\rd \tau )=1_{\tau>0}\cdot {\cos(\pi (\frac4{\gamma^2}-1))} \cdot \frac{\gamma}{2\pi } \theta_1(\frac{\gamma^2}{8}, \frac{\gamma^2}{4}i\tau) \rd\tau$. Then  the law of $\mathrm{Mod}(\eta_1)$ has a density proportional to $e^{\frac{\pi\kappa}{4} (\frac{4}{\kappa}-1)^2 \tau}  m(\rd\tau)$. Moreover,
		\begin{equation}\label{eq:moment-1}
		\E[e^{-2\pi \lambda \mathrm{Mod}(\eta_1)}]=\frac{ (\frac{4}{\kappa}-1)\cos(\pi (\frac{4}{\kappa}-1))}{  \sin( \pi (1-\frac{\kappa}{4}))} \times       \frac{  \sin(\frac{\kappa}{4} \pi \sqrt{   (\frac{4}{\kappa}-1)^2 -\frac{8\lambda}{\kappa}  } )}{ \sqrt{   (\frac{4}{\kappa}-1)^2 -\frac{8\lambda}{\kappa}  }\cos(\pi \sqrt{   (\frac{4}{\kappa}-1)^2 -\frac{8\lambda}{\kappa}  })}
		\end{equation}
		for \(\lambda >\frac{ 3\kappa}{32}+\frac{2}{\kappa} -1\). 
	\end{proposition}
	\begin{proof}
		For $\tau>0$, define $\LF_\tau(a,b)$ by the disintegration
		\begin{equation}\label{eq:LF-ab}
		\LF_\tau = \iint_0^\infty \LF_\tau(a,b) \rd a \rd b 
		\end{equation}
		{where $\LF_\tau(a, b)$ is supported on the event that the quantum boundary lengths of  $\{0\}\times [0,1]$ and $\{\tau\}\times [0,1]$  are $a$ and $b$, respectively.}
		Let $C_\gamma\in (0,\infty)$ be a  $\gamma$-dependent constant that can vary from line to line.  Let $f(\ell)= \ell |\LF_{\H}^{(\gamma, i)}(\ell)|$, which equals $C_\gamma {\ell^{1-\frac4{\gamma^2}}}$ for some $C_\gamma$.
		By~\eqref{eq:QA-weld} and Theorem~\ref{thm:QAconj}, for a non-negative measurable function $\rho$ on $(0,\infty)$ we have 
		\begin{align*}
		\LF_{\H}^{(\gamma, i)}(a)\times \CLE_\kappa^\H[\rho(\mathrm{Mod}(\eta_1))]= \int_0^\infty \int_0^\infty \rho(\tau) |\LF_\tau(a,b)| f(b) \,\rd b \, m (\rd\tau).
		\end{align*}
		{For a sample from $ \LF_{\H}^{(\gamma, i)}$ let $L$ be the quantum boundary length of $\partial \H$.} Integrating over $ae^{-a}\rd a$, we have 
		\begin{equation}\label{eq:m1}
		\LF_{\H}^{(\gamma, i)}[{L e^{-L}}] \times  \CLE_\kappa^\H[\rho(\mathrm{Mod}(\eta_1))]=  \int_0^\infty \rho(\tau) \LF_\tau[L_1e^{-L_1}f(L_0)] 
		\,m (\rd\tau).
		\end{equation}
		Recall from~\eqref{eq:fraction-real} that \(\LF_\tau[L_1e^{-L_1} L_0^{y}]=    \frac{\pi\gamma y\Gamma(1+y)}{   2\sin(\frac{\gamma^2}{4} \pi y  )}e^{ \frac{\pi\gamma^2 y^2\tau}{4}}\) for $y\in \R$. This gives  
		\[
		\frac{\LF_\tau[L_1 e^{-L_1} 
			f(L_0)]}{\LF_{\H}^{(\gamma, i)} [{L e^{-L}} ] }=C_\gamma e^{\frac{\pi\gamma^2 \tau}{4} \cdot (1-\frac{4}{\gamma^2})^2}=C_\gamma 
		e^{\frac{\pi\kappa}{4} \cdot (\frac{4}{\kappa}-1)^2 \tau}. 
		\]
		Therefore the law of $\mathrm{Mod}(\eta_1)$  is $C_\gamma 
		e^{\frac{\pi\kappa}{4} \cdot (\frac{4}{\kappa}-1)^2 \tau} m(\rd \tau)$ as desired. 
		Consequently,  
		\begin{equation}\label{eq:moment}
		\E[e^{-2\pi \lambda \mathrm{Mod}(\eta_1)}] =C_\gamma\int_0^\infty  e^{-2\pi \lambda \tau}  
		e^{\frac{\pi\kappa}{4} (\frac{4}{\kappa}-1)^2 \tau}  m(\rd\tau).
		\end{equation}
		By~\eqref{eq:theta} the function $\tau \mapsto \theta_1(\frac{\gamma^2}8, \frac{\gamma^2}4 i \tau)$ decays as $O(e^{-\frac{\pi \gamma^2}{16} \tau})$ as $\tau \to \infty$. Therefore $\E[e^{-2\pi \lambda \mathrm{Mod}(\eta_1)}]$ is finite and analytic {in $\lambda$} on a complex neighborhood of $(\frac{3\kappa}{32} + \frac2\kappa-1, \infty)$. By~\eqref{eq:moment} and~\eqref{eq:theta-L}, for $\lambda$ {sufficiently large,  we can evaluate the integral on the right side of~\eqref{eq:moment-C} to get}
		\begin{equation}\label{eq:moment-C}
		\E[e^{-2\pi \lambda \mathrm{Mod}(\eta_1)}]=    \frac{C_\gamma  \sin(\frac{\kappa}{4} \pi \sqrt{   (\frac{4}{\kappa}-1)^2 -\frac{8\lambda}{\kappa}  } )}{ \sqrt{   (\frac{4}{\kappa}-1)^2 -\frac{8\lambda}{\kappa}  }\cos(\pi \sqrt{   (\frac{4}{\kappa}-1)^2 -\frac{8\lambda}{\kappa}  })}.
		\end{equation}
		By the analyticity in $\lambda$~\eqref{eq:moment-C} holds for  all $\lambda \in (\frac{3\kappa}{32} + \frac2\kappa-1, \infty)$. Finally, setting $\lambda=0$ we see that {the constant $C_\gamma$ in~\eqref{eq:moment} equals} $\frac{ (\frac{4}{\kappa}-1)\cos(\pi (\frac{4}{\kappa}-1))}{  \sin( \pi (1-\frac{\kappa}{4}))}$, completing the proof.
	\end{proof}
	
	Our derivation for $\eta_1$ works exactly the same way for $\eta_j$ with $j\ge2$. We first introduce the analog of $\QA^\gamma$ following Section~\ref{sec:QA}. 
	\begin{definition}
		Fix $j\ge2$. For $a>0$, let $(\phi,\Gamma$)  be a sample from $\LF_\H^{\gamma,i}(a)\times \CLE_\kappa^\H$ and let $\eta_j$ be the $j$-th outermost loop of $\Gamma$ surrounding $i$.
		For $b>0$, let $\QA_j^\gamma(a,b)^{\#}$ be the  conditional law of the quantum surface $(A_j, \phi)/{\sim_\gamma}$ conditioning on $\mathcal L^\gamma_\phi(\eta) = b$. 
	\end{definition}
	\begin{proposition}\label{prop:QAk}
		Let $\QA^\gamma_1=\QA^\gamma$. For $j\ge2$,  let  $\QA^\gamma_j(a,b)= |\QA^\gamma_j(a,b)| \QA^\gamma_j(a,b)^{\#}$ where the total mass $|\QA_j^\gamma(a,b)|$ of $\QA_j^\gamma(a,b)$ is given by 
		\begin{equation}\label{eq:def-QAn}
		|\QA_j^\gamma(a,b)|=\int_0^\infty \ell  |\QA^\gamma(a,\ell)| |\QA^\gamma_{j-1}(\ell,b)| \rd \ell.
		\end{equation}
		Then the joint law of $(A_j,h)/{\sim_\gamma}$ and $(D_j,\phi,0)/{\sim_\gamma}$ is
		\begin{equation}\label{eq:QA-weld-j}
		\int_0^\infty b {\QA^\gamma_j}(a,b)\times \Md(\gamma;b) \rd b \quad \textrm{for }j\ge 1.
		\end{equation}
	\end{proposition}
	\begin{proof}
		By the Markov property of $\CLE_\kappa$, conditioning on $\eta_j$, the law of the loops of inside $D_j$ is a $\CLE_\kappa$ on $D_j$.  Iteratively applying~\eqref{eq:QA-weld} we get~\eqref{eq:QA-weld-j}.
	\end{proof}
	We now extend  Theorem~\ref{thm:QAconj} to  $\QA^\gamma_j$.
	\begin{proposition}\label{prop:QA-LF-k}
		Fix $j\ge2$. Let $\QA^\gamma_j=\iint_0^\infty \QA^\gamma_j(a,b) \rd a \rd b$. Let  $m_j(\rd\tau)$ be the unique  measure on $(0,\infty)$ such that
		\begin{equation}
		\label{eq:cj-Lap-k}
		\int_0^\infty e^{ -\frac{\pi\gamma^2 x^2\tau}{4}} m_j(\rd\tau)= {\cos^{j}(\pi (\frac4{\gamma^2}-1))} 
		\cdot \frac{2\sinh(\frac{\gamma^2}{4} \pi x  )}{\pi \gamma x\cosh^j(\pi x)} \quad \textrm{for } x\in \R.
		\end{equation}
		Sample $(\tau, \phi)$ from $\LF_\tau (\rd\phi) m_j(\rd \tau)$. Then the law of the quantum surface $(\cC_\tau,\phi)/{\sim_\gamma}$ is $\QA^\gamma_j$.
	\end{proposition}
	{To prove Proposition~\ref{prop:QA-LF-k},} we first prove the analog of Lemma~\ref{lem:QA-Mellin}.
	\begin{lemma}\label{lem:QA-Mellin-k}
		For $j\ge2$, we have
		\begin{align}
		\label{eq:QAn-Mellin-k} 
		\QA^\gamma_j [L_1 e^{-L_1} L_0^{ix} ]= \frac{\cos(\pi (\frac4{\gamma^2}-1))}{\cosh(\pi x)}\QA^\gamma_{j-1}[L_1 e^{-L_1} L_0^{ix} ] \quad \textrm{for } x\in \R.
		\end{align}
	\end{lemma} 
	\begin{proof}
		By definition, $\QA^\gamma_j [L_1 e^{-L_1} L_0^{ix} ]=\iiint_0^\infty ae^{-a} b^{ix} \ell |\QA^\gamma_{j-1}(a,\ell)||\QA(\ell,b)| \rd\ell \rd a\rd b$, which equals  
		\begin{align*}
		\frac{\cos(\pi (\frac4{\gamma^2}-1))}\pi  \iiint_0^\infty \frac {ae^{-a} b^{ix}\ell|\QA^\gamma_{j-1}(a,\ell)|}{\sqrt{\ell b} (\ell+b)} \rd\ell \rd a \rd b.
		\end{align*}
		Setting $b=\ell t$ and using  $\int_0^\infty \frac { t^{ix} }{\sqrt{t}(1+t)}dt=\frac{\pi}{\cosh(\pi x)}$,  we get
		\begin{align*}
		&\iiint_0^\infty \frac {ae^{-a} b^{ix}\ell|\QA^\gamma_{j-1}(a,\ell)|}{\sqrt{\ell b} (\ell+b)}\rd\ell \rd a \rd b=
		\iiint_0^\infty \frac {ae^{-a} (\ell t)^{ix} |\QA^\gamma_{j-1}(a,\ell)|}{\sqrt{ t} (1+t)}\rd\ell \rd a\rd t\\
		=& \frac{\pi}{\cosh(\pi x)}\iint_0^\infty
		ae^{-a}\ell^{ix} |\QA^\gamma_{j-1}(a,\ell)|\rd\ell \rd a=\frac{\pi}{\cosh(\pi x)}\QA^\gamma_{j-1}[L_1 e^{-L_1} L_0^{ix} ]. \qedhere
		\end{align*} \end{proof}
	\begin{proof}[Proof of Proposition~\ref{prop:QA-LF-k}]
		By the identical argument as for Proposition~\ref{prop:LF}, there exists a measure $m_j(\rd \tau)$ on $(0,\infty)$ such that if
		we sample $(\tau, \phi)$ from $ \LF_\tau (d\phi) m_j(\rd \tau)$, then the law of the quantum surface $(\bbA_\tau,\phi)/{\sim_\gamma}$ is $\QA^\gamma_j$. 
		Similar as for the $j=1$ case in~\eqref{eq:cm-simple}, by Theorem~\ref{thm:KPZ} we have 
		\begin{align}\label{eq:cj-simple-k}
		\int_0^\infty e^{ -\frac{\pi\gamma^2 x^2\tau}{4}}m_j(\rd \tau )=\frac{2\sinh(\frac{\gamma^2}{4} \pi x  )}{\pi\gamma x\Gamma(1+ix)}\QA^\gamma_j [L_1 e^{-L_1} L_0^{ix}] \quad \textrm{for } x\in \R.
		\end{align}
		By Lemmas~\ref{lem:QA-Mellin} and~\ref{lem:QA-Mellin-k}, we conclude the proof.
	\end{proof}
	
	\begin{proof}[Proof of Theorem~\ref{thm:CLE-mod}]
		{First fix $\kappa=\gamma^2\in (\frac83,4)$.} By the same calculation as for~\eqref{eq:moment} we have
		\begin{equation}\label{eq:moment-k}
		\E[e^{-2\pi \lambda \mathrm{Mod}(\eta_j)}] =C_\gamma\int_0^\infty  e^{-2\pi \lambda \tau}  
		e^{\frac{\pi\kappa}{4} (\frac{4}{\kappa}-1)^2 \tau}  m_j(\rd\tau).
		\end{equation}
		for some $\gamma$-dependent constant $C_\gamma$. By Proposition~\ref{prop:QA-LF-k} and the argument for~\eqref{eq:moment},
		we have
		\begin{equation}\label{eq:moment-Ck}
		\E[e^{-2\pi \lambda \mathrm{Mod}(\eta_j)}]=    \frac{C_\kappa  \sin(\frac{\kappa}{4} \pi \sqrt{   (\frac{4}{\kappa}-1)^2 -\frac{8\lambda}{\kappa}  } )}{ \sqrt{   (\frac{4}{\kappa}-1)^2 -\frac{8\lambda}{\kappa}  }\cos^j(\pi \sqrt{   (\frac{4}{\kappa}-1)^2 -\frac{8\lambda}{\kappa}  })} \quad \text{ for }\lambda>\frac{ 3\kappa}{32} + \frac{2}{\kappa} -1
		\end{equation} for some constant $C_\kappa$.
		Setting $\lambda=0$ we get $C_\kappa=\frac{ (\frac{4}{\kappa}-1)\cos^j(\pi (\frac{4}{\kappa}-1))}{  \sin( \pi (1-\frac{\kappa}{4}))}$ as desired. This proves Theorem~\ref{thm:CLE-mod} for $\kappa\in (\frac{8}{3},4)$. The $\kappa=4$ case follows from sending $\kappa\to 4$. 
		{See \cite[Lemma A.5]{acsw-loop} for the needed continuity in $\kappa$.}
	\end{proof}

	\begin{remark}[Relation to Brownian motion]\label{rmk:BM}
		For $a\ge 0$, let $B^a$ be a Brownian motion starting from $a$ and ${Y_a}= \inf\{ t: |B^a_t| =1\}$ and $T_0 = \sup\{ t<{Y_0}: B^0_t=0  \} $.   Then as explained in e.g.~\cite{BPY-BM}, we have 
		\[
		\E[ e^{-\frac{\theta^2}{2}{Y_a}} ] = \frac{\cosh(a\theta)}{\cosh\theta}\quad  
		\textrm{and}\quad \E[ e^{-\frac{\theta^2}{2}T_0} ] = \frac{\tanh \theta}{\theta} \quad \textrm{for }\theta\in\R.
		\]
		For $j\ge 1$ let  $m_j(\rd\tau)$ be the measure such that $  \E[e^{-2\pi \lambda \mathrm{Mod}(\eta_j)}] =C_\gamma\int_0^\infty  e^{-2\pi \lambda \tau}  
		e^{\frac{\pi\kappa}{4} (\frac{4}{\kappa}-1)^2 \tau}  m_j(\rd\tau)$ as in~\eqref{eq:moment-k}.
		Let $M_j$ be  sampled from the probability measure proportional to  $m_j(\rd \tau)$. Then by~\eqref{eq:cm-Lap} we have the equality $2\pi M_1\overset{d}{=}  \frac{\kappa\pi^2}{4} T_0+ \frac{4\pi^2}{\kappa} {Y_{\frac{\kappa}{4}}}$ in law. For $\kappa=4$, since ${Y_1}=0$  we have   $M_1\overset{d}{=} \frac{\pi}{2} T_0$ as shown in \cite[{Theorem 1.3}]{ALS20}. Since $\E[ e^{-\frac{\theta^2}{2}{Y_0}} ] = \frac{1}{\cosh \theta }$, by~\eqref{eq:moment-k} we have $2\pi M_j\overset{d}{=} 2\pi M_1+\sum_{m=1}^{j-1}X_m$ where $X_m$'s are independent copies of ${Y_0}$ that are independent with $M_1$.
	\end{remark}
	
	\subsection{Joint law with the conformal radius: proof of Theorem~\ref{thm:mod-CR}}\label{subsec:CR}
	Let $\mathrm{Loop}_\kappa$ be the law of $\eta_1$, which is the outermost loop surrounding $i$ of a $\CLE_\kappa$ on $\H$ . 
	For a sample $\eta$ from $\mathrm{Loop}_\kappa$, let $\psi_\eta:\H\to D_\eta(i)$ be the unique conformal map with $\psi(i)=i$ and $\psi'(i)>0$. For $\alpha\in \R$, let  
	$\mathrm{Loop}_\kappa^\alpha$ be the  measure defined by reweighting $\mathrm{Loop}_\kappa$ as follows:   
	\eqb\label{eq-j-alpha}
	\frac{d\mathrm{Loop}_\kappa^\alpha}{d\mathrm{Loop}_\kappa}(\eta) = |\psi_\eta'(i)|^{2\Delta_\alpha - 2} \quad \textrm{where }\Delta(\alpha)=\frac\alpha2(Q-\frac\alpha2).
	\eqe
	Recall $\Md(\alpha)$ from Definition~\ref{def:QD}. We  use the following proposition to prove Theorem~\ref{thm:mod-CR}.
	\begin{proposition}\label{prop:weld-CR}
		For $\kappa\in (\frac{8}{3},4)$, $\gamma=\sqrt{\kappa}$ and $\alpha>\frac{\gamma}{2}$, let $(\phi, \eta_1)$ be sampled from $\LF^{\alpha, i}_\H\times \mathrm{Loop}_\kappa^\alpha$. Let $A_1$ be the annulus bounded by $\eta_1$ and $\partial \H$, and $D_1$ be the Jordan domain bounded by $\eta_1$. Then the joint law of   $(A_1,\phi)/{\sim_\gamma}$ and $(D_1,\phi,0)/{\sim_\gamma}$ is
		\begin{equation}\label{eq:QA-weld-k}
		\int_0^\infty b \QA^\gamma(a,b)\times \Md(\alpha;b) \,\rd b.
		\end{equation}
	\end{proposition}
	\begin{proof}
		When $\alpha=\gamma$, this is simply the $j=1$ case of Proposition~\ref{prop:QAk}.
		The general $\alpha$ case can be obtained from the $\alpha=\gamma$ case via a reweighting argument as in \cite[Section 4.2]{ARS-FZZ}. There the loop touches $\partial \H$ but the exact same argument works here.
	\end{proof}
	
	\begin{proof}[Proof of Theorem~\ref{thm:mod-CR}]
		Let ${f_\alpha}(\ell)=|\Md(\alpha; \ell)|\ell ={|\LF_\H^{\alpha, i}(\ell)|\ell}={C_{\alpha,\gamma}} \ell^{{\frac{2}{\gamma}(\alpha-Q)}}$ for some {$(\alpha,\gamma)$}-dependent constant {$C_{\alpha,\gamma}$}. By Proposition~\ref{prop:weld-CR}, the identity~\eqref{eq:m1} holds with $\gamma$-insertion replaced by a generic $\alpha$. Namely
		\begin{equation}\label{eq:malpha}
		\LF_{\H}^{(\alpha, i)}[{L e^{-L}}] \times  \mathrm{Loop}_\kappa^\alpha[\rho(\mathrm{Mod}(\eta_1))]=  \int_0^\infty \rho(\tau) \LF_\tau[L_1e^{-L_1}{f_\alpha}(L_0)]   m (\rd\tau).
		\end{equation} 
		where $m(\rd \tau)=1_{\tau>0}\cdot {\cos(\pi (\frac4{\gamma^2}-1))} \cdot \frac{\gamma}{4\pi } \theta_1(\frac{\gamma^2}{8}, \frac{\gamma^2}{4}i\tau)\rd \tau$] and  $\rho$ is a non-negative measurable function.
		
		Let $\E$ be the expectation for $\CLE_\kappa^\H$. 
		Let $X=e^{2\pi \mathrm{Mod}(\eta_1) }|\psi_\eta'(i)|$. Then for   $\alpha>\frac{\gamma}{2}$ we have 
		\begin{equation}\label{eq:CR-key}
		\E[X^{2\Delta_\alpha-2} \rho(\mathrm{Mod}(\eta_1))]=\mathrm{Loop}_\kappa^\alpha[ e^{2\pi \mathrm{Mod}(\eta_1)(2\Delta_\alpha-2) } \rho(\mathrm{Mod}(\eta_1))]=\int_0^\infty \rho(\tau) g(\tau, \alpha)m (\rd\tau)
		\end{equation}
		where 
		$g(\tau,\alpha)= \frac{ e^{2\pi(2\Delta_\alpha-2)\tau}\LF_\tau[L_1e^{-L_1}f(L_0)]}{\LF_{\H}^{(\alpha, i)}[{L e^{-L}}]}  $. Since by~\eqref{eq:fraction-real}, for $y \in (-1, \frac4{\gamma^2})$ we have
		\[
		\LF_\tau[L_1e^{-L_1} L_0^{y}]=    \frac{\pi\gamma x\Gamma(1+y)}{   2\sin(\frac{\gamma^2}{4} \pi y )}e^{ \frac{\pi\gamma^2 y^2\tau}{4}} \quad \textrm{and}\quad {e^{\frac{\pi\tau\gamma^2 }{4}(\frac{2}{
					\gamma}(Q-\alpha))^2 }e^{2\pi(2\Delta_\alpha-2)\tau}=e^{\pi\tau(Q^2-4)},}
		\]
		for $\alpha \in (\frac2\gamma, Q + \frac2\gamma)$ the function ${(\alpha,\tau)\mapsto} e^{2\pi(2\Delta_\alpha-2)\tau}\LF_\tau[L_1e^{-L_1}{L_0^{\frac{2}{\gamma}(\alpha-Q)}}]$ {can be factorized as a function of $\alpha$ times a function of $\tau$.} Therefore  $g(\alpha,\tau)$, and subsequently~\eqref{eq:CR-key}, can be factorized as a function of $\alpha$ {times} a function of $\tau$ for $\alpha$ in an open interval. 
		We conclude $X$ is independent of $\mathrm{Mod}(\eta_1)$. Since $X=e^{2\pi \mathrm{Mod}(\eta_1) }|\psi_\eta'(i)|$, by the definition of conformal radius we obtain Theorem~\ref{thm:mod-CR} with $j=1$ and $\kappa\in (8/3,4)$. 
		For $j\ge 2$, the identity~\eqref{eq:CR-key} holds with $m(\rd\tau)$ replaced by $m_j(\rd\tau)$ from Proposition~\ref{prop:QA-LF-k}. This gives Theorem~\ref{thm:mod-CR} for $\kappa\in (8/3,4)$. The $\kappa=4$ case follows  by sending  $\kappa\uparrow 4$ {as in the proof of Theorem~\ref{thm:CLE-mod}.}    
	\end{proof}
	
	\subsection{Nesting statistics: proof of Theorem~\ref{thm:CLE-moduli}}\label{subsec:nesting}
	By Theorem~\ref{thm:CLE-mod}, we have for 
	\begin{equation}\label{eq:beta}
	\sum_{j\ge1} \beta^{j-1}\E[e^{-2\pi \lambda \mathrm{Mod}(\eta_j)}]=\frac{(\frac{4}{\kappa}-1) \sin(\frac{\kappa}{4} \pi \sqrt{   (\frac{4}{\kappa}-1)^2 -\frac{8\lambda}{\kappa}  } )}{ \sin( \pi (1-\frac{\kappa}{4})) \sqrt{   (\frac{4}{\kappa}-1)^2 -\frac{8\lambda}{\kappa}  }} 
	\frac{\cos(\pi (\frac{4}{\kappa}-1))}{\cos(\pi \sqrt{   (\frac{4}{\kappa}-1)^2 -\frac{8\lambda}{\kappa}  })-\beta \cos(\pi (\frac{4}{\kappa}-1))}.
	\end{equation}
	Suppose $t=2\pi \lambda-\frac{\kappa\pi}{4}
	(\frac{4}{\kappa}-1)^2>0$. 
	Set $C_\kappa=  \frac{\sqrt{\kappa \pi}(\frac{4}{\kappa}-1)\cos(\pi (\frac{4}{\kappa}-1))}
	{ 2\sin( \pi (1-\frac{\kappa}{4})) }$. By~\eqref{eq:beta}, we have 
	\begin{equation*}
	\sum_{j\ge1} \beta^{j-1}\E [ e^{- t \mathrm{Mod}(\eta_j)} e^{  -\frac{\kappa\pi}{4}
		(\frac{4}{\kappa}-1)^2 \mathrm{Mod}(\eta_j)}]  =\frac{ \sinh(\frac{\kappa}{4} \pi \sqrt{\frac{4t}{\kappa \pi }    } )}{  \sqrt{\frac{4t}{\kappa \pi }  }} 
	\frac{C_\kappa}{\cosh(\pi \sqrt{ \frac{4t}{\kappa \pi }  })-\beta \cos(\pi (\frac{4}{\kappa}-1))}
	\end{equation*}
	Let $n=2\cos(\pi(1-\frac{4}{\kappa}))$.  For $n'\in [-2,2]$, let $\beta=\frac{n'}{n}$ so that 
	$\frac{n'}{2}= \beta \cos(\pi(1-\frac{4}{\kappa}))$. 
	Let $\chi'=-\arccos(\frac{n'}{2})$. Then for $f(\tau)=e^{- t \tau} e^{  -\frac{\kappa\pi}{4}
		(\frac{4}{\kappa}-1)^2 \tau}$, we have 
	\begin{equation}\label{eq:CLE-f}
	\E\left[ \sum_{j=1}^{\infty} \left(\frac{n'}{n}\right)^{j-1} f(\mathrm{Mod}(\eta_n))\right] = 
	\frac{C_\kappa \sinh(\sqrt{\kappa\pi t/4})}{{\sqrt {t}}({\cosh\sqrt{4\pi t/\kappa}-\cos \chi'})}.
	\end{equation} 
	Now Theorem~\ref{thm:CLE-moduli} is immediate from the following fact on the function $Z(\tau, \kappa, \chi')$ in~\eqref{eq:On}.
	\begin{lemma}\label{lem:Cardy-Lap}
		Let $\mathcal L[f(\tau)](t)=\int_0^\infty e^{-t\tau} f(\tau) d\tau$ 
		{be the Laplace transform}.  For $\kappa\in (\frac83,4]$ and $\chi'\in \R$, 
		\begin{align}
		\label{eq:Cardy-Lap}
		\mathcal L[ Z(\tau, \kappa, \chi') \eta(2i\tau)](t) = \frac{\sqrt{\pi}\sinh(\sqrt{\kappa\pi t/4})}{{\sqrt {2t}}({\cosh\sqrt{4\pi t/\kappa}-\cos \chi'})}\quad \textrm{for }t>0.
		\end{align} 
	\end{lemma}
	\begin{proof}
		By the modular property $\eta(i/\tau)=\sqrt{\tau}\eta(i \tau)$ of the  eta function, we have 
		\begin{equation}\label{eq:eta-modular}
		\eta(2i \tau)=(2\tau)^{-1/2}\eta(\frac{i}{2\tau})=(2\tau)^{-1/2}{q}^{\frac1{24}}\prod_{r=1}^\infty(1- q^r)\quad \textrm{with } q=e^{-\frac{\pi}{\tau}}.  
		\end{equation}
		By the representation~\eqref{eq:ZOn-tilde} of $Z(\tau, \kappa, \chi')$, for $g=\frac{4}{\kappa}$
		we have 
		\begin{equation}\label{eq:Z-eta}
		Z(\tau, \kappa, \chi') \eta(2i\tau)= \frac{1}{\sqrt{2\tau}} \sum_{p\in \mathbb Z} \frac{\sin(p+1)\chi'}{\sin \chi'}  q^{\frac{g}{4}(p-\frac{1-g}{g})^2}
		\end{equation}
		Since $\mathcal L[\frac{1}{\sqrt\tau} e^{-\frac{  a }{4\tau}}](t)= \frac{\sqrt{\pi}e^{- \sqrt{at}}}{\sqrt t}$ for each $a>0$, we have $\mathcal L[\frac{1}{\sqrt\tau}  q^{\frac{g}{4}(p-\frac{1-g}{g})^2}](t)= \frac{\sqrt{\pi}e^{-|p+1-g^{-1}|\sqrt{g\pi t}}}{\sqrt {t}}$ hence 
		\begin{equation}\label{eq:Zeta-Lap}
		\mathcal L[ Z(\tau, \kappa, \chi') \eta(2i\tau)](t)= \sqrt{\frac\pi2} \sum_{p\in \mathbb Z} \frac{\sin(p+1)\chi'}{\sin \chi'} \frac{e^{-|p+1-g^{-1}|\sqrt{g\pi t}}}{\sqrt {t}}.
		\end{equation}
		Shifting $p+1$ and dividing the sum according to $p>0$ or $p<0$, we get
		\begin{align*}
		&\sum_{p\in \mathbb Z} \frac{\sin(p+1)\chi'}{\sin \chi'} \frac{e^{-|p+1-g^{-1}|\sqrt{g\pi t}}}{\sqrt {t}}=\sum_{p\ge 1} \frac{\sin (p\chi')}{{\sqrt {t}}\sin \chi}(e^{-(p-g^{-1})\sqrt{g\pi t}}-e^{-(p+g^{-1})\sqrt{g\pi t}})\\
		=&\sum_{p\ge 1} \frac{2\sin (p\chi')\sinh(\sqrt{\pi t/g}) }{{\sqrt {t}}\sin \chi'}e^{-p\sqrt{g\pi t}}= 
		\frac{\sinh(\sqrt{\pi t/g})}{{\sqrt {t}}({\cosh\sqrt{g\pi t}-\cos \chi'})},
		\end{align*}
		where the last step uses 
		\[
		\sum_{p\ge 1} {2\sin (p\chi')} e^{-p\sqrt{g\pi t}}=\frac{\sin \chi'}{\cosh\sqrt{g\pi t}-\cos \chi'}.
		\]
		Now~\eqref{eq:Cardy-Lap} follows from~\eqref{eq:Zeta-Lap}.
	\end{proof}
	\begin{proof}[Proof of Theorem~\ref{thm:CLE-moduli}] 
		Comparing~\eqref{eq:CLE-f} and~\eqref{eq:Cardy-Lap}, we get  
		\begin{equation} 
		\E\left[ \sum_{j=1}^{\infty} \left(\frac{n'}{n}\right)^{j-1} f(\mathrm{Mod}(\eta_j))\right] =
		\frac{\sqrt{2}C_\kappa}{\sqrt{\pi}} \mathcal L[ Z(\tau, \kappa, \chi') \eta(2i\tau)](t)\quad \textrm{for }t>0,
		\end{equation}
		where $C_\kappa=  \frac{\sqrt{\kappa \pi}(\frac{4}{\kappa}-1)\cos(\pi (\frac{4}{\kappa}-1))}
		{ 2\sin( \pi (1-\frac{\kappa}{4})) }$ and $f(\tau)=e^{- t \tau} e^{  -\frac{\kappa\pi}{4}
			(\frac{4}{\kappa}-1)^2 \tau}$. This gives~\eqref{eq:CLE-main} for this choice of $f$. Varying $t$  {we see that~\eqref{eq:CLE-main} holds if $f$ is an arbitrary non-negative measurable function.}\end{proof}

	\section{The random modulus for the Brownian annulus}\label{sec:BA}
	In this section we prove Theorem~\ref{thm:BA-mod}.
	In Section~\ref{subsec:BA-def} we give a precise  construction of the conformally embedded Brownian disk and Brownian annulus as quantum surfaces in $\sqrt{8/3}$-LQG, with more background on Brownian surfaces provided in Appendix~\ref{app:BA}. In Section~\ref{subsec:BA-LF} we use the strategy in Section~\ref{sec:QA} based on the KPZ relation (Theorem~\ref{thm:KPZ}) to complete the proof.

	\subsection{Preliminaries on Brownian surfaces}\label{subsec:BA-def}
	The \emph{Brownian sphere} is a random metric-measure space  obtained in~\cite{legall-uniqueness,miermont-brownian-map} as the Gromov-Hausdorff-Prokhorov scaling limit of uniform quadrangulations, under the name \emph{Brownian map}.
	We write $\BS_2(1)^{\#}$ as the law of the unit-area Brownian sphere with two marked points. {For $A, L > 0$, the \emph{Brownian disk} with area $A$ and boundary length $L$ is a random metric-measure space {obtained} in~\cite[Section~2.3]{bet-mier-disk} as the scaling limit of random quadrangulations with boundary. We denote its law by $\BD_{0,1}(L;A)^{\#}$. We use the subscripts $\{0,1\}$ because a sample from $\BD_{0,1}(L;A)^{\#}$ has no interior marked points and one boundary marked point. For $L>0$, define
		\begin{equation}\label{eq:BD-def-more}
		\BD_{0,1}(L)^{\#}= \int_0^\infty   \frac{L^3}{\sqrt{2\pi A^5}} e^{-\frac{L^2}{2A}} \BD_{0,1}(L;A)^{\#} \, \rd  A.
		\end{equation} Then $\BD_{0,1}(L)^{\#}$ is the probability measure for the law of the Brownian disk with boundary length $L$ and free area. See Theorems~\ref{prop-BS-convergence} and~\ref{prop-brownian-disk} in Appendix~\ref{app:BA} for precise scaling limit results concerning $\BS_2(1)^{\#}$ and $ \BD_{0,1}(L)^{\#}$. Both the Brownian sphere and the Brownian disk can be defined purely in the continuum using the Brownian snake   as done in~\cite{legall-uniqueness,miermont-brownian-map,bet-mier-disk}, but this construction is not explicitly needed for Sections~\ref{sec:BA} and~\ref{sec:SAP}.  
		We now introduce a few variants of $\BS_2(1)^{\#}$ and $\BD_{0,1}(L)^{\#}$ which are more convenient for our purpose.}
	\begin{definition}\label{def:BS}
		Given a sample of $\BS_2(1)^{\#}$, if its area measure is rescaled by some $A>0$, we write $\BS_2(A)^{\#}$ for the law of this new metric-measure space of area $A$ with two marked points. Let
		\[\BS_2=\int_0^\infty A^{-3/2}\BS_2(A)^{\#} \rd A.\]
		We call a sample from $\BS_2$ a \emph{free Brownian sphere}. 
		{For $L>0$ and $A>0$, let
			\[
			\BD_{0,1}(L)=L^{-\frac{5}{2}} \BD_{0,1}(L)^{\#} \quad \textrm{and}\quad \BD_{0,1}(L;A) =  \frac{L^{\frac12}}{\sqrt{2\pi A^5}} e^{-\frac{L^2}{2A}} \BD_{0,1}(L;A)^{\#}.
			\]
			Given a sample from $\int_0^\infty \frac{A}{L}\BD_{0,1}(L;A)\; \rd A$,  forget the marked point on the boundary and add an interior marked point according to its area measure. We denote the law of the resulting metric-measure space with one interior marked point by $\BD_{1,0}(L)$, and   the probability measure proportional to $\BD_{1,0}(L)$ by $\BD_{1,0}(L)^{\#}$.}
	\end{definition}

	The following Proposition~\ref{prop-sphere-ball} and Lemma~\ref{lem:sphere-ann}  allow us to find Brownian disks and annuli inside a Brownian sphere. {Here whenever a subset of a metric space is viewed as a new metric space, we will use the internal metric on the subset.}
	\begin{proposition}[{\cite[Theorem 3]{legall-disk-snake}}]\label{prop-sphere-ball}
		Sample $(\cS, x, y)$ from $\BS_2$.   On the event $d(x,y)>2$ let $D_y$ be the connected component of $\cS\setminus B(x,1)$ containing $y$, where $B(x,1)$ is the metric ball of radius 1 around $x$. Conditioned on $d(x,y)>2$ and on the boundary length {$\cL_y$} of $D_y$, the conditional law of $(D_y, y)$ {is given by $\BD_{1,0}(\cL_y)^{\#} $ conditioned on the event that the distance between the interior marked point and the disk boundary is greater than 1}. 
	\end{proposition}
	In the setting of Proposition~\ref{prop-sphere-ball}, we write {$\cB^{\bullet}(x,1)=\cS\setminus D_y$} and called it a filled metric ball {of radius 1 around $x$}. We define $\cB^{\bullet}(y,1)$ in the same way with the role of $x$ and $y$ swapped.
	\begin{lemma}\label{lem:sphere-ann}
		Sample $(\cS, x, y)$ from $\BS_2$. Conditioned on $d(x,y)>2$ and on the boundary lengths {$\cL_x$   of ${\cB^\bullet}(x,1)$ and $ \cL_y$} of ${\cB^\bullet}(y,1)$, the  conditional law of  $\cS \backslash ({\cB^\bullet}(x,1) \cup {\cB^\bullet}(y,1))$ is  $\BA({\cL_x,\cL_y})^\#$. 
	\end{lemma}
	\begin{proof}
		This follows from Definition~\ref{def:BA} of $\BA(a,b)^\#$and Proposition~\ref{prop-sphere-ball}.
	\end{proof}
	Lemma~\ref{lem:sphere-ann} implies the desired symmetry of the Brownian annulus needed for our proof of Theorem~\ref{thm:BA-mod}.
	\begin{lemma}\label{lem:BA-symmetry}
		As measures on metric-measure spaces, we have $\BA(a,b)^{\#}=\BA(b,a)^{\#}$ for $a,b>0$.
	\end{lemma}
	\begin{proof}
		For $(\cS, x, y)$ sampled from $\BS_2$, the law of $(\cS, y, x)$ is also  $\BS_2$. Now Lemma~\ref{lem:BA-symmetry} follows from Lemma~\ref{lem:sphere-ann}.
	\end{proof}

	Miller and Sheffield \cite{lqg-tbm1,lqg-tbm2,lqg-tbm3} gave the $\sqrt{8/3}$-LQG description of the conformally embedded Brownian sphere and the Brownian disk. In this section we only need their disk result, which we recall now. The sphere result will be recalled as Theorem~\ref{thm:bs-lqg} in Section~\ref{sec:SAP}.
	Fix $\gamma=\sqrt{8/3}$. For $a>0$, recall that the measure $\LF^{(\gamma,i)}_\H(a)$ is defined by the disintegration $\LF^{(\gamma,i)}_\H=\int_0^\infty \LF^{(\gamma,i)}_\H(a) \rd a$ such that samples from $\LF_\H^{(\gamma, i)}(a)$ have boundary length $a$. By~\eqref{eq:LFl-explicit} one has  $|\LF^{(\gamma,i)}_\H(a)|<\infty$. Let $\LF^{(\gamma,i)}_\H(a)^{\#}$ be the probability measure proportional to  $\LF^{(\gamma,i)}_\H(a)$.
	Now sample $\phi$ from $\LF^{\gamma,i}_\H(a)^{\#}$. Let $\mu_\phi=e^{\gamma \phi(z)}dz$ be the quantum area measure ~\cite{shef-kpz} defined via the Gaussian multiplicative chaos.  By~\cite{DDDF-tightness,gm-uniqueness}, we can find a smooth regularization  $\phi_\eps$ and a normalizing sequence $z_\eps$ such that as $\eps\to 0$ the Riemannian metric  on $\H$ induced by metric tensor $z^{-1}_\eps e^{\frac{1}{4}\cdot2\gamma\phi_\eps} (\rd x^2+\rd y^2)$ converges in probability to a metric $d_\phi$ on $\H$, which is called the $\sqrt{8/3}$-LQG metric of $\phi$. Moreover, modulo a multiplicative constant, $d_\phi$ agrees with the metric constructed earlier in~\cite{lqg-tbm1,lqg-tbm2} via the quantum Loewner evolution. Here the exponent $\frac14$ is the Hausdorff dimension of a Brownian surface.  The following theorem summarizes the relation between $\LF^{(\gamma,i)}_\H$ and the Brownian disk. 
	\begin{theorem}\label{thm:bm-lqg}
		There exist constants  $c_1$ and $c_2$ such that the following holds.\footnote{The metric constant $c_1$ can be set to $1$ by choosing the  normalizing sequence in the definition of $d_\phi$ appropriately. The area constant $c_2$ is canonically defined due to the different definitions of area for Brownian surfaces and $\sqrt{8/3}$-LQG surfaces. As explained in~\cite[Remark 3.12]{ghs-metric-peano}, $c_2=\sqrt3$.} Suppose   $\phi$ is sampled  from $\LF^{\gamma,i}_\H(a)^{\#}$ for some $a>0$. Let $(\cD,p, d,\mu)$ be the (marked) metric-measure space given by  $(\H,i, c_1 d_\phi,c_2\mu_\phi)$. Then $(\cD,p, d,\mu)$ is
		a Brownian disk with boundary length $a$ and an interior marked point; {namely its law is $\BD_{1,0}(a)^{\#}$. Moreover, the notion of boundary length for the Brownian disk agrees with the $\sqrt{8/3}$-LQG length.
			Finally,}  $(\cD,p, d,\mu)$  and the quantum surface $(\H, \phi,i)/{\sim_\gamma}$ are measurable with respect to each other. 
	\end{theorem}
	\begin{proof}
		By~\cite[Theorem 3.4]{ARS-FZZ},   $(\H, \phi,i)/{\sim_\gamma}$  is a quantum disk with boundary length $a$ and an interior  marked point as defined in the mating-of-trees framework~\cite{wedges}. {By~\cite[{Corollary 1.5}]{lqg-tbm2} together with~\cite[Section 6]{legall-disk-snake}, it is possible to require the two notions of boundary lengths agree and find scaling constants} $c_1$ and $c_2$ such that {the law of $(\cD,p, d,\mu)$ is $\BD_{1,0}(a)^{\#}$}.
		The measurability result was  first proved in~\cite[{Theorem 1.4}]{lqg-tbm3} and a more constructive proof was given in~\cite{gms-poisson-voronoi}. 
	\end{proof} 
	\begin{remark}\label{rmk:BD-nonprob}
		{By~\eqref{eq:LFl-explicit} with $\gamma=\sqrt{8/3}$, we have 
			\[|\LF^{\gamma,i}_\H(a)|=Ca^{-\frac{4}{\gamma^2}}=Ca^{-\frac32} =C|\BD_{1,0}(a)|\quad \textrm{for some constant $C>0$}.\]
			Therefore Theorem~\ref{thm:bm-lqg} still holds with $\LF^{\gamma,i}_\H(a)$ and $C\cdot \BD_{1,0}(a)$ in place of  $\LF^{\gamma,i}_\H(a)^{\#}$ and 
			$\BD_{1,0}(a)^{\#}$.}
	\end{remark}

	Let $d_\H=c_1d_\phi$ and $\mu_\H=c_2 \mu_\phi$. 
	Then $(\H,d_\H,\mu_\H)$ is a  conformal  embedding of $(\cD,p, d,\mu)$ as discussed below Definition~\ref{def:BA} of the Brownian annulus. On the event that $d(p,\partial \cD)>1$,   let $\cA$ be the (annular) connected component of $\cD \setminus  B_d(p,1)$ whose boundary contains $\partial \cD$. {Let $B^{\bullet}(p,1)=\cD\setminus\cA$, which we call the filled metric ball of radius $1$ around $p$.}
	Let $B_{d_\H}(i,1)=\{x\in \H: d_\H(i,x)\le 1  \}$ and $A$ the annular component of $\H\setminus B_{d_\H}(0,1)$. Then $(A,d_\H,\mu_\H)$ is a conformal embedding of  $(\cA,d,\mu)$. Recall that   $\mathrm{Mod}(\cA)$  is the modulus of the planar  annulus $A$.
	\begin{lemma}\label{lem:Mod-det}
		The metric-measure space $(\cA,d,\mu)$  and quantum surface $(A,\phi)/{\sim_\gamma}$ are measurable with respect to each other. In particular, $\mathrm{Mod}(\cA)$ is measurable  with respect to  $(\cA,d,\mu)$.
	\end{lemma}
	Lemma~\ref{lem:Mod-det} follows from the measurability  results in~\cite{lqg-tbm3,gms-poisson-voronoi} for the Brownian sphere and disk. See Appendix~\ref{app:BA} for a proof.
	Our next lemma specifies the law of ${\cL}$ in Definition~\ref{def:BA} of $\BA(a,b)^{\#}$. It follows from several known {facts}  on the enumerations and scaling limits of quadrangulations. We also give its proof in Appendix~\ref{app:BA}.
	{To draw the analog with Section~\ref{sec:QA} we introduce the function 
		\begin{equation}\label{eq:ball-pfn}
		|\mathrm{Ball}_1(\ell)| =e^{-\frac{9\ell}{2}} \quad \textrm{for }\ell>0.
		\end{equation}
		In light of Proposition~\ref{prop-ball-count}, $|\mathrm{Ball}_1(\ell)|$ can be thought of as the partition function of the filled metric ball with radius $1$ and perimeter $\ell$. } The following is analogous to Theorem~\ref{thm:QA-welding}.

	\begin{lemma}\label{lem:ball-length}
		Fix $a>0$. Let $(\cD,p, d,\mu)$ be a Brownian disk {sampled from $\BD_{1,0}(a)^{\#}$.}
		On the event that $d(p,\partial \cD)>1$,   let $\cA$ be the annulus in Definition~\ref{def:BA} and let  ${\cL}$ be the length of  $\partial \cA\setminus\partial \cD$, which is also the boundary length of the filled metric ball ${\cB^\bullet}(p,1)$. Then the law of ${\cL}$ conditioned on $d(p,\partial \cD)>1$ is  \begin{equation}\label{eq:ball-length}
		\frac{1}{Z(a)}\frac{1_{b>0} a^{3/2}
			{b|\mathrm{Ball}_1(b)|}
			\rd b}{\sqrt{ab}(a+b)}\quad \textrm{where }  Z(a)=\int_0^\infty   \frac{a^{3/2}{b|\mathrm{Ball}_1(b)|}\rd b}{\sqrt{ab}(a+b)}.
		\end{equation}
		{Moreover, $\BD_{1,0}(a)^{\#}[d(p,\partial \cD)>1]=Z(a)/Z(\infty)$ where $Z(\infty)=\int_0^\infty b^{\frac12}|\mathrm{Ball}_1(b)| \rd b$. Finally, conditioning on ${\cB^\bullet}(p,1)$, the conditional law  of $\cA$ only depends on ${\cL}$.}
	\end{lemma}
	
	{We record the following Brownian sphere analog of Lemma~\ref{lem:ball-length}, which we need in Section~\ref{sec:SAP}. See Appendix~\ref{app:BA} for the proof.
		\begin{lemma}\label{lem:sphere-ball}
			For  $(\cS, x, y)$ sampled from $\BS_2$, let  $\cL_x$  and $ \cL_y$  be the boundary lengths of $ \cB^\bullet(x,1)$ and $\cB^\bullet(y,1)$, respectively. Then under the restriction of $\BS_2$ to the event $d(x,y)>2$, the joint law of   $(\cL_x,\cL_y)$ has density
			\[C\cdot 1_{a>0;b>0} \frac{\sqrt{ab}|\mathrm{Ball}_1(a)||\mathrm{Ball}_1(b)|}{a+b}   \,\rd a \, \rd b\qquad \textrm{for some constant } C\in(0,\infty).\]
		\end{lemma}
	}

	\subsection{Liouville field description and the law of the modulus}\label{subsec:BA-LF}
	Let \begin{equation}\label{eq:ba-partition}
	\BA(a,b)=\frac{\BA(a,b)^{\#}}{2\sqrt{ab}(a+b)} \quad \textrm{and}\quad \BA=\iint_0^\infty\BA (a,b)\,\rd a \, \rd b.
	\end{equation}
	In light of Lemma~\ref{lem:Mod-det}, we can view the measures $\BA$ and $\BA(a,b)$ as laws of quantum surfaces.  
	The proof of Theorem~\ref{thm-BA-main} is parallel to that of Theorem~\ref{thm:QAconj}. 
	We first give the counterpart of Theorem~\ref{thm:QA-welding}. 
	{This is essentially a reformation of Lemma~\ref{lem:ball-length} where we describe Brownian disks in terms of Liouville fields.}
	\begin{proposition}\label{prop:BA-welding}
		Set $\gamma=\sqrt{8/3}$. Sample $\phi$ from {$\LF_\H^{\gamma,i}$. On the event $E=\{d_\H(i, \partial \H) > 1\}$,} let  $A_\eta$ be the annular connected component of $\H\setminus B_{d_\H} (i,1)$. Let $\eta$ be the loop corresponding to $\partial A_\eta\setminus \partial\H$ and $\cL^\gamma_\phi (\eta)$ be its quantum length.
		Let $D_\eta$ be the Jordan domain bounded by $\eta$. {Let $f(b)=b|\mathrm{Ball}_1(b)|$. 
			Then under  the measure  $M= \frac{1_E}{f(\mathcal L^\gamma_\phi(\eta) )}\cdot\LF^{\gamma,i}_\H$,  the law of $(A_\eta, \phi)/{\sim_\gamma}$ is $C\cdot \BA$ for some constant $C>0$. Moreover, conditioning on $(D_\eta,\phi)/{\sim_\gamma}$, the conditional law  of $(A_\eta,\phi)/{\sim_\gamma}$ only depends on $\cL^\gamma_\phi (\eta)$.}
	\end{proposition}
	
	\begin{proof}
		In this proof $C$ is a positive constant that could change from line to line.
		By Theorem~\ref{thm:bm-lqg} and Remark~\ref{rmk:BD-nonprob}, {under $\LF_{\mathbb H}^{\gamma, i}$} the law of $(\H,\phi,i)/{\sim_\gamma}$ is $C \cdot  \BD_{1,0}$.
		Note that 
		\[  \frac{a^{3/2}
			b|\mathrm{Ball}_1(b)|
		}{\sqrt{ab}(a+b)} = C\frac{|\BA(a,b)|b|\mathrm{Ball}_1(b)|}{|\BD_{1,0}(a)| }.\]
		By Lemma~\ref{lem:ball-length}, under $M$ the joint law of $L_\phi^\gamma(\partial \H)$ and $L_\phi^\gamma(\eta)$ is  
		$C1_{a>0,b>0}\cdot |\BA(a,b)|\, \rd a \rd b$. 
		By the definition of $\BA(a,b)^{\#}$, under $M$ the law of $(A_\eta, \phi)/{\sim_\gamma}$ is $C1_{a>0,b>0}\cdot |\BA(a,b)|\BA(a,b)^{\#}=C\cdot \BA$ as desired. Moreover, the last assertion of Proposition~\ref{prop:BA-welding} follows from the last assertion of Lemma~\ref{lem:ball-length}.
	\end{proof}

	\begin{lemma}\label{lem-BA-m}
		Set $\gamma=\sqrt{8/3}$.
		There exists a measure $m(\rd \tau)$ on $(0,\infty)$ such that if
		we sample $(\tau, \phi)$ from $ \LF_\tau (d\phi) m(\rd \tau)$, then the law of the quantum surface $(\cC_\tau,\phi)/{\sim_\gamma}$ is $\BA$. 
	\end{lemma}
	\begin{proof}
		The argument is almost identical to that of Proposition~\ref{prop:LF} so we only point out the differences.  
		Suppose we are in the setting of Proposition~\ref{prop:BA-welding} where $M$ is the reweighted measure.  Let $\hat\tau$ be such that $A_\eta$ is conformally equivalent to $\cC_{\hat \tau}$, and let $\psi: \cC_{\hat\tau} \to A_\eta$  be the conformal map such that $\psi(0)=0$. Let $\hat \phi(z)=\phi \circ \psi(z) +Q\log |\psi(z)'|$ for $z\in \cC_{\hat\tau}$. Let $\hat\phi^{\mathrm{har}}$ be the harmonic extension of $\hat \phi|_{\cC(\frac{2\hat \tau}{3}, \hat \tau)}$ onto $\cC(0, \frac{2\hat \tau}{3})$ with zero normal derivative on $\{0\}\times  [0,1]$.  By Proposition~\ref{prop:BA-welding}, under $M$,  the law of $(A_\eta, \phi)/{\sim_\gamma}$ is  $C\cdot \BA$. Therefore,  under   $M$  the law of $(\hat \phi,\hat \tau)$ is a rotational invariant embedding of $C\cdot\BA$.
		Now the statement of Lemma~\ref{lem:DMP2} still holds for $(\phi,\eta)$ under $M$ with the same argument except the following modification: when we apply the domain Markov property from Lemma~\ref{lem:local}, we are in the second scenario of that lemma since the metric $d_\H$ is determined by the field $\phi$ locally. This means that conditioning on $\hat \tau$ and $\hat \phi|_{\cC(\frac{2\hat \tau}{3}, \hat \tau)}$, the conditional law of 
		$\hat \phi   - \hat \phi^{\mathrm{har}}$ is a GFF on $\cC(0, \frac{2\hat \tau}{3})$  with the appropriate  boundary condition.
		
		Another key ingredient for the proof  of Proposition~\ref{prop:LF} is the symmetry of $\QA$ under flipping the sides, namely Proposition~\ref{prop:QA-symmetry}. 
		By Lemma~\ref{lem:BA-symmetry}, we have $\BA(a,b)^{\#}=\BA(b,a)^{\#}$  
		{which gives the desired symmetry for $\BA$.}   The rest of the argument {based on Proposition~\ref{prop:inv-unique}} is identical to that of Proposition~\ref{prop:LF}.
	\end{proof}

	\begin{theorem}\label{thm-BA-main}
		The measure  $m(\rd \tau )$ in Lemma~\ref{lem-BA-m} equals $1_{\tau>0} 2^{-1/2}\eta(2i\tau) \,\rd \tau$.
	\end{theorem}
	
	\begin{proof}
		The argument is essentially the same as that of Theorem~\ref{thm:QAconj}. {By~\eqref{eq:ba-partition}
			\begin{align*}
			\BA[L_1 e^{-L_1} L_0^{ix} ]=\iint_0^\infty ae^{-a} b^{ix} |\BA(a,b)|\,\rd a\,\rd b=\frac{1}2 \iint_0^\infty \frac {ae^{-a} b^{ix}}{\sqrt{ab} (a+b)}\,\rd a\,\rd b.
			\end{align*} 
			By~\eqref{eq:integral}  we have $\BA[L_1 e^{-L_1} L_0^{ix} ]=\frac{\pi\Gamma(1+ix)}{2\cosh(\pi x)}$.
			Now the KPZ relation (Theorem~\ref{thm:KPZ}) yields that}
		\[ \int_0^\infty e^{-\pi \gamma^2 x^2\tau /4} m(\rd \tau) = \frac{2\sinh(\frac{\gamma^2}4 \pi x)}{\pi \gamma x \Gamma(1+ix)} \BA[L_1e^{-L_1} L_0^{ix}] 
		= \frac{\sinh(\frac{\gamma^2}4 \pi x)}{\gamma x \cosh(\pi x)} \quad \text{for }x \in \R.
		\]
		By \eqref{eq:eta-L} in  Appendix~\ref{app:theta}, we have
		\[
		\int_0^{\infty} 
		e^{ -\frac{\pi\gamma^2 x^2\tau}{4}} \eta(2i \tau)   \,\rd \tau= \frac{\sqrt3}2 \frac{\sinh(\frac{2\pi x}3)}{\cosh(\pi x)}
		\quad \textrm{for }x\in \R.
		\]
		Since $\gamma = \sqrt{8/3}$, {comparing the last two equations} we get $m(\rd\tau)=1_{\tau>0} 2^{-1/2}\eta(2i\tau) \rd\tau$.
	\end{proof}
	
	\begin{proof}[Proof of Theorems~\ref{thm:BA-mod} and~\ref{thm:BAconj}]
		By Theorem~\ref{thm-BA-main} we get Theorem~\ref{thm:BAconj}. 
		{For  Theorem~\ref{thm:BA-mod}, note that $X_\tau$ agrees  in law with $\frac{\cL_h^\gamma(\partial_1\cC_\tau)}{\cL_h^\gamma(\partial_0\cC_\tau)}$ in Theorem~\ref{thm:GMC-ratio} with $\gamma=\sqrt{\frac83}$. By Lemma~\ref{lem:GMC}, when $\gamma=\sqrt{\frac83}$ we have
			\[\LF_\tau[f(\mathcal L^\gamma_\phi (\partial_0\cC_\tau)) g(\mathcal L^\gamma_\phi (\partial_1\cC_\tau))]=  \frac{2}{\gamma}\iint_0^\infty \ell^{-1}f(\ell)  g (\ell r) \rho_\tau(r)  \,\rd r\,\rd\ell= \frac{2}{\gamma}\iint_0^\infty a^{-2}f(a)  g (b) \rho_\tau(\frac{b}a)  \,\rd b \, \rd a.\] 
			Therefore the joint density of the two boundary lengths $(L_0,L_1)$ under $\LF_\tau$ is $ \frac{2}{\gamma}a^{-2} \rho_\tau(\frac{b}a) \,\rd a \,\rd b$. Now by Theorem~\ref{thm:BAconj}, the law of $\mathrm{Mod}(
			\cA)$ under $\BA(a,b)^{\#}$ is proportional to $1_{\tau>0}\rho_\tau(\frac{b}a) \eta(2i\tau) \,\rd \tau$.} 
	\end{proof}

	{
		\begin{remark}[$\BA$ as the limit of $\QA$]
			Note that $\lim_{\gamma\to \sqrt{8/3}} \frac{\pi}{2\cos(\pi(\frac{4}{\gamma^2}-1))}|\QA^\gamma(a,b)|=|\BA(a,b)|$ for $\QA^\gamma$ from Section~\ref{subsec:QA-LF}.
			In fact Theorems~\ref{thm:BAconj} and~\ref{thm:QAconj} show that $\frac{\pi}{2\cos(\pi(\frac{4}{\gamma^2}-1))}\QA^\gamma(a,b)$ converges to $\BA(a,b)$ at the quantum surface level, since both the law of the modulus and the field under the $\cC_\tau$ embedding converge.
		\end{remark}
	}

	{
		\section{The annulus partition function of the SLE$_{8/3}$ loop}\label{sec:SAP}
		In this section we prove  Theorem~\ref{thm-werner-partition}.
		In Section~\ref{subsec:loop-zipper} we recall the coupling between the SLE$_{8/3}$ loop and the conformally embedded Brownian sphere. In Section~\ref{subsec:SAP-proof}  we use this coupling and  the KPZ relation from Theorem~\ref{thm:KPZ} to complete the proof.

		\subsection{The SLE$_{8/3}$ loop zipper on the Brownian sphere}\label{subsec:loop-zipper}
		We first recall the $\sqrt{8/3}$-LQG construction of Miller and Sheffield \cite{lqg-tbm1,lqg-tbm2,lqg-tbm3} for the conformally embedded Brownian sphere, and  Zhan's construction of the $\SLE_{8/3}$ loop measure~\cite{zhan-loop-measures}. These two constructions (i.e.\ Theorem~\ref{thm:bs-lqg} and Definition~\ref{def:Zhan}) are included mainly for concreteness. The only input we need is the {conformal welding} result {from~\cite{AHS-loop}} for their coupling, which we recall as Theorem~\ref{thm:loop-zipper}.
		\begin{theorem}[\cite{lqg-tbm1,lqg-tbm2,lqg-tbm3}]\label{thm:bs-lqg}
			Fix $\gamma=\sqrt{8/3}$ and $Q=\frac{2}{\gamma}+\frac{\gamma}{2}$. Let $(B_s)_{s \geq 0}$ be a standard Brownian motion  conditioned on $B_{s} - (Q-\gamma)s<0$ for all $s>0$, and let $(\wt B_s)_{s \geq 0}$ be an independent copy of $(B_s)_{s \geq 0}$. Let 
			\[Y_t =
			\left\{
			\begin{array}{ll}
			B_{t} - (Q -\gamma)t  & \mbox{if } t \geq 0 \\
			\wt B_{-t} +(Q-\gamma) t & \mbox{if } t < 0
			\end{array}
			\right. .\] 
			Let $\cC=\R \times [0,2\pi]/\sim$ be a horizontal cylinder. Let $h^1(z) = Y_{\mathrm{Re} z}$ for each $z \in \cC$.
			Let $h^2_\cC$ be independent of $h^1$  and have the law of the lateral component of the Gaussian free field on $\cC$. Let $\hat h=h^1+h^2_\cC$.
			Let  $\mathbf c\in \R$ be sampled from $ \frac\gamma2 e^{2(\gamma-Q)c}dc$ independently of $\hat h$ and set $h=\hat h+\mathbf c$. Let $(\cS,x, y)$ be the (marked) metric-measure space given by  $(\cC,\infty,-\infty, c_1 d_h,c_2\mu_h)$, with the constants $c_1$ and $c_2$ from Theorem~\ref{thm:bm-lqg}. 
			Then there exists a constant $C\in (0,\infty)$ such that the law of $(\cS,x, y)$ is $C\cdot \BS_2$. Moreover,  $(\cS,x, y)$ and the quantum surface $(\cC, h,+\infty,-\infty)/{\sim_\gamma}$ are measurable with respect to each other. 
		\end{theorem}
		
		The quantum surface $(\cC, h,+\infty,-\infty)/{\sim_\gamma}$ is the \emph{two-pointed quantum sphere} defined in~\cite{wedges}. For more details of its definition such as how to take the zero-probability conditioning when defining $Y$, and the definition of the lateral component of the GFF on $\cC$, we refer to~\cite[Section 4]{wedges}. See also the preliminary sections of~\cite{ahs-disk-welding,AHS-SLE-integrability}.
		For Zhan's construction of the $\SLE_{8/3}$ loop measure, we first recall that for $\kappa\in(0,8)$ and two distinct points $p,q\in\C$, the \emph{two-sided whole plane SLE$_\kappa$} is the probability measure on pairs of curves $(\eta_1,\eta_2)$ on $\C$ connecting $p$ and $q$ where $\eta_1$ is a so-called whole-plane $\SLE_\kappa(2)$ from $p$ to $q$, and conditioning on $\eta_1$, the curve $\eta_2$ is a chordal $\SLE_\kappa$ on  the complement of $\eta_1$. 
		We can view $\SLE^{p \rightleftharpoons q}_\kappa$ as a measure on oriented loops by reversing the direction of $\eta_2$. Given a loop $\eta$ sampled from $\SLE^{p \rightleftharpoons q}_\kappa$, 
		with probability 1  its $d_\kappa$-dimensional Minkowski content  $\mathrm{Cont}(\eta)$ exists~\cite{lawler-rezai-nat}, where $d_\kappa := 1+\frac{\kappa}{8}$ .
		\begin{definition}\label{def:Zhan}
			Fix $\kappa=\frac{8}3$. Let $\SLE^{p \rightleftharpoons q}_\kappa$ be two-sided whole plane SLE$_\kappa$ between $p,q$. Zhan's $\SLE_{8/3}$ loop measure is the infinite measure on unmarked oriented loops defined by 
			\begin{equation}\label{eq:zhan-loop}
			\mu_\C(\rd\eta)= \mathrm{Cont}(\eta)^{-2} \iint_{\C\times \C}  |p-q|^{-2(2-d_\kappa)} \SLE^{p \rightleftharpoons q}_\kappa(\rd\eta) \,\rd p\, \rd q.
			\end{equation}
			For $\tau>0$,  let $Z_{8/3}(\tau) = \mu_{\C}[E_\tau]$ where $E_\tau$ is the set of non-contractible loops in $\bbA_\tau=\{ e^{-2\pi \tau}<|z|<   1\}$. 
		\end{definition}
		Zhan's construction works for all $\kappa\in(0,8)$ and it was shown in~\cite{zhan-loop-measures} that modulo a multiplicative constant, the $\kappa=\frac{8}{3}$ case as in Definition~\ref{def:Zhan} gives the loop measure defined by Werner~\cite{werner-loops}; see~\cite[Remark 5.3]{zhan-loop-measures}. In particular, it has the following conformal restriction property.
		\begin{theorem}[\cite{werner-loops,zhan-loop-measures}]\label{thm:conf-res}
			For each domain $\Omega\subset \C$, let $\mu_\Omega$ be the restriction of $\mu_\C$ to loops in $\Omega$.   Then for any conformal map 
			$f: \Omega \to \Omega'$ between two domains,  
			we have $f\circ \mu_\Omega = \mu_{\Omega'}$. Namely, for $\eta$ sampled from $\mu_\Omega$, the law of $f\circ \eta$ is  $\mu_{\Omega'}$.
		\end{theorem}

		We now recall a  {conformal welding result} for $\mu_\C$ {proved in~\cite{AHS-loop}}. This particular variant was proved as \cite[Proposition 6.5]{acsw-loop}
		based on~\cite{AHS-loop}. 
		\begin{theorem}\label{thm:loop-zipper}
			Fix $\gamma=\sqrt{8/3}$. Let $\mathbb F_\C$ be the law of the field  $\phi = h \circ \log - Q \log |\cdot|$ on $\C$ where $h$ as in Theorem~\ref{thm:bs-lqg} and  $\log$ is the conformal map from $\C$ to $\cC$.
			Let $\hat\mu_\C$ be the restriction of $\mu_\C$ to the set of loops separating $0$ and $\infty$. Now 
			sample $(\phi,\eta)$ from $\mathbb F_\C\times \hat\mu_\C$. Let $D_0$ and $D_\infty$ be the two connected components of $\C\setminus \eta$  such that $0\in D_0$.  
			Then the joint law of $(D_0, \phi, 0)/{\sim_\gamma}$ and $(D_\infty, \phi,\infty)/{\sim_\gamma}$ is
			\[C \int_0^\infty \ell \BD_{1,0}(\ell) \times  \BD_{1,0}(\ell)\, \rd \ell \quad \text{ for some constant } C \in (0,\infty).\]
			Here $\BD_{1,0}(\ell)$ is identified as a law of quantum surface in $\sqrt{8/3}$-LQG using Theorem~\ref{thm:bm-lqg}.
		\end{theorem}

		\subsection{The annulus partition function from the KPZ relation}\label{subsec:SAP-proof}
		The {idea behind  our proof of Theorem~\ref{thm-werner-partition} is that if we glue together two  samples of $\BA$, this gives a sample of $\BA$ decorated by a non-contractible $\SLE_{8/3}$ loop, where the law of the modulus of the annulus is reweighted by the partition function $Z_{8/3}$. To make it rigorous, we realize the Brownian annulus as a subset of the Brownian sphere and apply~Theorem~\ref{thm:loop-zipper}.  
			\begin{lemma}\label{lem:BA-sphere}
				Let $\phi$ be a sample from $\mathbb F_\C$ in Theorem~\ref{thm:loop-zipper}. Let $E$ be the event that $d(0,\infty)>2$, where $d=c_1d_\phi$ is the Brownian metric. On the event $E$, let $B^\bullet(0,1)$ and $B^\bullet(\infty,1)$ be the filled metric ball of radius 1 around $0$ and $\infty$, respectively. Let $\cL_1$ and $\cL_\infty$ be the boundary lengths of $B^\bullet(0,1)$ and $B^\bullet(\infty,1)$, respectively. 
				Let $A=\C\setminus (B^\bullet(0,1))\cup B^\bullet(\infty,1))$. 
				Let  $f(b)=b|\mathrm{Ball}_1(b)|$ as in Proposition~\ref{prop:BA-welding}.  
				Let $M$ be $\frac{1}{f(\cL_0)f(\cL_\infty)} \mathbb F$ restricted to $E$. Then under $M$ the law of $(A,\phi)/{\sim_\gamma}$ is $C\cdot \BA$ for some constant $C>0$.
			\end{lemma} 
			\begin{proof}
				The result follows from Lemmas~\ref{lem:sphere-ann} and~\ref{lem:sphere-ball} via the same argument as in Proposition~\ref{prop:BA-welding}.
			\end{proof}

			Recall from  Lemma~\ref{lem-BA-m} and  Theorem~\ref{thm-BA-main} that  if
			we sample $(\phi,\tau)$ from $ \LF_\tau (\rd\phi) m(\rd \tau)$ with 
			\begin{equation}\label{eq:BA-mod-recall}
			m(\rd \tau )=1_{\tau>0} 2^{-1/2}\eta(2i\tau) \,\rd\tau,
			\end{equation} then the law of the quantum surface $(\cC_\tau,\phi)/{\sim_\gamma}$ is $\BA$. 
			\begin{lemma}\label{lem:reweight}
				Let $(\phi,\eta)$ be a sample from $\mathbb F_\C\times \hat \mu_\C$ in Theorem~\ref{thm:loop-zipper}. Let $E=\{d(0,\infty)>2\}$,  $f(\ell)=\ell|\mathrm{Ball}_1(\ell)|$, and $A=\C\setminus (B^\bullet(0,1))\cup B^\bullet(\infty,1))$ as in Lemma~\ref{lem:BA-sphere}. Let $F$ be the event that $\eta \subset A$ and $\eta$ is non-contractible within $A$. Let $\hat M$ be the measure $\frac{1}{f(\cL_0)f(\cL_\infty)} \mathbb F\times \hat\mu_\C$ restricted to $E\cap F$. Then under $\hat M$ the law of  $(A,\phi)/{\sim_\gamma}$ is that of $(\cC_\tau,\phi)/{\sim_\gamma}$ with $(\phi,\tau)$ sampled from $C\LF_\tau(\rd \phi) Z_{8/3}(\tau) m(\rd \tau)$  for some constant $C>0$.
			\end{lemma} 
			\begin{proof}
				By the conformal  restriction property Theorem~\ref{thm:conf-res}, given $A$, the $\hat \mu_\C$-mass of the event $F$ is $Z_{8/3}(\mathrm{Mod}(A))$, where $\mathrm{Mod}(A)$ is the modulus of $A$. By Lemma~\ref{lem:BA-sphere}, we obtain the desired marginal law of $(A,\phi)/{\sim_\gamma}$ under $\hat M$ after  integrating over $\eta$.  
			\end{proof}
			Similar to $|\QA^\gamma_2(a,b)|$ in~\eqref{eq:def-QAn}, for $a>0$ and $b>0$ we let \begin{equation}
			|\BA_2(a,b)|=\int_0^\infty |\BA(a,\ell)|\ell |\BA(\ell,b)|\,\rd \ell.
			\end{equation}
			Now the following lemma together with Lemma~\ref{lem:reweight} will allow us to apply the KPZ relation in Theorem~\ref{thm:KPZ}. 
			\begin{lemma}\label{lem:SAP-length}
				In the setting of Lemma~\ref{lem:reweight},  the joint law of $(\cL_0,\cL_\infty)$ under $\hat M$ has density  
				\[C\cdot 1_{a>0;b>0} |\BA_2(a,b)| \,\rd a \, \rd b\qquad \textrm{for some constant } C\in(0,\infty).\]
			\end{lemma} 
			\begin{proof} By Theorem~\ref{thm:loop-zipper} and Lemma~\ref{lem:ball-length}, under   $\mathbb F\times \hat\mu_\C$ restricted to $E\cap F$ the joint law of $(\cL_0,\cL_\infty)$ is 
				\[C\cdot 1_{a>0;b>0} |\mathrm{Ball}_1(a)|a|\BA_2(a,b)| b|\mathrm{Ball}_1(b)|  \,\rd a\, \rd b\qquad \textrm{for some constant } C\in(0,\infty).\]
				Since $f(\ell)=\ell|\mathrm{Ball}_1(\ell)|$, we get the desired law of $(\cL_0,\cL_\infty)$ under $\hat M$.
			\end{proof} 
			
			\begin{proof}[Proof of  Theorem~\ref{thm-werner-partition}.]
				By Theorem~\ref{thm:KPZ}  and Lemmas~\ref{lem:reweight}--~\ref{lem:SAP-length}, we have
				\begin{align}\label{eq:BA-simple} 
				\int_0^\infty e^{ -\frac{\pi\gamma^2 x^2\tau}{4}} Z_{8/3}(\tau)m(\rd \tau )=
				\frac{2\sinh(\frac{\gamma^2}{4} \pi x  )}{\pi\gamma x\Gamma(1+ix)} \iint_0^\infty ae^{-a}b^{ix} |\BA_2(a,b)| \,\rd a \, \rd b \quad \textrm{for }x\in \R,
				\end{align} 
				where $\gamma=\sqrt{8/3}$ and $m(\rd \tau )=1_{\tau>0} 2^{-1/2}\eta(2i\tau) \,\rd \tau$ is as  in~\eqref{eq:BA-mod-recall}. 
				By the same calculation in the proof of Lemma~\ref{lem:QA-Mellin-k}  we get 
				\[
				\iint_0^\infty ae^{-a}b^{ix} |\BA_2(a,b)| \rd a\rd b  = \frac{C_1}{\cosh(\pi x)} \BA[L_1 e^{-L_1} L_0^{ix}]  \quad \textrm{for some constant }C_1>0.
				\]
				Recall from the proof of Theorem~\ref{thm-BA-main} that   $\BA[L_1 e^{-L_1} L_0^{ix} ]=\frac{\pi\Gamma(1+ix)}{2\cosh(\pi x)}$. Therefore 
				\begin{equation}\label{eq:lap-sap}
				\int_0^\infty e^{- \frac{2\pi x^2}3 \tau} \eta(2i\tau) Z_{8/3} (\tau) \, \rd\tau = C_2 \frac{\sinh(\frac{2\pi x}3)}{x \cosh(\pi x)^2} \qquad \text{ for }x \in \R \textrm{ and }C_2=\gamma^{-1}C_1.
				\end{equation} 
				Now we take the inverse Laplace transform to compute $Z_{\mathrm{8/3}}(\tau)$. Let $\cL$ and $\cL^{-1}$ denote the Laplace transform and its inverse. Let $F(s) := \frac{\sinh(\alpha \sqrt s)}{\sqrt s\cosh (\sqrt s)^2}$ with $\alpha = \frac23$.	Since $\frac{1}{\cosh (u)^2}= \sum_{n=1}^{\infty} (-1)^{n-1} 4n  e^{-2n u}$, we have
				\[
				F(s)= \frac{e^{\alpha \sqrt s}-e^{-\alpha \sqrt s}}{2\sqrt s \cosh (\sqrt s)^2}= 2 \sum_{n=1}^{\infty} (-1)^{n-1} n \frac{e^{-(2n-\alpha)\sqrt s}-e^{-(2n+\alpha)\sqrt s}}{\sqrt s}.
				\]
				Since $ \mathcal L^{-1}[\frac{e^{-a\sqrt s}}{\sqrt s}](t) = \frac{e^{-\frac{a^2}{4t}}}{\sqrt{\pi t}}$
				for $a\ge 0$, we have 
				\[
				\mathcal L^{-1}[F(s)](t) = 2 \sum_{n=1}^{\infty} (-1)^{n-1} n \frac{e^{-\frac{(2n-\alpha)^2}{4t}}-e^{-\frac{(2n+\alpha)^2}{4t}}}{\sqrt{\pi t}}= \frac{2}{\sqrt{\pi t}} \sum_{k\in \mathbb Z}^{\infty} k(-1)^{k-1}    {e^{-\frac1t(k-\frac13)^2}}.\] 
				By~\eqref{eq:lap-sap} we have $\cL[ \eta(2i\tau) Z_{8/3}(\tau)] (\frac {2s}{3\pi}) = C_2 F(s)$. Therefore
				\begin{equation}\label{eq-SAW-0}
				\eta(2i\tau) Z_{8/3}(\tau) = \frac{2C_2}{3\pi} \cL^{-1}[F(s)] (\frac{2\tau}{3\pi}) = \frac{4C_2}{\pi \sqrt3}  \cdot  (2\tau)^{-1/2}\sum_{k \in \mathbb Z} k (-1)^{k-1} e^{-\frac{3\pi}{2\tau} (k-\frac13)^2 } .
				\end{equation}
				On the other hand, recall from~\eqref{eq:eta-modular} that $\eta(2i \tau)=(2\tau)^{-1/2} {q}^{\frac1{24}}\prod_{r=1}^\infty(1- q^r)$ with $q=e^{-\pi/\tau}$.
				Therefore 
				\begin{equation}\label{eq-SAW-1}
				Z_{\mathrm{Cardy}} (q) \eta(2i \tau)=(2\tau)^{-1/2}{q}^{\frac1{24}}\sum_{k\in \mathbb Z} k (-1)^{k-1} q^{\frac{3k^2}2-k+\frac18}=(2\tau)^{-1/2} \sum_{k\in \mathbb Z} k (-1)^{k-1} q^{\frac{3}2(k-\frac13)^2}.
				\end{equation}
				Comparing~\eqref{eq-SAW-0} and~\eqref{eq-SAW-1}, we see that $Z_{8/3}=CZ_{\mathrm{Cardy}}$ where $C=\frac{4C_2}{\pi\sqrt3}$. 	
			\end{proof} 
		}

		\appendix
		
		\section{Background on special functions}\label{app:theta}
		The Dedekind eta function with imaginary argument is given by
		\begin{equation}
		\eta(i\tau)=e^{-\frac{\pi\tau}{12}} \prod_{k=1}^\infty (1-e^{-2\pi k\tau}) \quad \textrm{with }\tau>0.
		\end{equation}
		We use the following convention for the Jacobi theta function   with imaginary argument:
		\begin{equation}\label{eq:theta}
		\theta_{1}\!\left(x , i\tau \right) := -i {e}^{-\pi \tau / 4} \sum_{n=-\infty}^{\infty} {\left(-1\right)}^{n} {e}^{-n \left(n + 1\right)\pi  \tau} {e}^{(2 n + 1)\pi i x} \quad \textrm{for }x\in \mathbb R \textrm{ and }\tau\in (0,\infty).
		\end{equation} 
		By Euler's  pentagonal  identity  \(\prod_{k=1}^\infty(1- x^k)=\sum_{n\in \mathbb Z} (-1)^n x^{n(3n-1)/2}\), we have  \(\theta_1(\frac{1}{3},i\tau) = {\sqrt 3\eta(3i\tau)}\).  
		
		By e.g.~\cite[Section 11, Appendix 2]{bm-handbook}, we have the following Laplace transforms:
		\begin{align}
		\label{eq:theta-L}
		\int_{0}^{\infty} {e}^{-a \tau} \theta_{1}\!\left(x , i b \tau\right) \, d\tau &= \sqrt{\frac{\pi}{a b}} \frac{\sinh\!\left(2 x \sqrt{\frac{\pi a}{b}}\right)}{\cosh\!\left(\sqrt{\frac{\pi a}{b}}\right)} \quad &&\textrm{for } a > 0, b > 0 \;\mathbin{\operatorname{and}}\; x \in  \left[-1/2, 1/2\right].\\ 
		\label{eq:eta-L}
		\int_{0}^{\infty} {e}^{-a \tau} \eta(i\tau) d\tau &= \sqrt{\frac{\pi}{a}} \frac{\sinh\!\left(2 \sqrt{\frac{\pi a}{3}}\right)}{\cosh\!\left(\sqrt{3\pi a}\right)} &&\textrm{for }a>0.
		\end{align}
		\section{Background on Brownian surfaces}\label{app:BA}
		In this appendix we provide more background on Brownian surfaces and give proofs for Lemmas~\ref{lem:Mod-det}--~\ref{lem:sphere-ball}. {Some of these lemmas may already exist in the literature, but we cannot locate them hence include a proof for completeness.}
		We first recall
		a scaling limit result for the Brownian sphere due to \cite{legall-uniqueness,miermont-brownian-map}.
		\begin{theorem}[\cite{legall-uniqueness,miermont-brownian-map}]\label{prop-BS-convergence}
			Let $\cS_n$ be uniformly sampled from the set of quadrangulations of the sphere with $n$ faces and two distinguished directed edges. View $\cS_n$ as a two-pointed metric-measure space as follows. The metric is the graph distance on the vertex set rescaled by $(\frac98)^{1/4}n^{-1}$. The measure is the probability measure on its vertex set where  the weight of a vertex is proportional to its degree. Then $\cS_n$ converges in law to the \emph{unit area Brownian sphere}, with respect to the Gromov-Hausdorff-Prokhorov  topology as compact metric-measure spaces with two marked points.
		\end{theorem}
		
		{We now recall a scaling limit result for $\BD_{0,1}(L)^{\#}$ proved by Gwynne and Miller~\cite{gwynne-miller-simple-quad}.
			The topology of convergence is the natural one for compact metric-measure spaces decorated with a continuous curve, which is called the Gromov-Hausdorff-Prohkorov-uniform (GHPU) topology; see~\cite[Section~1.2.3]{gwynne-miller-simple-quad} for its definition.} To state the result, for a positive integer $a$ we let $\cQ^{\disk}_{0,1}(2a)$ be the set of quadrangulations with simple boundary, having boundary length $2a$, and having a marked directed boundary edge. {For later use} we also let {$\cQ^{\disk}_{1,1}(2a)$} be the set of such quadrangulations where we further mark a directed interior edge (i.e.\ an edge whose endpoints are both interior vertices).
		\begin{theorem}[{Theorem 1.4 in} \cite{gwynne-miller-simple-quad}] \label{prop-brownian-disk}
			Let $L>0$ and let ${(a_n)}$ be a sequence of integers satisfying $\lim_{n\to\infty} \frac{{2a_n}}{3n^2} = L$. {Let $\BD^n_{0,1}(2a_n)^{\#}$ be the probability measure on  $\cQ^{\disk}_{0,1}(2a_n)$ where each map with $v$ vertices is assigned mass proportional to $\rho^v$ with $\rho = \frac1{12}$.
				Let {$(\cD_n,e_n)$} be a sample from  $\BD^n_{0,1}(2a_n)^{\#}$ where $e_n$ is the boundary marked edge of $\cD_n$.}
			View $\cD_n$ as a metric measure space where the metric is graph distance rescaled by $n^{-1}$, and the mass at each vertex $v$ with degree $\mathrm{deg}(v)$ is $\frac29 n^{-4}\mathrm{deg}(v)$. {We also view the boundary of $\cD_n$ as a continuous curve starting and ending at $e_n$.}
			Then with respect to the {GHPU} topology, as $n \to \infty$, {the law of $(\cD_n,e_n)$ converges weakly to $\BD_{0,1}(L)^{\#}$ defined in~\eqref{eq:BD-def-more}.} 
		\end{theorem}
		
		Let $\cQ_\ann(2a, 2b)$ be the set of annular quadrangulations with labelled boundaries, {where the number of edges on the two boundaries are  $2a$ and $ 2b$, respectively}. The weak limit of {the critical Boltzman measure on $\cQ_\ann (a,b)$ should be $\BA(a,b)$.} 
		Although this scaling limit result is not needed for our paper, we will use the following enumeration asymptotic, {which explains why we set $|\BA(a,b)|$ to be proportional to $\frac{1}{\sqrt{ab}(a+b)}$.}
		
		\begin{lemma}\label{lem-boltz-len}
			Let $\rho = \frac1{12}$, $\theta = \frac1{54}$.  Let $\cQ^\disk_{1,1}(2a;v)\subset  \cQ^\disk_{1,1}(2a)$ be the subset of maps with $v$ interior vertices. Then
			\[\sum_{v \geq 2} \# {\cQ^{\disk}_{1,1}}(2a;v) \rho^v\theta^a \sim \frac1{\sqrt{27\pi a}}, \quad \textrm{where  $\sim$ indicates the ratio tends to 1 as } a\to \infty. \]
			Likewise, let $\cQ^2_\ann(2a,2b)$ be the set of such quadrangulations in $\cQ_\ann(2a,2b)$ where each boundary has a marked edge. Let $\cQ_\ann(2a, 2b;v)\subset \cQ_\ann(2a, 2b)$ and $\cQ^2_\ann(2a, 2b;v)\subset \cQ^2_\ann(2a, 2b)$ be the subset of maps with $v$ interior vertices, respectively.
			Then with $\sim$ indicating the ratio tends to 1 as $a,b \to \infty$, we have
			\[\sum_{v \geq 0} \#\cQ_\ann(2a, 2b;v) \rho^v \theta^{a+b} \sim \frac{1}{\pi\sqrt{ab}(a+b)},\quad \textrm{and} \quad \sum_{v \geq 0} \#\cQ^2_\ann(2a, 2b;v) \rho^v \theta^{a+b} \sim \frac{\sqrt{ab}}{\pi(a+b)}.\]
		\end{lemma} 
		\begin{proof}
			The key is the following enumeration from \cite[Theorem 1.2]{bernardi-fusy}:
			\begin{equation}\label{eq-bernardi-fusy}
			\# \cQ^2_\ann(2a,2b;v) = \frac{3^v (3(a+b)+2v-1)!}{v!(3(a+b)+v)!} \cdot 4ab \binom{3a}a\binom{3b}b.
			\end{equation}
			We claim that 
			\[
			\sum_{v \geq 0} \#\cQ^2_\ann(2a, 2b;v) \rho^v = \frac{4ab}{3(a+b)} 8^{a+b} \binom{3a}a \binom{3b}b , \quad\textrm{and}\quad
			\sum_{v \geq 2} \# {\cQ^{\disk}_{1,1}}(2a;v) \rho^v = \frac{2a}{9(a+1)} 8^a \binom{3a}a.
			\]
			The first identity follows from \eqref{eq-bernardi-fusy} and \cite[(5.68), (5.70)]{concrete-mathematics}. The second is obtained by identifying the interior marked edge as a boundary of length 2, then using the first identity with $b = 1$. Stirling's approximation then gives both asymptotics. Finally the asymptotic for $\#\cQ_\ann$ follow from that of $\#\cQ^2_\ann$ {by dropping the boundary marked edges}.
		\end{proof}

		We now turn to the proof of Lemmas~\ref{lem:ball-length} and~\ref{lem:sphere-ball}.
		We first state the asymptotic of quadrangulation metric balls that we extract from \cite{curien-legall-peeling}.
		For positive integers $r,a$ we
		let $\cQ_\ball(r;2a)$ be the set of quadrangulations with the disk topology with boundary length $2a$, with a distinguished oriented edge called the \emph{root edge}, such that the four vertices of every quadrangle sharing an edge with the boundary are at distances $r-1,r,r+1,r$  from the root vertex, respectively.
		\begin{proposition}
			\label{prop-ball-count}{Let $\rho=\frac1{12}$ and $\theta=\frac1{54}$ as in Lemma~\ref{lem-boltz-len}. Let $\P^{\ball}(n)$ be the probability measure on $\bigcup_{a = 1}^\infty \cQ_\ball(n; 2a)$ where each element is assigned weight  proportional to $\rho^V (\rho/\theta)^{L}$, where $V$ is the number of internal vertices and $L$ is the number of boundary vertices.  
				Let $\cL_n$ be the number of boundary vertices of a sample from $\P^{\ball}(n)$.
				Then $(3n^2)^{-1} \cL_n$ converges in law as $n \to \infty$ to the random variable with density proportional to $1_{\ell>0} |\mathrm{Ball}_1(\ell)|d\ell$, where $|\mathrm{Ball}_1(\ell)| =e^{-\frac{9\ell}{2}}$ as defined in~\eqref{eq:ball-pfn}.}
		\end{proposition}
		\begin{proof}
			{Recall that the {uniform infinite planar quadrangulation} (UIPQ) is the local limit of uniform quadrangulations; see e.g.~\cite[Section 6.2]{curien-legall-peeling}. Let $\P^{\ball}_{\mathrm{UIPQ}}(n)$ be the law of the filled metric ball around the root at radius $n$ on a UIPQ. Here a filled metric {ball} is the union of the ball itself  and all its bounded complementary components.
				Let $\cQ^\disk_{0,1}(2a;v)$  be the set of quadrangulations with simple boundary, $2a$ boundary vertices and $n$ interior vertices, and a  marked boundary edge.  As explained in~\cite[Section 6.2]{curien-legall-peeling},
				\[
				\#\cQ^\disk_{0,1}(2a;v)=\frac{3^{v-1}(3a)!(3a-3+2v)!}{v!a!(2a-1)!(v+3a-1)!}
				\qquad \textrm{and}\qquad  \#\cQ^\disk_{0,1}(2a;v) \sim C^\square(a)\rho^{-v} v^{-5/2} \textrm{ as }v\to \infty
				\]
				for $ C^\square(a) = \frac{8^{a-1} (3a)!}{3\sqrt\pi a! (2a-1)!}$.
				Therefore, by the definition of UIPQ, the probability measure $\P^{\ball}_{\mathrm{UIPQ}}(n)$ is proportional to $\theta^{\cL_n/2}C^\square(\cL_n/2)\P^{\ball}(n)$.
				By the main result of~\cite{curien-legall-peeling}, the law of $(3n^2)^{-1}\cL_n$ under $\P^{\ball}_{\mathrm{UIPQ}}(n)$ weakly converges to the probability measure  with density proportional to  $1_{\ell>0}\ell^{1/2}e^{-9\ell/2}d\ell$, which is the law of the boundary length of the filled metric ball of radius 1 around the root on the Brownian plane. This statement for triangulations (namely UIPT) was proved as Theorem 2 in~\cite{curien-legall-peeling} and then extended to the UIPQ case  in Section 6.2 there. Since $C^\square(a) \sim \frac1{8\sqrt3\pi}\theta^{-a} \sqrt a$,  the law of $(3n^2)^{-1}\cL_n$ under $\P^{\ball}(n)$ weakly converges to the probability measure  with density proportional to  $1_{\ell>0} \ell^{-1/2}\cdot \ell^{1/2}e^{-9\ell/2}\,\rd \ell=1_{\ell>0} |\mathrm{Ball}_1(\ell)|\,\rd \ell$.}
		\end{proof}
		
		\begin{proof}[Proof of Lemma~\ref{lem:ball-length}]
			Let $\rho=\frac1{12}$ and $\theta=\frac1{54}$ as in Lemma~\ref{lem-boltz-len} and Proposition~\ref{prop-ball-count}. Let $a_n = \lfloor \frac32n^2 a \rfloor$ and  $\BD^n_{1,1}(2 a_n)^{\#}$ be the probability measure on $\cQ^{\disk}_{1,1}(2 a_n)$ where each element is assigned weight proportional to $\rho^V$, with $V$ being the number of internal vertices. Let $(\cD_n,p_n,e_n)$ be a sample from $\BD^n_{1,1}(2 a_n)^{\#}$, where $p_n$ is the interior marked edge and $e_n$ is the boundary one.   Let $E^n$ be the event that the  graph distance between $p_n$ and $\partial \cD_n$ is at least $n$. 
			On the event $E^n$, let $\cB^\bullet_1(p_n)$ be  the filled metric ball of $\cD_n$ around $p_n$ of radius $n$. Here a filled metric {ball} is the union of the ball itself  and all its complementary components except the one touching the boundary. Let $\cL_n$ be the number of edges on the boundary of $\cB^\bullet_1(p_n)$. {Let $M^{\ball}_{\disk}(n;2a_n)$} be the (non-probability) measure describing the law of $\cB^\bullet_1(p_n)$ under the restriction of $\BD^n_{1,1}(2 a_n)^{\#}$ to $E^n$. Let
			$Z_n= \sum \rho^V (\rho/\theta)^{L}$ where the summation is over  $\bigcup_{x = 1}^\infty \cQ_\ball(n; 2x)$ as in the definition of $\P^{\ball}(n)$ in Proposition~\ref{prop-ball-count}. 
			Then with $\cQ_{\ann}^{2}(2a_n,\cL_n;v)$ and  $ \cQ^{\disk}_{1,1}(2a_n;v)$ defined in Lemma~\ref{lem-boltz-len}, we have
			\begin{align}\label{eq:RN-ball}
			\mathbb M^{\ball}_{\disk}(n;2a_n)
			= \frac{\sum_{v}\# \cQ_{\ann}^2(2a_n,\cL_n;v) \rho^v\theta^{a_n + b}}{\sum_{v} \#  \cQ^{\disk}_{1,1}(2a_n;v) \rho^v\theta^{a_n}}  \cdot Z_n \P^{\ball}(n).
			\end{align}
			By definition the probability $\BD^n_{1,1}(2 a_n)^{\#}(E^n)$ is the total mass $|\mathbb M^{\ball}_{\disk}(n;2a_n)|$ of $\mathbb M^{\ball}_{\disk}(n;2a_n)$. By Theorem~\ref{prop-brownian-disk}, $\BD^n_{1,1}(2 a_n)^{\#}(E^n)$ converges to $\BD_{1,1}(a)^{\#}[d(p,\partial \cD)>1]$. Recall that $|\BA(a,b)|=\frac{1}{2\sqrt{ab}(a+b)}$ and $|\BD_{1,1}(a)|=a^{-1/2}$.
			By the asymptotics in Lemma~\ref{lem-boltz-len} and Proposition~\ref{prop-ball-count},  we have  
			\[
			|\mathbb M^{\ball}_{\disk}(n;2a_n)|=C_1 (1+o_n(1))Z_n \int_0^\infty  \frac{ab|\BA(a,b)| |\mathrm{Ball}_1(b)|}{|\BD_{1,1}(a)|} \rd b
			\]
			for some constant $C_1>0$. Therefore  $Z_n$ converges to a constant $C_2>0$ such that
			\[
			\BD_{1,1}(a)^{\#}[d(p,\partial \cD)>1]=C_1C_2\int  \frac{ab|\BA(a,b)| |\mathrm{Ball}_1(b)|}{|\BD_{1,1}(a)|} \rd b=C Z(a)
			\]
			for some constant $C>0$ not depending on $a$ and $Z(a)$ as defined in Lemma~\ref{lem:ball-length}. Therefore  
			\[
			\BD_{1,0}(a)^{\#}[d(p,\partial \cD)>1]=BD_{1,0}(a)^{\#}[d(p,\partial \cD)>1]=CZ(a).
			\]
			Sending $a\to \infty$ we get $C=Z(\infty)^{-1}$ as desired. 
			Now combining~\eqref{eq:RN-ball}, the asymptotics in Lemma~\ref{lem-boltz-len} and Proposition~\ref{prop-ball-count}, and that $Z_n\to C_2$, we see that the law of $(3n^2)^{-1}\cL_n$  under the probability measure proportional to $M^{\ball}_{\disk}(n;2a_n)$ weakly converges to the law of $\cL$ described in Lemma~\ref{lem:ball-length}. Finally, the conditional independence of the annulus $\cA$ and the filled metric ball $\cB^{\bullet}(p,1)$ is an instance of the domain Markov property of the peeling process; see e.g.~\cite[Corollary 9]{LeGall-Star}.
		\end{proof}
		\begin{proof}[Proof of Lemma~\ref{lem:sphere-ball}]
			{We claim that under $\BS_2$ restricted to $d(x,y)>1$, the law of $\cL_x$ has density 
				$$1_{\ell>0} |\BD_{1,0}(\ell)| \ell|\mathrm{Ball}_1(\ell)|\,\rd \ell.$$
				This follows from the same argument as in the proof of Lemma~\ref{lem:ball-length} with $a_n$ being $1$ for all $n$ instead. We omit the detailed proof. Now Lemma~\ref{lem:sphere-ball} follows from Proposition~\ref{prop-sphere-ball} and Lemma~\ref{lem:ball-length}.}
		\end{proof}
		We finally explain how the measurability result Lemma~\ref{lem:Mod-det} follows from known results.
		\begin{proof}[Proof of Lemma~\ref{lem:Mod-det}]
			Suppose $(\mathcal S,x,y)$ is a Brownian sphere with two marked points conditioned on the event that $d(x,y)>2$. Sample a point $q$ on $\partial \mathcal B^\bullet (x,1)$ according to the boundary length measure. We claim that the filled metric ball $(\cB^\bullet (x,1),q)$ and the Brownian disk $(\mathcal S\setminus \cB^\bullet (x,1),y, q )$ as marked metric-measure spaces together determine $(\mathcal S, x,y,q)$. To see this, from the Brownian snake construction of the Brownian map, $(\mathcal S,x)$ is determined by the geodesic tree  $\mathcal T_x$ rooted at $x$ and its dual tree $\tilde{\mathcal T}_x$ mated together as a pair of rooted planar trees. Here $\tilde{\mathcal T}_x$ is a Brownian continuum random tree  (CRT) and the contour function of $\mathcal T_x$ is a Brownian snake indexed by the CRT.
			For each point $z\in \mathcal S$, if $z\in B^\bullet (x,1)$, then the geodesic from $z$ to $x$ must be inside  $\cB^\bullet (x,1)$. Otherwise, the geodesic from $z$ to $x$ is a concatenation of two internal geodesics, one in  $\mathcal S\setminus \cB^\bullet (x,1)$ and the other in $\cB^\bullet (x,1)$. Therefore, knowing $(\cB^\bullet (x,1),q)$ and $(\mathcal S\setminus \cB^\bullet (x,1),q)$, we can determine how $\mathcal T_x$ and $\tilde {\mathcal T}_x$ are mated together to form  $(\mathcal S,x)$. The point $q$ is added to fix how boundary points on  $\cB^\bullet (x,1)$ and $\mathcal S\setminus \cB^\bullet (x,1)$ are identified.
			
			The Brownian snake construction of the Brownian sphere was extended to the Brownian disk case; see~\cite[{Theorem 1}]{legall-disk-snake}. 
			On the event $d(p,\partial \mathcal D)>1$, we let $\mathcal A$ be the annulus bounded by $\partial \cB^\bullet(p,1)$ and $\partial \mathcal D$. Sample a point $q$ on $\partial \cB^\bullet(p,1)$ from its length measure. Then the same argument as above shows that  $(\mathcal A,q)$ and $(\partial \cB^\bullet(p,1),q)$ together determine $(\mathcal D,p,q)$. This further determines the conformal structure of  $(\mathcal D,p,q)$ and hence that of $(\mathcal A,q)$. On the other hand, let $\cL$ be the boundary length of $( \cB^\bullet (p,1),q)$. By the last assertion of Lemma~\ref{lem:ball-length}, once we rescale the metric and measure of $(\cB^\bullet (p,1),q)$ by $\cL^{-1/2}$ and $\cL^{-2}$ respectively, then the rescaled metric measure space is independent of
			the conformal structure of $(\mathcal A,q)$. Therefore the conformal structure of $(\mathcal A,q)$ is determined by the metric-measure structure of $(\mathcal A,q)$ alone. Forgetting the extra marked point $q$ we are done. 
		\end{proof}

		\bibliographystyle{alpha}
		\def\cprime{$'$}

	\end{document}